\documentclass[11pt]{article}
\usepackage{a4wide}
\usepackage{amsmath, amsthm, xcolor}
\usepackage{amssymb}
\usepackage{graphicx,soul}
\usepackage[normalem]{ulem}
\usepackage{relsize}
\usepackage{enumerate}

\newcommand{\hidden}[1]{}

\usepackage{color}
\usepackage[mathscr]{euscript}
\usepackage{tikz}

\numberwithin{equation}{section}

\newcommand{\R}{{\Bbb R}}

\newcommand{\Z}{{\Bbb Z}}
\newcommand{\N}{{\Bbb N}}

\newcommand{\I}{{\Bbb I}}

\newcommand{\T}{{\Bbb T}}

\newcommand{\K}{{\mathcal K}}

\newtheorem{theorem}{Theorem}[section]
\newtheorem{proposition}{Proposition}[section]

\newtheorem{corollary}{Corollary}[section]
\newtheorem{lemma}{Lemma}[section]

\newtheorem{thmBV}{Theorem BV  \!\!\!\!}

\newtheorem{thmR}{Theorem R \!\!\!}

\newtheorem{thmWW}{Theorem WW \!\!\!}

\theoremstyle{remark}
\newtheorem{remark}{{Remark}}

\def\ep{ \epsilon }

\newcommand{\vz}{\mathbf{z}}
\newcommand{\vw}{\mathbf{w}}
\newcommand{\vx}{\mathbf{x}}

\newcommand{\vp}{\mathbf{p}}

\newcommand{\cH}{\mathcal{H}}
\newcommand{\UP}{\mathcal{U}(\Psi)}

\begin{document}
\title{Diophantine approximation and the Mass Transference Principle: incorporating the unbounded setup}


 \author{ Bing Li \\ {\small\sc (SCUT) } \and Lingmin Liao \\ {\small\sc (Wuhan) }  \and
Sanju Velani \\ {\small\sc (York) }   \and  Baowei Wang \\ {\small\sc (Hust) }  \and  \vspace*{4ex}
 \and ~ Evgeniy Zorin \\ {\small\sc (York)}}

\date{}

 \maketitle

\begin{abstract}
We develop the Mass Transference Principle for rectangles of Wang \& Wu (Math. Ann. 2021) to incorporate the `unbounded' setup; that is, when along some direction the lower order (at infinity) of the side lengths of the rectangles under consideration is infinity. As applications, we obtain the Hausdorff dimension of naturally occurring  $\limsup$ sets within the classical framework of simultaneous Diophantine approximation and the dynamical framework of shrinking target problems. For instance, concerning the former,  for  $\tau >0$,  let $S(\tau)$ denote the set of $(x_1,x_2)\in \R^2 $ simultaneously  satisfying  the inequalities
$
\|q x_1 \| \, < \, q^{-\tau}  $ and $ \|q x_2 \| \, < \, e^{-q}
$
 for infinitely many $q \in \N$.  Then, the `unbounded'  Mass Transference Principle enables us to show that
 $\dim_{\rm H} S(\tau)   \, = \,   \min \big\{   1, 3/(1+\tau)  \big\}  \, $.

\end{abstract}

\bigskip

\section{Introduction}  \label{int}

To set the scene, we start by recalling the so called `weighted' version of the classical Jarn\'{\i}k--Besicovitch Theorem in the theory of simultaneous Diophantine approximation. For $ 1 \le i \le d$, let $\psi_i: \R^+ \to \R^+$ be a real positive function. For convenience, let $$\Psi:=(\psi_1,\dots,\psi_d)$$ and in turn, let $W_d(\Psi)$ denote the set of points $\vx=(x_1, \ldots, x_d)\in\I^d :=[0,1]^d$ for which the inequalities
\begin{equation}  \label{aug}
 \| q x_i \| \le  \psi_i(q)  \quad  (1\le   i \le   d) \ \
\end{equation}
hold for infinitely many $q\in\mathbb{N}$. Throughout,  $\| \ . \ \|$ denotes the distance to the nearest integer.
  Given  $\boldsymbol{\tau} = (\tau_1, \ldots, \tau_d) \in \R^d$  with $  \tau_i > 0 $  for all $1\le i\le d$,  in the case $ \psi_i: q \to  q^{-\tau_i} $ we write  $W_d(\boldsymbol{\tau})$  for  $W_d(\Psi)$.
  On making use of the  Borel--Cantelli lemma from probability theory, it is  relatively straightforward to show  that
$$
m_d( W_d(\Psi)) = 0  \quad { \rm if } \quad  \textstyle{\sum_{q = 1}^{\infty} } \psi_1(q) \times \dots \times \psi_d(q)   \, < \, \infty \, ,
$$
where $m_d$ is the $d$-dimensional Lebesgue measure.  In this situation, it is natural to consider the size of the set under consideration  in terms of Hausdorff dimension $\dim_{\rm H}$.
The following statement is due to Rynne \cite{Rynne}.
It builds on the fundamental works of Jarn\'{\i}k, Besicovitch, Eggleston and Dodson -- see \cite{Rynne} for the references. Throughout, for any  $\mathbf{t} = (t_1, \ldots, t_d)  \in (\R^+)^d $ we let
\begin{equation}
\label{weneedto}
	\zeta_i(\mathbf{t}):=\frac{d+1  + \sum_{k: t_k<t_i}(t_i-t_k)}{1+t_i}  \qquad   \quad  (1\le   i \le   d)    \, .
\end{equation}


\begin{thmR} Let $\boldsymbol{\tau} = (\tau_1, \ldots, \tau_d) \in \R^d$ and assume that  $0<\tau_i<\infty$ for all $1\le i\le d$
  and  that $\tau_1+\cdots+\tau_d>1$. 
Then
	$$\dim_{\rm H} W_d(\boldsymbol{\tau})=\min_{1\le   i\le   d}\zeta_i(\boldsymbol{\tau})  \, . $$
\end{thmR}

Regarding the general function  $\Psi=(\psi_1,\cdots,\psi_d) : \R \to (\R^+)^d$, as pointed out explicitly by Rynne \cite[Corollary~1]{Rynne}, his theorem easily allows us to deduce that if the limits
\begin{equation}  \label{fg}
\lambda_i:=\lim_{n\to\infty}\frac{-\log\psi_i(n)}{\log n}   \qquad (1\le i\le d) \end{equation}
exist and  that  $0<\lambda_i<\infty$ for all $1\le i\le d$
  and   $\lambda_1+\cdots+\lambda_d>1$, then
  \begin{equation}  \label{ryncor} \dim_{\rm H} W_d(\Psi)=\min_{1\le   i\le   d}\zeta_i(\boldsymbol{\lambda})  \qquad {\rm where }  \qquad \boldsymbol{\lambda} = (\lambda_1, \ldots, \lambda_d)  \, .
  \end{equation}
	In other words, if the order at infinity of the functions $\psi_i$  ($1\le i\le d$) exist {(and indeed are finite)} then we are in good shape.
However, this is a pretty restrictive condition on $\Psi=(\psi_1,\cdots,\psi_d)$ and the problem of determining the dimension of $W_d(\Psi)$ for general monotonic $\Psi$ remained open until recently.  In view of earlier work of Dodson (see Remark~\ref{remq} below) the general consensus was that the dimensional formula given by \eqref{ryncor} is true with the order replaced by the lower order at infinity of  $\psi_i$   (i.e. replacing the limit in \eqref{fg} by the limit infimum).  In fact, Rynne  stated   in his paper that this can be proved by combining  his arguments with those stemming from the work of Dodson.  However, this turns out not to be the case.

The breakthrough for determining the dimension of $W_d(\Psi)$ for general monotonic $\Psi$ was made in the pioneering work of Wang $\&$  Wu \cite{WW2021} in which they developed the so-called  `rectangles to rectangles' Mass Transference Principle  -- see \S\ref{RR-baby}. In short, as an application of their extremely versatile mass transference principle, they showed that $\dim_{\rm H} W_d(\Psi)$  is dependent on the  set   $\UP$  of accumulation points $ \mathbf{t}=(t_1,t_2, \ldots, t_d)$ of the sequence
\begin{align}\label{sequence-psi}
\Big\{\Big(\frac{-\log\psi_1(n)}{\log n},\cdots,
\frac{-\log\psi_d(n)}{\log n}\Big)\Big\}_{n\ge   1}  \, .
\end{align}

\begin{thmWW}
 For  $1\leq i\leq d$, let  $\psi_i: \mathbb{R}^+\to\mathbb{R}^+$ be a real, positive, non-increasing function and suppose that $\UP$ is bounded. Then
  $$\dim_{\rm H} W_d(\Psi)=\sup_{\mathbf{t}\in\UP}  { \min } \left\{ \min_{1\le   i\le   d}\zeta_i(\mathbf{t})   { , d }  \right\} .$$
\end{thmWW}

\medskip

\begin{remark} \label{remq}
It is easily seen that if $ \psi_i: q \to  q^{-\tau_i} $  ($  \tau_i > 0 $)  for all $1\le i\le d$,  then $\UP = \boldsymbol{\tau} = (\tau_1, \ldots, \tau_d)  { \in (\R^+)^d}$  and so
 Theorem~WW implies Theorem~R.  Moreover, in the case  $\psi_1= \ldots =\psi_d := \psi $, let us  write $W_d(\psi)$ for $ W_d(\Psi) $ and observe that  Theorem~WW implies that for $\lambda > 1/d $
$$
\dim_{\rm H} W_d(\psi)=\frac{d+1}{1+\lambda}  \, ,
$$
where $\lambda$ is the lower order of $\psi$.
This is Dodson's theorem.  Specialising further, given a real number $\tau > 0$, we write $W_d(\tau) $ for  $W_d(\psi:  x \mapsto x^{-\tau})$ and we recovers the  simultaneous version of the classical Jarn\'{\i}k--Besicovitch theorem; that is to say that $\tau $  replaces  $\lambda$ in the statement of Dodson's theorem.
\end{remark}

\medskip

\begin{remark} \label{remunded}
Theorem R explicitly requires  that the exponents $\tau_i  \  ( 1 \le i \le d) $ are bounded or in its slightly more general form that the  order $\lambda_i $ at infinity of the functions $\psi_i$  ($1\le i\le d$) are bounded.   The situation in which one or more of the exponents is  unbounded cannot be inferred from the bounded statement -- see below for an explicit  example that illustrates this point.
Turning our attention to  Theorem~WW, as already mentioned the proof makes use of the `rectangles to rectangles'  Mass Transference Principle of  Wang $\&$  Wu. It turns out that in order to apply the latter it is necessary to impose the condition that  the  set  $\UP$  of accumulation points is bounded.  The main goal of this work is to  extend their mass transference principle  to  incorporate the `unbounded' setup (see Theorem~\ref{pointcase} in \S\ref{MTP}) and thereby  remove the need for  $\UP$ to be bounded within the framework of Theorem~WW -- see Theorem~\ref{Runbounded} below.  In turn, this means that we can establish a version of  Theorem~R free from the condition that the exponents $\tau_i$  are bounded -- see Corollary~\ref{corR} below.
\end{remark}

\medskip


\noindent We  now give a  concrete example that  illustrates the  point made in Remark~\ref{remunded}.   Given a real number $\tau >0$,  let $S(\tau)$ denote the set of $(x_1,x_2)\in \I^2 $ for which the inequalities
$$
\|q x_1 \| \, < \, q^{-\tau}   \qquad   {\rm and}  \qquad \|q x_2 \| \, < \, e^{-q}
$$
hold for infinitely many $q \in \N$.  Obviously,  the lower order of the function $ \psi : x \mapsto e^{-x} $ is infinity.    It is easily seen that for any $\nu > 0$,
$$
W(\tau) \times \{0\}    \ \subseteq \  S(\tau) \ \subseteq \  W(\tau,\nu)  \,
$$
where $ W(\tau,\nu) = W_2(\Psi) $ with $\psi_1: x \mapsto x^{-\tau} $ and $\psi_2 :  x \mapsto x^{-\nu} $.
It then follows on applying  Theorem R  and on letting $\nu \to \infty $,  that
$
\dim S(\tau) = 1 $  for $1/2 \le \tau  \le 1 $  and  that for $\tau > 1 $
$$
\frac{2}{1+\tau}   \, \le \, \dim_{\rm H} S(\tau)   \, \le \,   \min \Big\{   1, \frac{3}{1+\tau}  \Big\}  \, .
$$
The upshot of this is that Theorem R does not provide a precise formula for $ \dim_{\rm H} S(\tau) $ when $\tau > 1$.  To the best of our knowledge it cannot  be deduced   from existing results.   In short, the problem of determining a precise formula highlights a `hole' in classical theory of Diophantine approximation.  Just as importantly, it also brings to the forefront the limit of current results and frameworks.    As alluded to in Remark~\ref{remunded} above, we extend the  Mass Transference Principle of  Wang $\&$  Wu to incorporate the unbounded setup (such as in the above concrete example when the lower order of an approximating function under consideration  is unbounded). In turn, this allows us to establish  the following generalisation of Theorem~WW from which it is easy to deduce that, for $\tau > 0 $,
\begin{equation} \label{easy}
 \dim_{\rm H} S(\tau)   \, = \,   \min \Big\{   1, \frac{3}{1+\tau}  \Big\}  \, .
\end{equation}

\medskip

\begin{theorem}\label{Runbounded}
For  $1\leq i\leq d$, let  $\psi_i: \mathbb{R}^+\to\mathbb{R}^+$ be a real, positive, non-increasing function.
 Then
  $$\dim_{\rm H} W_d(\Psi)=
  \sup_{\mathbf{t}\in\UP}  \min \big\{ \min_{  i \in \mathcal{L}(\mathbf{t}) }\zeta_i(\mathbf{t}),\ \#\mathcal{L}(\mathbf{t})\big\},$$
	where given $\mathbf{t}=(t_1,\dots,t_d) \in (\R^+\cup \{+\infty\})^d$  we set  $\mathcal{L}(\mathbf{t}):=\{1\leq i\leq d: t_i<+\infty\}$.
\end{theorem}

\bigskip

\begin{remark} \label{def2024} Note that with reference to \eqref{weneedto}, the definition of $ \zeta_i(\mathbf{t}) $    is equally valid for $\mathbf{t} \in (\R^+\cup \{+\infty\})^d$ if we restrict $1\leq i\leq d $ to those with $ t_i<+\infty$; that is to say,  $ i \in \mathcal{L}(\mathbf{t})$.  In the case $\#\mathcal{L}(\mathbf{t}) \neq  d $, we set
$$
\zeta_i(\mathbf{t}) := \#\mathcal{L}(\mathbf{t})     \qquad  {\rm for }  \quad  i \notin \mathcal{L}(\mathbf{t})    \, .
$$ In view of this,  when $\#\mathcal{L}(\mathbf{t}) \neq  d $, the conclusion of the theorem can be simplified to
$$
\dim_{\rm H} W_d(\Psi)=
  \sup_{\mathbf{t}\in\UP}  \; \min_{  1 \le i \le d } \zeta_i(\mathbf{t})  \, .
$$
In particular, with the above extension of  the definition of $ \zeta_i(\mathbf{t}) $ in mind,  note that the dimension formula as stated in Theorem~WW is valid in both the bounded  case  ((i.e., when $\#\mathcal{L}(\mathbf{t}) =   d $) and the unbounded case (i.e., when $\#\mathcal{L}(\mathbf{t}) \neq  d $).  The point being that in that the latter
$$ { \min } \left\{ \min_{1\le   i\le   d}\zeta_i(\mathbf{t})   { , d }  \right\}  =   \min_{  1 \le i \le d } \zeta_i(\mathbf{t}) .
$$
\end{remark}

\noindent To see that the theorem implies  the dimension result \eqref{easy} for $S(\tau)$,  let  $\psi_1 :  x \mapsto x^{-\tau} $ and $\psi_2: x\to e^{-x}$ and simply note that $\UP =\{(\tau, +\infty)\}$ and $\zeta_1(\mathbf{t}) =3/(1 + \tau)$.


As a corollary of Theorem \ref{Runbounded}, we have the following unbounded version of Theorem~R  for general functions.
\begin{corollary}  \label{corR}
For  $1\leq i\leq d$, let  $\psi_i: \mathbb{R}^+\to\mathbb{R}^+$ be a real, positive, non-increasing function such that  the limit $\lambda_i $   given by \eqref{fg}
exists.
Then
\[
\dim_{\rm H} W_d(\Psi)=\min\Big\{\min_{i \in\mathcal{L}(\boldsymbol{\lambda})}\zeta_i(\boldsymbol{\lambda}), \ \#\mathcal{L}(\boldsymbol{\lambda}) \Big\}  \, ,
\]
where $\boldsymbol{\lambda}=(\lambda_1, \dots, \lambda_d) $ and $\mathcal{L}(\boldsymbol{\lambda})=\{1\leq i\leq d: 0< \lambda_i<+\infty\}$.

\end{corollary}

\noindent Note that when $\boldsymbol{\lambda } $ is bounded  (i.e. when
$\boldsymbol{\lambda }  \in (\R^+)^d$ ), then $\#\mathcal{L}(\boldsymbol{\lambda})= d $   and we easily deduce Theorem~R.

Another application of our modified `rectangles to rectangles' Mass Transference Principle (Theorem \ref{pointcase} in \S\ref{RR-baby})  is to the shrinking target problem for matrix transformation of tori.  It allows us to solve the general dimension  problem discussed in \cite[Section~5.3]{LLVZstand}. In its simplest form the general problem can be viewed as the $\times2 \times3$ analogue of the above classical problem concerning the dimension of the set $S(\tau)$.  Indeed, for $\tau >0$  let $S^*(\tau)$ denote the set of $(x_1,x_2)\in \I^2 $ for which the inequalities
$$
\|2^n x_1 \| \, < \, e^{-n\tau}   \qquad   {\rm and}  \qquad \|3^n x_2 \| \, < \, e^{-n^2}
$$
hold for infinitely many $n \in \N$.  Then,
our  general dimension result for matrix transformation of tori (\S\ref{RR-baby}, Theorem~\ref{rectangledimresult}) implies the following statement.

\begin{theorem} \label{matrixEG} For $ \tau > 0$,
\begin{equation*} \label{yes} \dim_{\rm H}  S^*(\tau)  \,  = \,   \min \Big\{  1 , \ \frac{\log 2+\log 3}{\log 2+\tau} \Big\}    \, .
\end{equation*}
\end{theorem}

\bigskip

\begin{remark}  \label{Yubin}
In a recent  work, He \cite{He} has established a  Mass Transference Principle from the view point of Hausdorff content.  In particular, under the assumption of uniform local ubiquity his  Theorem~2.9  implies `full'  Hausdorff measure statements  and  can be utilized (see \cite[Remark 8]{He}) to show that   $   {\cal H}^{s}( S^*(\tau) ) = \infty $ at the critical exponent $s= \dim_{\rm H}  S^*(\tau) \, $.   However,  as discussed in Remark~\ref{NoToHMbut} below, his work does not  imply our  Mass Transference Principle (\S\ref{RR-baby}, Theorem \ref{pointcase}). Indeed,  our work is based on  the `original'  weaker full measure assumption and  as shown in
Remark~\ref{NoToHM} below, it is simply  not possible to obtain  a Hausdorff measures statement under the framework of Theorem~\ref{pointcase}.


\end{remark}

\section{The Mass Transference Principle(s) }  \label{RR-baby}

We start by defining Hausdorff measure and dimension
for completeness and for establishing some notation. Let $E$ be a subset of $\R^d$.   For $\rho
> 0$, a countable collection $ \left\{B_{i} \right\} $ of
Euclidean balls of diameter $|B_i| \le   \rho $ for each
$i$ such that $E \subset \bigcup_{i} B_{i} $ is called a $ \rho
$-cover for $E$.  Let $s$ be a non-negative number and define $$
 {\cal H}^{s}_{\rho}(E)
  \; = \; \inf \left\{ \sum_{i} |B_i|^s
\ :   \{ B_{i} \}  {\rm \  is\ a\  } \text{$\rho$-cover of}\  X
\right\} \; , $$ where the infimum is taken over all possible $
\rho $-covers of $E$. The {\it s-dimensional Hausdorff measure}
${\cal H}^{s} (E)$ of $E$ is defined by $$ {\cal H}^{s} (E) =
\lim_{ \rho \rightarrow 0} {\cal H}^{s}_{ \rho } (E) = \sup_{ \rho
> 0} {\cal H}^{s}_{ \rho } (E)
$$ \noindent and the {\it Hausdorff dimension}  $\dim_{\rm H}E$ of $E$ by
$$ \dim_{\rm H} \, E = \inf \left\{ s : {\cal H}^{s} (E) =0 \right\} =
\sup \left\{ s : {\cal H}^{s} (E) = \infty \right\} \, . $$

\vskip 9pt


\noindent Further details and alternative definitions of Hausdorff measure  and dimension can be found in \cite{F,MAT}.

We now  describe a powerful mechanism, the so called Mass Transference Principle, for obtaining  lower bounds for the Hausdorff dimension of a large class of  $\limsup$ sets.
 To set the scene, let $X$ be a locally compact subset of $\R^d$ equipped with a non-atomic probability measure $\mu$. Suppose there
exist constants $ \delta > 0$, $0<a\le 1\le b<\infty$ and $r_0 > 0$ such
that
\begin{equation} \label{MTPmeasure}
 a \, r ^{\delta}  \ \le    \  \mu(B)  \ \le    \   b \, r
^{\delta}
\end{equation}
for any ball $B=B(x,r)$ with $x\in X$ and radius $r\le r_0$. Such a measure is said to be \emph{$\delta$-Ahlfors regular}. It is well known that if $X$ supports a $\delta$-Ahlfors regular measure $\mu$, then we have
 \begin{equation} \label{ARdim}
\dim_{\rm H} X = \delta
\end{equation} and moreover
that $ \mu$ is comparable  to the $\delta$-dimensional Hausdorff measure ${\cal H}^\delta$ -- see \cite{F,MAT} for
the details. The latter implies that \eqref{MTPmeasure} is valid with $\mu$ replaced by  $\cH^{\delta} $.  Next, given $s > 0$  and a ball $B=B(x,r)$ we define
the scaled ball
$$
B^s:=B\big(x,r^{\frac{s}{\delta}}\big)\,.
$$
So, by definition $B^{\delta}=B$.  The original  Mass Transference Principle
(MTP) established in  \cite{BV2006}, allows us to transfer
$\cH^\delta$-measure theoretic statements for $\limsup$ subsets of
$X$ arising from balls to general $\cH^s$-measure theoretic statements.

\begin{thmBV}[MTP: balls to balls]\label{MTPBB}
		Let $X$ be a locally compact subset of $\R^d$ equipped with a  $\delta$-Ahlfors regular measure $\mu$.    Let $\{B_n\}_{n \in \N}$ be a sequence of balls in $X$ with radius $r(B_n)\to 0$ as $n\to\infty$.
		Let $s \ge  0 $ and suppose that
		$$\mathcal{H}^\delta \big( \limsup_{n\to\infty}B_n^s\big)=\mathcal{H}^\delta(X).$$
		Then, for any ball $B$	 $$\mathcal{H}^s\big(B\cap \limsup_{n\to\infty}B_n\big)=\mathcal{H}^s(B\cap X) \, $$ In particular, if $ s \le \delta$ then $$\dim_{\rm H} \big(\limsup_{n\to\infty}B_n   \big)  \ge   s \, .$$
	\end{thmBV}

\medskip

Note that the `in particular'  part follows directly from the definition of dimension and the  fact that $\cH^s(X) > 0 $  when $ s \le  \delta = \dim X$.   Although Theorem~BV has numerous applications, it has its limitations due to the fact that  the $\limsup$ sets are defined via balls.   Indeed, it is not applicable to the classical `rectangular' $\limsup$   subset   $W_d(\boldsymbol{\tau})$ of $\I^d$.  In short, $ \vw \in W_d(\boldsymbol{\tau})$ if and only if  there exist infinitely many  $(\vp,q) \in \Z^d \times \N$ such that
\begin{eqnarray*}
\vw  \in   R\big((\vp,q), \boldsymbol{\tau} \big) &:=&  \Big\{ \vx  \in \I^d :  |x_i-p_i/q| \le q^{-(\tau_i+1)} \  \ (1\le   i \le   d) \Big\} \\[1ex] &  = & \prod_{i=1}^d B(p_i/q, \psi(q)/q) \cap \I^d \, ;
 \end{eqnarray*}
that is,
$$
W_d(\boldsymbol{\tau}) =  \limsup_{q \to \infty} \bigcup_{\vp \in \Z^d} R\big((\vp,q),   \boldsymbol{\tau} \big)  \ .
$$
Clearly, the sets  $R\big((\vp,q), \boldsymbol{\tau} \big)$ are rectangular in shape and so Theorem~BV is only applicable when $\tau_1 = \ldots = \tau_d$ ; i.e. when the rectangles are balls.  The upshot of this basic and classical  setup is that an analogue of the   Mass Transference Principle  for $\limsup$ sets arising naturally via rectangles  is highly desirable.  As mentioned in the introduction, the breakthrough was made in the pioneering work of Wang $\&$  Wu \cite{WW2021} who obtained a dimension version of Theorem~BV  for $\limsup$ sets defined via rectangles within the `bounded' setup.   We now describe our main result which is a generalisation of their statement  to the `unbounded' setup.


\

\subsection{Mass Transference Principle for rectangles \label{MTP}  }


\bigskip



%


The lower bound for the dimension given by the Mass Transference Principle   is in terms of the so-called dimensional number.  This we first describe.

Let $p$ be a positive integer and let $\boldsymbol{\delta}=(\delta_1,\dots, \delta_p)\in (\mathbb{R}^+)^p $ be given.   Now for any
 $\mathbf{u}=(u_1,\dots, u_p)\in (\mathbb{R}^+)^p $ and
$\mathbf{v}=(v_1, \ldots,v_p)\in (\mathbb{R}^+\cup\{+\infty\})^p$ such that
\begin{equation} \label{111}
u_i\leq  v_i  \quad  \ {\text{for each}}\ \ 1\le i\le p  \, ,
\end{equation}
 let
\begin{equation} \label{112} \mathcal{L}(\mathbf{v}):=\{1\le   i \le   p: v_i<+\infty\}  \ \quad \text{and}  \
 \quad \mathcal{L}_\infty(\mathbf{v}):=\{1\le   i \le   p: v_i=+\infty\}  \end{equation}
and for each $ i \in \mathcal{L}(\mathbf{v}) $,  define the quantity
\begin{equation*}
	{s}(\mathbf{u}, \mathbf{v},i)=
	\sum_{k\in \mathcal{K}_1(i)}\delta_k+\sum_{k\in \mathcal{K}_2(i)}\delta_k\left(1-\frac{v_k-u_k}{v_i}\right) +\sum_{k\in \mathcal{K}_3(i)}\frac{\delta_ku_k}{v_i}
\end{equation*}
where
\begin{equation*}
	\mathcal{K}_1(i):= \{1\le   k\le   p: u_k> v_i\} \, ,  \ \ \ \mathcal{K}_2(i):= \{k\in\mathcal{L}(\mathbf{v}): v_k \le v_i\}, \
\end{equation*}
 and
\begin{equation*}
	\mathcal{K}_3(i):=\{1, \dots, p\}\setminus (\mathcal{K}_1(i)\cup \mathcal{K}_2(i)).
\end{equation*}
For each $ i \in \mathcal{L}_\infty(\mathbf{v}) $, we set
$$
{s}(\mathbf{u}, \mathbf{v},i)=\sum_{k\in\mathcal{L}(\mathbf{v})}\delta_k.
$$
Then, the  quantity
\begin{equation}\label{dimensionalnumber}
s_0(\mathbf{u},\mathbf{v}) := \min_{1 \le i \le p }  \{ {s}(\mathbf{u}, \mathbf{v},i) \}   =  \min \left\{\min_{i\in \mathcal{L}(\mathbf{v})}\{ {s}(\mathbf{u}, \mathbf{v},i) \}, \  \sum_{i\in\mathcal{L}(\mathbf{v})}\delta_i\right\}
\end{equation}
is called the \textit{dimensional number} with respect to $\mathbf{u}$ and $\mathbf{v}$.  Clearly, the right hand side term in the above follows directly from the definition of the quantity ${s}(\mathbf{u}, \mathbf{v},i)$ and also note that in view of \eqref{111}, for any $ 1 \le k \le p $ we  have that
\begin{equation}\label{dimensionalnumberUB}
s_0(\mathbf{u},\mathbf{v})  \leq {s}(\mathbf{u}, \mathbf{v},k) \leq  \sum_{ 1 \le i \le p }\delta_i  \, .
\end{equation}
Before proceeding with the statement of the Mass Transference Principle,  a few remarks are in order.

\bigskip

\begin{remark}
 In the `bounded'  setup (i.e. when  $\mathcal{L}_\infty(\mathbf{v})= \emptyset $), it is easily seen that the above definition of dimensional number coincides   with that of   Wang $\&$  Wu \cite[\S3.2]{WW2021}.   Indeed, on putting  $u_i=a_i$ and $v_i=a_i+t_i$ in the above recovers their formulation.
\end{remark}

\medskip

\begin{remark}\label{rem:6}
 It is evident from  the definition  of the quantities  ${s}(\mathbf{u}, \mathbf{v},i)$ and $s_0(\mathbf{u},\mathbf{v})$ that they remain unchanged when  $\mathbf{u} $ and  $\mathbf{v}$ are multiplied by a constant $c>0$;  that is
 \begin{equation}  \label{cucu}
 {s}(\mathbf{u}, \mathbf{v},i)={s}(c\mathbf{u}, c\mathbf{v},i) \quad   {\rm and }  \quad  s_0(\mathbf{u},\mathbf{v})=s_0(c\mathbf{u},c\mathbf{v})   \, .
 \end{equation}
\end{remark}

\noindent Further properties of the dimensional number, such as continuity, are discussed in  \S\ref{Subsec:dim-number}.   At this point it suffices  to simply note that Proposition~\ref{ll1} in \S\ref{Subsec:dim-number} implies that the dimensional number is continuous in $\mathbf{v}$; that is
\begin{equation}  \label{fufu}
\lim_{\epsilon \to 0} s_0(\mathbf{u}, (1\pm\epsilon)\mathbf{v})  = s_0(\mathbf{u}, \mathbf{v})   \,
\end{equation}
 for any $\epsilon>0$.

 Now, let us turn to the Mass Transference Principle for rectangles.

Fix an integer $p \ge 1$ and for $1 \le i \le p$, let
 $X_i$ be a subset of $\R^{d_i}$.  Obviously, if  $B_{i}$ is a ball in $X_i$ then  $\prod_{i=1}^pB_{i}$ is in general a rectangle in the product space $ \prod_{i=1}^p   X_i$.  The Mass Transference Principle for rectangles  provides a lower bound for the dimension of  $\limsup $ sets arising from sequences of such rectangles in the setting of  locally compact sets  equipped with Ahlfors regular measures.

%
%
%

\begin{theorem}\label{pointcase}
For each $1 \le i \le p$, let
 $X_i$ be a locally compact subset of $\R^{d_i}$ equipped with a  $\delta_i$-Ahlfors regular measure $\mu_i$.  Let $\{B_{i,n}\}_{n \in \N}$ be a sequence of balls $B_{i,n}$ in $X_i$ with radius $r(B_{i,n}) \to 0$ as $n\to\infty$ for each $1\le   i\le   p$.  Furthermore, for each $1\le   i\le   p$ assume that there exist sequences of positive real numbers $\{r_n\}_{n\in\mathbb{N}}$ and $\{v_{i,n}\}_{n\in\mathbb{N}}$ such that

\begin{equation}\label{ezsv}  r(B_{i,n})=r_n^{v_{i,n}}   \, 
 \end{equation} 
with
\begin{equation} \label{r_n_tends_to_zero}
\lim_{n\to\infty} r_n=0  
\end{equation}
 and
\begin{align}\label{cond:r-n}
   \lim\limits_{n\to \infty}v_{i,n}=v_i     
\end{align}
for some ${\mathbf{v}}=(v_1,\dots, v_p)\in(\mathbb{R}^+ \cup\{+\infty\} )^p$.
 Finally, suppose there exists $\mathbf{u}=(u_1,\dots, u_p)\in (\mathbb{R}^+)^p $ verifying~\eqref{111}  
and
		\begin{equation}\label{fullmeasure1}	
\mu_1\times\cdots\times\mu_p\left(\limsup_{n\to\infty}{ \prod_{i=1}^p}B_{i,n}^{s_{i,n}}\right)= \mu_1\times\cdots\times\mu_p \left( {\prod_{i=1}^p}X_i\right),
		\end{equation}
where \begin{equation}\label{sin}
s_{i,n}:=\frac{u_i\delta_i}{v_{i,n}}  \, . \end{equation}  
Then
\begin{equation} \label{pointcase_result}
			\dim_{\rm H}\left(\limsup_{n\to\infty} {\prod_{i=1}^p}B_{i,n} \right)\ge  s_0(\mathbf{u},\mathbf{v}).
		\end{equation}
\end{theorem}

\bigskip
\begin{remark}\label{rem6666}
Note that since we assume that ${\mathbf{u}}, {\mathbf{v}} $ satisfy  \eqref{111}, we have that \eqref{dimensionalnumberUB} holds. Then,  this  together  with the fact that \eqref{ARdim} holds for each set $X_i$ ($1 \le i \le p$), implies  that
$$
s_0(\mathbf{u},\mathbf{v})  \leq  \sum_{ 1 \le i \le p }\delta_i  = \dim  \Big( {\prod_{i=1}^p}X_i\Big)  \, .
$$
Thus, the lower bound inequality \eqref{pointcase_result}  is consistent with the fact that `rectangular' $\limsup$ set under consideration is a subset of the product space $ {\prod_{i=1}^p}X_i $ and so
$$
\dim_{\rm H}\left(\limsup_{n\to\infty} {\prod_{i=1}^p}B_{i,n} \right) \le \dim  \Big( {\prod_{i=1}^p}X_i\Big)  \, .
$$
\end{remark}

%
%

\bigskip

\begin{remark}\label{rem666}
 Note that in view of  \eqref{ezsv} and \eqref{sin},  for each $1\le i\le p$, the scaled ball
\begin{equation} \label{B_in_s_in}
B_{i,n}^{s_{i,n}}:=     B(x_{i,n},r(B_{i,n})^{\frac{s_{i,n}}{\delta_i}})  =   B(x_{i,n},r_n^{u_i}) \, .
\end{equation}  Thus the full measure condition \eqref{fullmeasure1} can be rewritten as
  \begin{equation}\label{fullmeasure2}	\mu_1\times\cdots\times\mu_p\Big(\limsup_{n\to\infty}{ \prod_{i=1}^p}B(x_{i,n},r_n^{u_i})\Big)= \mu_1\times\cdots\times\mu_p \Big( {\prod_{i=1}^p}X_i\Big).
  \end{equation}
Also, observe that  in view of \eqref{r_n_tends_to_zero},  if we have the strict inequality     $u_i <  v_i$,  it follows that for $n$ sufficiently large (so that $r_n<1$) the radius  of any ball $B_{i,n}$  from the given sequence of balls  is less than the radius of the scaled ball  $B_{i,n}^{s_{i,n}}$.
Thus $$R_n:={\prod_{i=1}^p}B_{i,n}  \ \subset  \  R_n^{\mathbf{s}_n}:= {\prod_{i=1}^p}B_{i,n}^{s_{i,n}}$$
  and in short, if the $ \limsup $ set of the  `blown up' rectangles  $ R_n^{\mathbf{s}_n} $  has full  measure then the  Mass Transference Principle provides a lower bound for the dimension of the $ \limsup $ set of the given rectangles  $ R_n $.  We will see in Remark~\ref{111_to_222S} that when it comes to proving Theorem~\ref{pointcase} we can assume without loss of generality that the strict inequality
$u_i\ < v_i$  holds.

   \end{remark}

\bigskip

\begin{remark}\label{rem:7a} It is worth emphasizing that  the assumption  on the existence of the sequences $\{r_n\}_{n\in\mathbb{N}}$ and $\{v_{i,n}\}_{n\in\mathbb{N}}$  is   automatically  satisfied for balls satisfying the following relatively mild `limit condition' on the radii.  As in the theorem,  for each $1\leq i\leq p$,  let   $\{B_{i,n}\}_{n \in \N}$ be a sequence of balls $B_{i,n}$ in $X_i$ with radius $r(B_{i,n}) \to 0$ as $n\to\infty$.
 \emph{We say that the collection $\{B_{i,n}: 1\leq i\leq p, n \in \N \}$ of balls   satisfies the limit condition {\rm \textbf{LC}} if for each pair $(i,j)$ with $1\leq i, j\leq p$, the limit
$$
a_{i,j}:=\lim_{n\to\infty} {\log r(B_{i,n}) \over \log r(B_{j,n})}
$$
exists (possibly equal to $0$ or $+\infty$).}
To see that (\ref{ezsv}) and (\ref{cond:r-n}) are satisfied for such  collection of balls,   let
$$
r_n^*:=\max\{r(B_{i,n}): 1\le i\le p\}.
$$
Then, since $p$ is finite  there exists an index $1\le j_o\le p$ and an infinite sequence $\{n_k\}_{k \in \N}$ of integers such that $r_{n_k}^*=r(B_{j_o, n_k})$ for all $k\in \N$. Now in view of limit condition  \textbf{LC}, for each  $1\le i\le p$ the limit  $$
a_{i, j_o}=\lim_{n\to \infty}\frac{\log r(B_{i,n})}{\log r(B_{j_o, n})}=\lim_{k\to \infty}\frac{\log r(B_{i, n_k})}{\log r(B_{j_o, n_k})}\ge 1
$$ exists.  The upshot of this  is that  (\ref{ezsv}) and (\ref{cond:r-n}) are satisfied  on taking
$$r_n:=r(B_{j_o, n}),   \qquad  v_{i,n}:=\frac{\log r(B_{i,n})}{\log r(B_{j_o, n})}   \qquad {\rm and } \qquad  v_i := a_{i, j_o} \, .  $$
\end{remark}

\bigskip

\begin{remark}  \label{WW rem}
In the statement of Theorem~\ref{pointcase}, if for each $1\le   i\le   p$ the  sequence of positive real numbers $\{v_{i,n}\}_{n\in\mathbb{N}}$ can be taken  to be a finite positive constant sequence; that is
\begin{equation} \label{const_vi}
0< v_{i,n}=v_i < \infty   \text{ for all } n \in \N  \, , \end{equation} then the theorem  reduces to (bounded)  Mass Transference Principle  of  Wang \& Wu \cite[Theorem 3.4]{WW2021} - see also \cite[Theorem 11]{LLVZstand}.  Note that in this case,  we can simply write $s_i$ for  $s_{i,n}$  since by definition it is independent of $n \in \N$.  To be precise, the Mass Transference Principle of  Wang \& Wu  incorporates  the more general framework of resonant sets and under \eqref{const_vi} our theorem reduces to theirs when the resonant sets are restricted to points. In \S\ref{extendingour},   we briefly discuss the resonant set framework and   we formulate the general form of Theorem~\ref{pointcase}  that incorporates resonant sets.   
\end{remark}

\bigskip

\begin{remark}  \label{WwwwwW}    With the previous remark in mind, even in bounded case when for each $1\le   i\le   p$
\begin{equation} \label{finite_vi}
v_i<+\infty 
\, ,
\end{equation}
 the statement of our theorem is more general than that of Wang \& Wu \cite[Theorem 3.4]{WW2021}. In short we have the `extra' flexibility of  the  sequence $\{v_{i,n}\}_{n\in\mathbb{N}}$ at out disposal.  This, for example, allows to directly deal  with the following situation.
For $ 1 \le i \le 2$, let $ \psi_i: \mathbb{R}^+\to\mathbb{R}^+ :  x \to  \psi_i(x) :=  x^{-\tau_i}(\log x)^{-(1+\tau_i)}  $ where $ \tau_i > 0$, and consider the corresponding set $W_2(\Psi)$.  Now with reference to   Theorem~\ref{pointcase}, for $n \in \N$  let  $r_n = n^{-1} $,  $v_{i,n} = (1+\tau_i)(1+(\log \log n)/\log n))$ and $v_i = 1+\tau_i$.    For convenience assume that $\tau_1\ge \tau_2$  and  $\tau_1+\tau_2>1$.  Then, Theorem~\ref{pointcase} implies that \begin{equation}  \label{lkj}
\dim_{\rm H}W_2(\Psi)  \ge \min\left\{\frac{3+\tau_1-\tau_2}{1+\tau_1}, \frac{3}{1+\tau_2}\right\} \, .  \end{equation}
Clearly, the (bounded)  Mass Transference Principle  of  Wang \& Wu is not applicable (at least directly)  since the exponents $v_{i,n}$ are not constant. Indeed, for the same reason, even when $\tau_1=\tau_2$ so that $ W_2(\Psi) $ is a  $\limsup$ of balls, the original Mass Transference Principle  (Theorem~BV) is not applicable.

\end{remark}

\bigskip

\begin{remark}  \label{WwwwwWsv}
It is worth pointing  out that in the bounded case our more general theorem  can in fact  be easily deduced from the
 theorem of Wang \& Wu (that is, the statement of Theorem~\ref{pointcase} in which \eqref{const_vi} is assumed).
To see this, assume that we are within the framework of Theorem~\ref{pointcase}  with  the additional  assumption \eqref{finite_vi}.
Choose and fix $\epsilon>0$    and let
$$
\mathbf{\overline{v}}= (\overline{v}_1,\dots,\overline{v}_p) := (1+\epsilon)\mathbf{v}.
$$
In view of \eqref{cond:r-n}, for $n$ sufficiently large we have that
\begin{equation} \label{v_in_is_small}
v_{i,n} \leq \overline{v}_i \quad\text{ for each } \ i=1,\dots,p.
\end{equation}
 Also, note  that in view of~\eqref{r_n_tends_to_zero}, for $n$ sufficiently large we have that
\begin{equation} \label{r_n_less_than_one}
r_n<1  \, .
\end{equation}
Assume for the moment that \eqref{v_in_is_small} and ~\eqref{r_n_less_than_one} hold for all $n\in\N$ rather than just for all but finitely many indices.  Then,  inequality~\eqref{v_in_is_small} implies that
\begin{equation*} \label{ie_s_in}
s_{i,n}=\frac{u_i\delta_i}{v_{i,n}}\geq\frac{u_i\delta_i}{\overline{v}_i}
=:\overline{s}_{i} \quad\text{ for each } \ i=1,\dots,p \ \text{ and }  \ \forall \ n \in \N .
\end{equation*}

\noindent Now denote by $\tilde{B}_{i,n}$  the ball with the same centre as $B_{i,n}$ and radius $r(\tilde{B}_{i,n})=r_n^{\overline{v}_i}$.
Then, by definition (see~\eqref{B_in_s_in}) we have that
\begin{equation*} \label{ie_s_inball}
B_{i,n}^{s_{i,n}}=\tilde{B}_{i,n}^{\overline{s}_i} \quad\text{ for each } \ i=1,\dots,p \ \text{ and }  \ \forall \ n \in \N .
\end{equation*}
 The upshot of this is that we can apply Theorem~\ref{pointcase} with 
the same vector $\mathbf{u}$  but with $v_{i,n}$ replaced by the finite constant sequence $v_{i,n}=\overline{v}_i$ for all $i=1,\dots,p$ and $n\in\N$  (i.e. we are within the setup of the theorem of Wang \& Wu).  On applying the theorem we conclude that
\begin{equation*} \label{BB1}
\dim_{\rm H}\left(\limsup_{n\to\infty}\prod_{i=1}^p\tilde{B}_{i,n} \right) \geq s_0(\mathbf{u},\mathbf{\overline{v}}).
\end{equation*}
In turn, together with the fact that inequalities~\eqref{v_in_is_small} and ~\eqref{r_n_less_than_one} imply that $\tilde{B}_{i,n}\subset B_{i,n}$ for every $i=1,\dots,p$ and all $n\in\N$, it follows that   for any $\epsilon >0$
\begin{equation} \label{BB3}
\dim_{\rm H}\left(\limsup_{n\to\infty} {\prod_{i=1}^p}B_{i,n}\right) \ \geq \ \dim_{\rm H}\left(\limsup_{n\to\infty} {\prod_{i=1}^p}\tilde{B}_{i,n} \right)
\ \geq    \
s_0(\mathbf{u}, (1+\epsilon)\mathbf{v})  \, .
\end{equation}
The desired statement follows from \eqref{BB3} by  making use of \eqref{fufu}.  Finally,  observe that the  $\limsup$ sets under consideration remain unchanged if we drop a finite set of indices $n$ and so, in the above argument,  we can indeed assume that \eqref{v_in_is_small} and~\eqref{r_n_less_than_one} hold for all $n\in\N$.
\end{remark}
\bigskip

\begin{remark} \label{111_to_222S}  It is convenient to note at this point that when it comes to proving Theorem~\ref{pointcase},
we can assume without loss of generality that
\begin{equation} \label{222S} 
u_i\ < v_i  \quad  \ {\text{for each}}\ \ 1\le i\le p  \, .
\end{equation}
Recall that these inequalities being strict is  also partly the subject  of Remark~\ref{rem666}.
To see that \eqref{222S} suffices, assume that we are within the framework of Theorem~\ref{pointcase} with $\mathbf{u}$ and $\mathbf{v}$ satisfying~\eqref{111}. Let $\epsilon>0$ be arbitrary. Then the vectors $(1-\epsilon)\mathbf{u}$ and $\mathbf{v}$ satisfy~\eqref{222S} and it is easily seen via \eqref{fullmeasure2} that if the   full measure condition  \eqref{fullmeasure1} holds for $\mathbf{u}$ and $\mathbf{v}$  then it also holds with $\mathbf{u}$ replace by $(1-\epsilon)\mathbf{u}$.  Thus,  Theorem~\ref{pointcase} implies  that
\begin{equation} \label{BB4}
\dim_{\rm H}\left(\limsup_{n\to\infty} {\prod_{i=1}^p}B_{i,n}\right) \geq s_0((1-\epsilon)\mathbf{u}, \mathbf{v}).
\end{equation}
In view of  \eqref{cucu}
\[
s_0((1-\epsilon)\mathbf{u}, \mathbf{v})=s_0\left(\mathbf{u}, \textstyle{\frac{1}{(1-\epsilon)}}\mathbf{v}\right)    \, ,
\]
 and so the desired statement \eqref{pointcase_result} follows  from \eqref{BB4} by  making use of \eqref{fufu}.
\end{remark}
\bigskip

\begin{remark} \label{NoToHM}   With the original Mass Transference Principle in mind (Theorem~BV), we note that  it is not possible to obtain  an analogue of  Theorem~\ref{pointcase} for  Hausdorff measures without imposing stronger assumptions. The reason for this is simple.  Even though the radii of the balls $B_{i,n}$ depend on the sequence $v_{i,n}$, the Hausdorff dimension $s_0$ of the associated $\limsup$ set depends only on its limiting behaviour $v_i$. However, as the following situation shows the $s_0$-Hausdorff measure can be drastically effected by the particular  sequence $v_{i,n}$ under consideration.  For  $\tau > 1$, let $ \psi: \mathbb{R}^+\to\mathbb{R}^+ :  x \mapsto  \psi(x) :=  x^{-\tau}(\log x)^{-(1+\tau)}  $ and consider the corresponding  sets $W_1(\tau)$  and  $W_1(\Psi)$.  With reference to Theorem~\ref{pointcase}, the first of these  is  a $\limsup$ set of balls  with  $r_n = n^{-1} $ and  $v_{1,n} =1+ \tau $ while the second is  a $\limsup$ set of balls  with  $r_n = n^{-1} $ and  $v_{1,n} = (1+\tau)(1+(\log \log n)/\log n))$.  In both cases, $ v_1= \lim_{n \to \infty} v_{1,n} =   1+\tau$ and so Theorem~\ref{pointcase} implies the lower bound (the hard part)  in  following dimension statement:
$$
\dim_{\rm H}W_{1}(\tau) = \dim_{\rm H}W_{1}(\psi) = s_0 :=  \frac{2}{1+ \tau} \, .
$$
Regarding the $s_0$-dimensional Hausdorff measure,  it follows from Jarn\'{\i}k's Theorem  \cite[Theorem 1.3.4]{durham} that
$$
\mathcal{H}^{s_0} \big( W_{1}(\tau) \big) = \infty    \qquad   {\rm and }    \qquad \mathcal{H}^{s_0} \big( W_{1}(\psi) \big) = 0  \, .
$$
In fact, the first of these statements can easily be deduced from Theorem~BV while the other can be easily deduced via a simple covering argument.
\end{remark}

\smallskip

\begin{remark} \label{NoToHMbut}
As mentioned in  Remark~\ref{Yubin}, recently He \cite{He} has established a `rectangular to rectangular'  Mass Transference Principle from the view point of Hausdorff content. In short,
he `essentially' shows  \cite[Theorem 2.9]{He} that if we  replace the full measure condition \eqref{fullmeasure1}	 by the stronger assumption of uniform local ubiquity (see \cite[Definition 1.7]{He}), then information regarding  the Hausdorff content of `blown up'  $\limsup$ sets implies `full'  Hausdorff measure statements (as well as the  large intersection property) for the original sets.   We say `essentially' since the full measure part of the statement also requires that the  implicit  sequences corresponding to  $s_0(\mathbf{u},\mathbf{v}(n))$ are  non-increasing.
Clearly, such an assumption is necessary as shown by the analysis in Remark~\ref{NoToHM} concerning the $\mathcal{H}^{s_0}$-measure of set $W_1(\psi)$  --  the sequence  $s_0(\mathbf{u},\mathbf{v}(n))=2/v_{1,n}$ is increasing.
Without the  non-increasing hypothesis, \cite[Theorem 2.9]{He} yields a dimension result only akin to our Theorem~\ref{pointcase} but under the significantly stronger assumption of uniform local ubiquity. Indeed, it is worth emphasising that  under the weaker  full measure condition \eqref{fullmeasure1} of Theorem~\ref{pointcase},   He obtains a dimension statement  \cite[Theorem 2.10]{He} but only within the bounded setup of Wang \& Wu (see Remark~\ref{WW rem}). The upshot is that under \eqref{fullmeasure1}, the results in \cite{He} cannot be applied to obtain statements such as Theorem~\ref{matrixEG} or   \eqref{easy} within the unbounded setup or indeed \eqref{lkj} within the bounded setup.


\end{remark}

%
%

\subsection{Proof of Theorem~\ref{Runbounded} modulo  Theorem~\ref{pointcase} \label{classicalproof} }

 The proof of Theorem~\ref{Runbounded} will involve  establishing  the upper bound (Proposition \ref{upperbound-unbdd-d} below) and lower bound (Proposition \ref{lowerbound-d} below) for $ \dim_{\rm H} W_d(\Psi) $ separately.

\subsubsection{The  upper bound  }

The goal is to prove the following statement.

\begin{proposition}\label{upperbound-unbdd-d}

For  $1\leq i\leq d$, let  $\psi_i: \mathbb{R}^+\to\mathbb{R}^+$ be a real, positive, non-increasing function.
 Then
  $$\dim_{\rm H} W_d(\Psi) \leq
  \sup_{\mathbf{t}\in\UP}  \min \big\{ \min_{  i \in \mathcal{L}(\mathbf{t}) }\zeta_i(\mathbf{t}),\ \#\mathcal{L}(\mathbf{t})\big\},$$
	where given $\mathbf{t}=(t_1,\dots,t_d) \in (\R^+\cup \{+\infty\})^d$  we set  $\mathcal{L}(\mathbf{t}):=\{1\leq i\leq d: t_i<+\infty\}$.		
\end{proposition}

	\begin{proof}
%
When $\UP$ is bounded, by definition  we have that $\#\mathcal{L}(\mathbf{t}) = d $ and so the  desired statement corresponds to the result of Wang \& Wu (\S\ref{int}, Theorem~WW).  Thus, without loss of generality,  assume that $\UP$ is unbounded.

For any $\epsilon>0$, let $M>0$ be sufficiently large so that
\begin{align}\label{epsilon-M}
\frac{d+1}{1+M}<\epsilon.
\end{align}
Let $
\widetilde{\Psi}:=(\widetilde{\psi}_i)_{1\leq i \leq d}  \ $ where $$  \widetilde{\psi}_i(q):=\max\{\psi_i(q), \ q^{-M}\} \ \ \qquad  (1\le i\le d)
$$ and define $W_d(\widetilde{\Psi})$ and $\mathcal{U}(\widetilde{\Psi})$ accordingly.  Then, by definition
$
W_d(\Psi) \subset W_d(\widetilde{\Psi})
$ and in turn, since $\mathcal{U}(\widetilde{\Psi})$ is bounded,  it follows via Theorem~WW that \begin{equation}\label{f5}
\dim_{\rm H} W_d(\Psi)\leq  \dim_{\rm H} W_d(\widetilde{\Psi})= \sup_{{\widetilde{\bold{t}}}\in\mathcal{U}(\widetilde{\Psi})} \min \left\{\min_{1\le i\le d} {\zeta}_i({\widetilde{\bold{t}}}), \, d\right\}.
\end{equation}

\noindent Now, let $\widetilde{\bold{t}}\in \mathcal{U}(\widetilde{\Psi})$ be a point such that
\begin{equation}\label{special-t}
\min \left\{\min_{1\le i\le d} {\zeta}_i(\widetilde{\bold{t}}), \, d\right\}>\sup_{{\widetilde{\bold{t}}}\in\mathcal{U}(\widetilde{\Psi})} \min \left\{\min_{1\le i\le d} {\zeta}_i({\widetilde{\bold{t}}}), \, d\right\}-\epsilon  \, .
\end{equation}
 Note that for any $\widetilde{\bold{t}}\in \mathcal{U}(\widetilde{\Psi})$,  there exists an unbounded integer sequence $\{q_k\}_{k\ge 1}$ such that
$$
\lim_{k\to\infty}\frac{-\log \widetilde{\psi}_i(q_k)}{\log q_k}=\widetilde{t}_i  \qquad \forall \ 1\le i\le d
$$
By the definition of $\widetilde{\Psi}$,  we clearly  have that $\widetilde{t}_i   \le M $ for each  $1\le i\le d$.   With this in mind, let  $\mathcal{L}_M(\widetilde{\bold{t}}):=\{1\le i\le d: \widetilde{t}_i<M\}$. In turn,
 let $t_i=\widetilde{t}_i$ for $i\in \mathcal{L}_M(\widetilde{\bold{t}})$ and for $i\not\in \mathcal{L}_M(\widetilde{\bold{t}})$, let $$t_i=\lim_{j\to\infty}\frac{-\log {\psi}_i(q_{k_j})}{\log q_{k_j}},$$
for some subsequence $q_{k_j}$ of $\{q_k\}_{k\ge 1}$.
Then, it is easy to check that  ${\bold{t}}:=({t}_1,\cdots, {t}_d)\in \mathcal{U}(\Psi)$. Moreover, by the definitions of $\widetilde{\psi}_i$ and $\mathcal{L}_M(\widetilde{\bold{t}})$, one has that
\begin{align}\label{relation-t-i}
{t}_i=\widetilde{t}_i<M   \ \  {\text{for}}\ i\in \mathcal{L}_M(\widetilde{\bold{t}})    \quad \text{and} \quad   {t}_i\ge \widetilde{t}_i=M   \ \ {\text{for}}\ i\not\in \mathcal{L}_M(\widetilde{\bold{t}}).\end{align}

\noindent We now show that
\begin{equation}\label{two-xi}
 {\zeta}_i(\widetilde{\bold{t}})  < {\zeta}_i({\bold{t}})+\epsilon    \qquad   \forall \ 1\le i\le d  \, .
\end{equation}
To this end, first note that $\mathcal{L}_M(\widetilde{\bold{t}})\subset \mathcal{L}({\bold{t}})$ and that for $i\not\in \mathcal{L}_M(\widetilde{\bold{t}})$, by \eqref{relation-t-i} and \eqref{epsilon-M}, we have
\begin{align}\label{compare-zeta}
\begin{split}
\zeta_i(\widetilde{\bold{t}})&=\frac{d+1+\sum_{k: \widetilde{t}_k<\widetilde{t}_i}(\widetilde{t}_i-\widetilde{t}_k)}{1+\widetilde{t}_i}=\frac{d+1+\sum_{k\in \mathcal{L}_M(\widetilde{\bold{t}})}(M-t_k)}{1+M}\\
&\le \# \mathcal{L}_M(\widetilde{\bold{t}})-\frac{\sum_{k\in \mathcal{L}_M(\widetilde{\bold{t}})}t_k}{1+M}+\epsilon.
\end{split}
\end{align}
 To continue, we consider the following three cases:
\begin{itemize}
  \item For $i\in \mathcal{L}_M(\widetilde{\bold{t}})$, it is clear that $\zeta_i(\widetilde{{\bold{t}}})=\zeta_i({\bold{t}})$ since both of them are only related to the components in $\mathcal{L}_M(\widetilde{\bold{t}})$ which are equal for  $\widetilde{\bold{t}}$ and ${\bold{t}}$.
  \item For $i\in \mathcal{L}({\bold{t}})\setminus \mathcal{L}_M(\widetilde{\bold{t}})$, by \eqref{relation-t-i} and \eqref{compare-zeta},
  \begin{align*}
    \zeta_i({\bold{t}})&=\frac{d+1+\sum_{k: {t}_k<{t}_i}({t}_i-{t}_k)}{1+{t}_i}\ge \frac{d+1+\sum_{k\in \mathcal{L}_M(\widetilde{\bold{t}})}({t}_i-{t}_k)}{1+{t}_i}\\
    &\ge \# \mathcal{L}_M(\widetilde{\bold{t}})-\frac{\sum_{k\in \mathcal{L}_M(\widetilde{\bold{t}})}t_k}{1+M}\ge \zeta_i(\widetilde{\bold{t}})-\epsilon.
  \end{align*}
  \item For $i\not\in \mathcal{L}({\bold{t}})$, it is clear that $i\not\in \mathcal{L}_M(\widetilde{\bold{t}})$ and so  by definition (see Remark~\ref{def2024}) and \eqref{compare-zeta},
  $$
  \zeta_i({\bold{t}})=\# \mathcal{L}({\bold{t}})\ge \# \mathcal{L}_M(\widetilde{\bold{t}})\ge \zeta_i(\widetilde{\bold{t}})-\epsilon.
  $$
\end{itemize}
Together, these three cases imply  \eqref{two-xi} and so it  follows that for any $\widetilde{\bold{t}}\in \mathcal{U}(\widetilde{\Psi})$ there exists a corresponding ${\bold{t}} \in \mathcal{U}(\Psi)$ such that
$$
 \min \left\{\min_{1\le i\le d} {\zeta}_i(\widetilde{\bold{t}}), \, d\right\} \ < \   \min \left\{\min_{1\le i\le d} {\zeta}_i({\bold{t}}), \, d\right\}+\epsilon.$$
In particular,  on combining with \eqref{f5} and  \eqref{special-t} we find  that
\[
\dim_{\rm H} W_d(\Psi)\ \leq \ \min \left\{\min_{1\le i\le d} {\zeta}_i({\bold{t}}), \, d\right\}+2\epsilon  \ \leq  \  \sup_{\mathbf{t}\in\mathcal{U}({\Psi})}\min \left\{\min_{1\le i\le d} {\zeta}_i({\bold{t}}), \, d\right\}+2\epsilon  \, .
\]
Now $ \epsilon > 0$ is arbitrary and  so the desired statement follows on letting $\epsilon \to 0$ and noting  that by definition (see Remark~\ref{def2024}), in the unbounded case for any $ \mathbf{t}\in\mathcal{U}({\Psi}) $ with $\#\mathcal{L}(\mathbf{t}) \neq  d $ we have  that
$$ { \min } \left\{ \min_{1\le   i\le   d}\zeta_i(\mathbf{t})   { , d }  \right\} \  =  \   \min_{  1 \le i \le d } \zeta_i(\mathbf{t})    \ = \ { \min } \left\{ \min_{1\le   i\le   d}\zeta_i(\mathbf{t})   { , \#\mathcal{L}(\mathbf{t}) }  \right\}   \, .
$$

\medskip

We conclude by mentioning that the above argument  is equally valid if  $\mathcal{L}_M(\widetilde{\bold{t}}) = \emptyset$ for all  $M > 0$ sufficiently large. Indeed, in this case \eqref{compare-zeta} becomes $\zeta_i(\widetilde{\bold{t}})\leq \epsilon$ for all $1\leq i \leq d$.

\bigskip

\end{proof}

\subsubsection{The  lower bound  modulo  Theorem~\ref{pointcase} \label{lowerbound-th1-1}  }

To prove the complementary lower bound for  $\dim_{\rm H} W_d(\Psi)$  we will  make use of the following  full measure result.  The proof will be given at the end of this section.

\begin{lemma}\label{ll2}
Let $d \in \N$ and   $\mathbf{a} =({a}_1,\cdots,{a}_d)  \in (\R^+)^d  $ such that  $$
a_i \ge 1   \quad  (1\le i\le d)  \qquad   { and } \qquad
{a}_1+{a}_2+\cdots+{a}_d= d+1.
$$  Let  $M= 4^{d+1}$  and $\widetilde{M}=\max_{1\leq i \leq d}M^{1/a_i}$. Then, for any sequence $\{q_\ell\}_{\ell\ge 1}$ of integers, the limsup set $$W:=\limsup_{\ell\to \infty}E_\ell$$ with
$$
E_\ell:=\bigcup_{\frac{1}{M}q_\ell\le q \le q_\ell}\ \bigcup_{0\le p_1,\cdots, p_d  \le q}\ \prod_{i=1}^d B\left(\frac{p_i}{q},  \left(\frac{\widetilde{M}}{q_\ell}\right)^{a_i}\right)
$$ is of full Lebesgue measure.
\end{lemma}


The following proposition constitutes the desired lower bound.

\begin{proposition}\label{lowerbound-d}
  For  $1\leq i\leq d$, let  $\psi_i: \mathbb{R}^+\to\mathbb{R}^+$ be a real, positive, non-increasing function.
 Then
  $$\dim_{\rm H} W_d(\Psi) \geq
  \sup_{\mathbf{t}\in\UP}  \min \left\{ \min_{  i \in \mathcal{L}(\mathbf{t}) }\zeta_i(\mathbf{t}),\ \#\mathcal{L}(\mathbf{t})\right\},$$
	where given $\mathbf{t}=(t_1,\dots,t_d) \in (\R^+\cup \{+\infty\})^d$  we set  $\mathcal{L}(\mathbf{t}):=\{1\leq i\leq d: t_i<+\infty\}$.
\end{proposition}


\begin{proof}

Fix any point $\mathbf{t}=(t_1,\dots,t_d)\in\UP$. The proof of the proposition follows on showing \begin{align}\label{lower-continuity}
 \dim_{\rm H} W_d(\Psi)\geq  \min\left\{\min_{i\in \mathcal{L}(\mathbf{t})}\zeta_i(\mathbf{t}),\ \#\mathcal{L}(\mathbf{t})\right\}   \, .
 \end{align}
 Without loss of generality, we assume that $\psi_i(q) \to 0$ as $q\to \infty$   for all $1\leq i \leq d$.  Indeed, if this was not the case we consider the function $
\widetilde{\Psi}:=(\widetilde{\psi}_i)_{1\leq i \leq d}  \ $ where $$  \widetilde{\psi}_i(q):=\min\{\psi_i(q), \  (\log q)^{-1} \} \ \ \qquad  (1\le i\le d)
$$ and define $W_d(\widetilde{\Psi})$ and $\mathcal{U}(\widetilde{\Psi})$ accordingly.  Then, by definition $\widetilde{\psi}_i(q) \to 0$ as $q\to \infty$   for all $1\leq i \leq d$  and  $ W_d(\Psi)    \supset W_d(\widetilde{\Psi}) $.  Moreover,  $\mathcal{U}(\widetilde{\Psi})   =  \mathcal{U}({\Psi})$  and so the desired statement \eqref{lower-continuity} for $W_d(\Psi)$  follows on establishing \eqref{lower-continuity} with $\Psi$ replaced by $\widetilde{\Psi}$.    Also note that we can assume that  $\#\mathcal{L}(\mathbf{t}) \neq  0 $, since otherwise the right hand side of \eqref{lower-continuity} is equal to zero and the inequality  is trivially true.  Furthermore, we will assume that $\#\mathcal{L}(\mathbf{t}) \neq  d $.  Indeed, if $\#\mathcal{L}(\mathbf{t}) = d $ for all $\mathbf{t}\in\UP$  the desired statement corresponds to the result of Wang \& Wu (\S\ref{int}, Theorem~WW).  With these assumptions in mind, we proceed with  establishing \eqref{lower-continuity}.

\noindent \emph{Step 1.} \  We first deal with the situation  when

\begin{equation}\label{4sv}
\sum_{i\in\mathcal{L}(\mathbf{t})} t_i  \, <  \, 1  \, . \end{equation}  For notational simplicity, we assume $\mathcal{L}(\mathbf{t})=\{1,\cdots, m\}$ with  $1\leq m \leq d-1$.  Since $\mathbf{t}\in\UP$ and satisfies \eqref{4sv}, it follows form the  definition of $\UP$  that there exists an unbounded sequence of integers $\{q_{\ell}\}_{\ell\ge 1}$
such that $$
q_{\ell}\cdot \prod_{i=1}^m\psi_i(q_{\ell})>1 \qquad  \forall \ \ell\ge 1.
$$
Thus, by  Minkowski's  Convex Body Theorem (\cite[Chp~2: Theorem~2B]{Sch1980}), for any  $\vx=(x_1, \ldots, x_m)\in\I^m :=[0,1]^m$   and each $\ell\ge 1$, there exists a non-zero integer point  $(p_1,\cdots, p_m, q) \in \Z^m \times \N $ with $1\le q\le q_{\ell}$ such that $$
|qx_i-p_i|<\psi_i(q_{\ell})\le \psi_i(q)   \ \ \qquad  (1\le i\le m)  \, .
$$
Since $\psi_i(q_{\ell})\to 0$ as $\ell\to \infty$ for all $1\leq i \leq d$,  
we conclude that there exist infinitely many integers $q\in \N$ such that $$
\|qx_i\|<\psi_i(q) \ \ \qquad  (1\le i\le m)  \, .
$$ Since $m<d$,  the upshot of the above is that
$$
\I^m  \ \times \prod_{i=m+1}^d\{0\}\subset W_d(\Psi)  \, .
$$
Hence,  $\dim_{\rm H} W_d(\Psi)\geq m=\#\mathcal{L}(\mathbf{t})$ and this verifies \eqref{lower-continuity} in the situation that \eqref{4sv} holds.

\medskip

\noindent \emph{Step 2.} \  We now deal with the `other' situation.  That is,  to verify  \eqref{lower-continuity} with  $\mathbf{t}\in\UP$ satisfying \begin{equation}\label{4}
\sum_{i\in\mathcal{L}(\mathbf{t})} t_i\ge 1.
\end{equation}
For this we will make use of Theorem~\ref{pointcase}.
To start with,  note that since $\psi_i(q) \to 0 $ as $q\to\infty$, we have that
$$t_i\geq 0  \qquad \forall  \ 1\le i\le d \, . $$  
Also, by  definition, for any $ \mathbf{t} \in\UP$ there exists an integer  subsequence $\{q_\ell\}_{\ell\in\mathbb{N}}$ such that
$$\lim_{\ell\to\infty}\frac{-\log\psi_i(q_\ell)}{\log q_\ell}=t_i\qquad \forall \ 1\le i\le d \, .  $$
In turn,  consider the  following $\limsup$ set
\begin{align}\label{subset}
W \, := \, \limsup_{\ell\to\infty}\bigcup_{\frac{1}{M}q_\ell\le q\le q_\ell} \ \bigcup_{0\le p_1,\cdots, p_d  \le q} \ \prod_{i=1}^d
B\left(\frac{p_i}{q},\frac{\psi_i(q_\ell)}{q_\ell}\right)  \,
\end{align} 
 where  $M:=4^{d+1}$.  Then,  since for each $1\leq i\leq d$,  the function
  $\psi_i$   is  non-increasing  it follows that
  $$ W \subset W_d(\Psi)  \, , $$
and in view of  Lemma \ref{ll2}
\begin{equation} \label{measuresubset}
m_d\left(\limsup_{\ell\to\infty}\bigcup_{{1 \over M}q_\ell\le q\le q_\ell} \ \bigcup_{0\le p_1,\cdots, p_d  \le q}  \ \prod_{i=1}^d
B\left(\frac{p_i}{q}, \ \Big(\frac{\widetilde{M}}{q_\ell}\Big)^{a_i}\right)\right)
=m_d(\I^d)
\end{equation}
for any  $\mathbf{a}=(a_1,\dots, a_d) $  with
$$a_i  \ge 1   \quad ( 1\leq i\leq d) \qquad {\rm and  }    \qquad  a_1+\cdots+a_d=d+1  \, . $$
Now with Theorem~\ref{pointcase} in mind,  for $\ell \in \N$, let
 $$r_\ell={\widetilde{M} \over q_\ell} \, $$
 and for $1\leq i \leq d$, let
\[
\delta_i=1,  \,  \qquad  u_i=a_i, \qquad v_{i, \ell}={\log \psi_i(q_\ell)- \log q_\ell \over \log \widetilde{M}- \log q_\ell},
\]
which implies that
\[
v_i:=\lim_{\ell\to\infty}v_{i, \ell}=1+t_i \,.
\]
Note that if
$
 B_{i,\ell} := \textstyle{B\Big(\frac{p_i}{q},\frac{\psi_i(q_\ell)}{q_\ell}\Big)}
$ is a ball appearing in \eqref{subset} then
$$
B_{i,\ell}^{s_{i,\ell}}  := \textstyle{B\Big(\frac{p_i}{q},\big(\frac{\widetilde{M}}{q_\ell}\big)^{a_i}\Big)}
\qquad  {\rm with } \quad s_{i,n}:=\frac{a_i}{v_{i,n}}    $$
is the corresponding `enlarged' ball appearing in the full measure statement \eqref{measuresubset}.
Thus, in view of  Theorem \ref{pointcase} we conclude that
\begin{equation}
\dim_{\rm H}W_d(\Psi)\ \geq \ \dim_{\rm H}  W  \ \ge \   \min\left\{\min_{i\in \mathcal{L}(\mathbf{v})}s(\mathbf{u},\mathbf{v},i), \ \#\mathcal{L}(\mathbf{v})\right\}
\end{equation}
with $\mathbf{v}=  1 + \mathbf{t}$ and for any $\mathbf{u}=\mathbf{a}=(a_1,\dots, a_d) $  with
\begin{equation}  \label{poil} 1 \le a_i  \le   1  + t_i  \quad ( 1\leq i\leq d) \qquad {\rm and  }    \qquad  a_1+\cdots+a_d=d+1  \, . \end{equation}
\noindent It is evident that $\mathcal{L}(\mathbf{v})=\mathcal{L}(\mathbf{t})$. Thus, for \eqref{poll} to imply (\ref{lower-continuity}), it suffices to show that on choosing $\mathbf{a}=(a_1,\dots, a_d) $ satisfying \eqref{poil}  appropriately,  we can guarantee that  \begin{align}\label{min_suffice}
\min_{i\in \mathcal{L}(\mathbf{v})}s(\mathbf{u},\mathbf{v},i)=\min_{i\in \mathcal{L}(\mathbf{v})}\zeta_i(\mathbf{t}).
\end{align}

\noindent We divide the  proof  of this into two cases.

\bigskip

\emph{Case 1. \ }   Suppose $\min\limits_{1\leq i\leq d}t_i> \frac{1}{d}$. We choose
$$a_k=1+\frac{1}{d}>1 \ 
\qquad \forall \ 1\le i\le d \, .$$
Then, $a_1+\cdots+a_d=d+1$.  Thus, $\mathbf{a}=(a_1,\dots, a_d) $ satisfies \eqref{poil}  and by definition,
for each $i\in \mathcal{L}(\mathbf{v})=\mathcal{L}(\mathbf{t})$,
\begin{equation}\label{assump-sum-t}
	s(\mathbf{u},\mathbf{v},i)=
	\sum_{k\in \mathcal{K}_1(i)}1+\sum_{k\in \mathcal{K}_2(i)}\left(1-\frac{t_k-{1\over d}}{1+t_i}\right) +\sum_{k\in \mathcal{K}_3(i)}\frac{1+\frac{1}{d}}{1+t_i}
\end{equation}
where
\begin{equation*}
	\mathcal{K}_1(i):= \{1\le   k\le   d: 1+\frac{1}{d}> 1+t_i\}=\emptyset,\ \ \mathcal{K}_2(i):= \{k\in\mathcal{L}(\mathbf{t}): t_k\leq  t_i\}, \
\end{equation*}
and
\begin{equation*}
	\mathcal{K}_3(i):=\{1, \dots, d\}\setminus (\mathcal{K}_1(i)\cup \mathcal{K}_2(i)).
\end{equation*}
So,
\begin{eqnarray*} s(\mathbf{u},\mathbf{v},i) & = &\frac{\sum_{k:t_k\leq t_i}(1+\frac{1}{d}+t_i-t_k)+\sum_{k:t_k> t_i}(1+\frac{1}{d})}{1+t_i} \\[2ex]
& = & \frac{d+1  + \sum_{k: t_k<t_i}(t_i-t_k)}{1+t_i}=\zeta_i(\mathbf{t})  \, .
\end{eqnarray*}
Hence, \eqref{min_suffice} holds as desired.

\emph{Case 2. \ } Suppose $\min\limits_{1\leq i\leq d}t_i\leq  \frac{1}{d}$.
Let
\[
\mathcal{T}:=\left\{t_j: j\in\mathcal{L}(\mathbf{t})  \ {\rm and} \ t_j\geq\frac{1-\sum_{i:t_i<t_j}t_i}{\#\{1\leq i\leq d: t_i\geq t_j\}}\right\}.
\]
Note that if $t_{\rm max} :=\max_{j \in \mathcal{L}(\mathbf{t}) }t_j$ and $\#_{\rm max} :=  \# \{  j \in \mathcal{L}(\mathbf{t}) : t_j = t_{\rm max} \} $, then by~\eqref{4} we have that
$$
\sum_{i:t_i<t_{\rm max} }t_i  \  + \  \#_{\rm max} \, \cdot  t_{\rm max}
\ = \
 \sum_{i  \in\mathcal{L}(\mathbf{t}) }t_i  \  \ge  \ 1
$$
and so $t_{\rm max} \in \mathcal{T}$.  Thus, $\mathcal{T}\neq \emptyset$.   Now,
 let $t_K$ be the smallest element of $\mathcal{T}$ and let
$$t^*:=\frac{1-\sum_{i:t_i<t_K}t_i}{\#\{1\leq i\leq d: t_i\geq t_K\}}  \, . $$
Clearly, $t^*$ is the right hand side of the inequality appearing in the definition of $\mathcal{T}$ with $j=K$.  Hence, trivially
\begin{equation} \label{nm}
t_K \geq t^*  \,
\end{equation}
and moreover $t^*>0$.  Indeed, to see the latter suppose it was not the case. Then
\[
\sum_{i:t_i<t_K}t_i \geq 1.
\]
Let $t_{i_0}:=\max\{t_i: t_i<t_K\}$. It follows that
\[t_{i_0}  \  \geq   \  1-\sum_{i\neq i_0:t_i<t_K}t_i   \ \geq \  {1-\sum_{i\neq i_0:t_i<t_K}t_i \over \#\{1\leq i\leq d: t_i\geq t_K\} +1} \ = \  {1-\sum_{i:t_i<t_{i_0}}t_i \over \#\{1\leq i\leq d: t_i\geq t_{i_0}\}} \, ,   \]
and so  $t_{i_0} \in   \mathcal{T} $. However, this  contradicts the fact that  $t_K$ is the smallest element of $\mathcal{T}$.
\medskip

Now, for $1\leq i \leq d$, let
\[
a_i:=\begin{cases}
  1+t_i\ \ &\text{if}\ t_i<t_K\\[2ex]
  1+t^* \ \ &\text{if}\ t_i\geq t_K.
\end{cases}
\]
Since $t_i\geq 0$ and $t^*>0$, we have that $a_i\geq 1$ for all $1\leq i\leq d$.
Moreover, it follows from the definition of $t^*$ that
\begin{align}\label{sum_ai}
a_1+a_2+\cdots+a_d=d+1.
\end{align}
Thus, $\mathbf{a}=(a_1,\dots, a_d) $ satisfies \eqref{poil}  and by definition,
for each $i\in \mathcal{L}(\mathbf{v})=\mathcal{L}(\mathbf{t})$,
 \begin{equation*}
	s(\mathbf{u},\mathbf{v},i)=
	\sum_{k\in \mathcal{K}_1(i)}1+\sum_{k\in \mathcal{K}_2(i)}\left(1-\frac{1+t_k-a_k}{1+t_i}\right) +\sum_{k\in \mathcal{K}_3(i)}\frac{a_k}{1+t_i}
\end{equation*}
where
\begin{equation*}
	\mathcal{K}_1(i):= \{1\le   k\le   d: a_k> 1+t_i\},\ \ \mathcal{K}_2(i):= \{k\in\mathcal{L}(\mathbf{t}): t_k \leq  t_i\}. \
\end{equation*}
and
\begin{equation*}
	\mathcal{K}_3(i):=\{1, \dots, d\}\setminus (\mathcal{K}_1(i)\cup \mathcal{K}_2(i)).
\end{equation*}

\medskip

To proceed, we further distinguish between two subcases:  indices $i\in\mathcal{L}(\mathbf{t})$ with  $t_i\geq t_K$ (Case~2.1) and, those with  $t_i < t_K$ (Case~2.2).

\medskip

\emph{ Case 2.1.} \ Suppose  $i\in\mathcal{L}(\mathbf{t})$ with $t_i\geq t_K$. Then, by definition, for all $1 \le k \le d $ such that $t_k<t_K$, we have $a_k=1+t_k<1+t_i$, and for all $1 \le k \le d $ such that $t_k\geq t_K$, we have $a_k=1+t^*\leq1+t_K\leq 1+ t_i$. Hence, $\mathcal{K}_1(i)=\emptyset$ and since  $i\in\mathcal{L}(\mathbf{t})$, it is obvious that $\mathcal{K}_2(i)=\{1\leq k \leq d:t_k\leq t_i\}$.  Therefore, it follows on using \eqref{sum_ai} that
\begin{eqnarray}
\label{qaz}
s(\mathbf{u},\mathbf{v},i)
&=&\sum_{k:t_k\leq t_i} \left(1-\frac{1+t_k-a_k}{1+t_i}\right)+ \sum_{k:t_k> t_i}\frac{a_k}{1+t_i}   \nonumber \\[2ex]
&=&\frac{d+1  + \sum_{k: t_k<t_i}(t_i-t_k)}{1+t_i}=\zeta_i(\mathbf{t})  \, .
\end{eqnarray}

\emph{Case 2.2.} \ Suppose  $i\in\mathcal{L}(\mathbf{t})$ with $t_i< t_K$. Then, by the definition of $t_K$,  we have  that
\begin{equation*} 
t_i < \frac{1-\sum_{k:t_k< t_i}t_k}{\#\{1\leq k\leq d: t_k\geq t_i\}}.
\end{equation*}
This is equivalent to
\begin{equation}\label{t-i<t-k}
\#\{1\leq k\leq d: t_k\geq t_i\}\cdot t_i \  < \  1-\sum_{k:t_k< t_i}t_k,
\end{equation}
which together with the fact that $t_i< t_K$,   implies that
\begin{equation*} \label{ti_leq_sum_tk}
\#\{1\leq k\leq d: t_k\geq t_K\} \cdot t_i \ <  \  1-\sum_{k:t_k< t_K}t_k.
\end{equation*}
In the other words,
\begin{equation*} \label{t_i-t_K}
	 t_i \ <  \ \frac{1-\sum_{k:t_k<t_K}t_k}{\#\{1\leq k\leq d: t_k\geq t_K\}} \ = \ t^*.
\end{equation*}
Thus, for all $1 \le k \le d $ such that $t_k\geq t_K $, the above implies that $a_k:=1+t^*>1 + t_i$.  Obviously, if $t_K > t_k > t_i$ then $a_k:=1+t_k >  1 + t_i$. Hence,
$\mathcal{K}_1(i)= \{1\le   k\le   d: t_k>t_i\}$. On the other hand, since $i\in\mathcal{L}(\mathbf{t})$ it follows that  $\mathcal{K}_2(i)= \{1\leq k\leq d : t_k\leq t_i\}$
and hence $\mathcal{K}_3(i)= \emptyset.$  Therefore, 
\begin{eqnarray} \label{qazsv}
s(\mathbf{u},\mathbf{v},i)
&=&\sum_{k:t_k>t_i}1+ \sum_{k:t_k\leq  t_i}(1-{1+t_k-a_k \over 1+t_i}) \nonumber \\[2ex]
&=& \sum_{k:t_k>t_i}1+ \sum_{k:t_k\leq t_i}(1-{1+t_k-(1+t_k) \over 1+t_i})  \nonumber \\[2ex]
&=&d  \, .
\end{eqnarray}


\noindent Now observe, that in view of \eqref{qaz} with $i=K$  and  \eqref{nm}, it follows that
\begin{eqnarray} \label{zeta_K_leq_d}
s(\mathbf{u},\mathbf{v},K) = \zeta_K(\mathbf{t})&:= &\frac{d+1  + \sum_{k: t_k<t_K}(t_K-t_k)}{1+t_K}  \nonumber \\[2ex]
& = & d + \frac{\#\{k:t_k\geq t_K\}}{1+t_K}\cdot\left(-t_K+ t^* \right) \nonumber \\[2ex]
&\leq &  d,
\end{eqnarray}
which together with \eqref{qazsv}  implies that
\begin{equation} \label{patel}
s(\mathbf{u},\mathbf{v},K)\leq s(\mathbf{u},\mathbf{v},i)  \qquad \forall \quad i\in\mathcal{L}(\mathbf{t})   \quad {\rm with} \quad t_i< t_K  \, .
\end{equation}
 On the other hand, for any $i\in\mathcal{L}(\mathbf{t})$ with  $t_i<t_K$, we use \eqref{t-i<t-k}  to obtain
\begin{eqnarray*}\label{t_i}
\zeta_i({\mathbf{t}})&=&\frac{d+1  + \sum_{k: t_k<t_i}(t_i-t_k)}{1+t_i} \nonumber \\[2ex]
&>&\frac{d+\#\{k: t_k<t_i\} \cdot t_i+\#\{k: t_k\geq t_i\} \cdot t_i}{1+t_i}  \nonumber \\[2ex]
&=& d \, ,
\end{eqnarray*}
which together with~\eqref{zeta_K_leq_d} implies that
\begin{equation} \label{patel2}
\zeta_K(\mathbf{t})<\zeta_i(\mathbf{t})  \qquad \forall \quad i\in\mathcal{L}(\mathbf{t})   \quad {\rm with} \quad t_i< t_K  \, .
\end{equation}

\medskip   On  combining our findings in the above subcases, namely \eqref{qaz}, \eqref{patel}  and \eqref{patel2}, we find that
\begin{equation*}\label{compare1}
 \min_{i\in \mathcal{L}(\mathbf{v})}s(\mathbf{u},\mathbf{v},i) \ = \min_{i\in \mathcal{L}(\mathbf{v}): \ t_i\geq t_K}s(\mathbf{u},\mathbf{v},i)\ =\min_{i\in \mathcal{L}(\mathbf{v}): \ t_i\geq t_K}\zeta_i(\mathbf{t})=\min_{i\in \mathcal{L}(\mathbf{v})}\zeta_i(\mathbf{t}) .
\end{equation*}
This establishes~\eqref{min_suffice} in Case~2 and thereby completes the proof of the proposition.
\end{proof}

The above proof of Proposition~\ref{lowerbound-d} is based on the validity of Lemma~\ref{ll2} and of course Theorem~\ref{pointcase}.
We now establish Lemma~\ref{ll2}.

\bigskip

\begin{proof}[Proof of Lemma \ref{ll2}]
By the Lebesgue density theorem, it suffices to show that there is a constant $c>0$ such that for any ball $B\subset \I^d= [0,1]^d$, we have \begin{equation}\label{q2}
m_d(B\cap W)\ge c \cdot m_d(B).
\end{equation}
To this end, we will show that for all large $\ell  \in \N$,
$$
m_d(B\cap E_\ell)\ge c \cdot m_d(B),
$$ which clearly implies (\ref{q2}).   Throughout, for convenience,  we work with  the max norm.  Thus, for any ball $B=B(r)$ of  radius $r$ we have  $m_d(B)= (2r)^d$.

Given
$\mathbf{a} =({a}_1,\cdots,{a}_d)  \in (\R^+)^d  $  as in the statement of the lemma, let $\bar{\mathbf{a}} =(\bar{a}_1,\cdots,\bar{a}_d)   $  where $\bar{a}_i := a_i - 1   $  for each $1\le i\le d$.   It follows that
$$
\bar{a}_i   \ge 0    \quad  (1\le i\le d)  \qquad   {\rm and } \qquad
\bar{a}_1+\bar{a}_2+\cdots+\bar{a}_d=1.$$
 Minkowski's Convex Body  Theorem  \cite[Chp2: Theorem 2B]{Sch1980}, implies that  for  any $\mathbf{x}\in \I^d$ and $ \ell \in \N$, there exists non-zero integer point $(p_1,\cdots, p_d, q)\in \Z \times \mathbb{N}$ with $1 \le q\le q_\ell$ such that $$
|qx_i-p_i|<q_\ell^{-\bar{a}_i} \ \ \qquad  (1\le i\le d)  \, .
$$
Thus,  for any ball $B=B(r)\subset \I^d$  \begin{align}
\label{trew}
m_d(B)=m_d\left(B\cap \bigcup_{1\le q\le q_\ell}   \  \bigcup_{0\le p_1,\cdots, p_d \le q} \  \prod_{i=1}^dB\Big(\frac{p_i}{q}, \frac{1}{q\cdot q_\ell^{\bar{a}_i}}\Big)\right).
\end{align}

\noindent We now  split the above union into two  according to the size of the denominator $q$: $$
I_1:=\bigcup_{1 \le q< \frac{1}{M}\cdot q_\ell} \  \ \bigcup_{0\le p_1,\cdots, p_d \le q} \ \prod_{i=1}^dB\Big(\frac{p_i}{q}, \frac{1}{q\cdot q_\ell^{\bar{a}_i}}\Big) $$
and
$$ I_2:=\bigcup_{\frac{1}{M}q_\ell\le q\le q_\ell} \  \ \bigcup_{0\le p_1,\cdots, p_d\le q}\ \prod_{i=1}^dB\Big(\frac{p_i}{q}, \frac{1}{q\cdot q_\ell^{\bar{a}_i}}\Big)  \,  .
$$  It is easily verified  that for any ball $B=B(r)\subset \I^d$ and fixed $ q \in \N$
$$
 \# \left\{  0\le p_1,\cdots, p_d\le q  :  B\cap \prod_{i=1}^dB\Big(\frac{p_i}{q}, \frac{1}{q\cdot q_\ell^{\bar{a}_i}} \Big) { \neq \emptyset}   \right\}  \ \le  \ (2rq+2)^d  \, . $$
 Here we use the fact that rational points with denominator $q$ are separated by $1/q$.
 Thus,
\begin{align*}
  m_d(B\cap I_1)&\le \sum_{q<\frac{1}{M}\cdot q_\ell}\ \sum_{0\leq p_1,\cdots, p_d \le q}m_d\left(B\cap \prod_{i=1}^dB\Big(\frac{p_i}{q}, \frac{1}{q\cdot q_\ell^{\bar{a}_i}}\Big)\right)\\[2ex]
  &\le \sum_{q<\frac{1}{M}\cdot q_\ell}\ (2rq+2)^d \cdot \frac{2^d}{q^d\cdot q_\ell}=\sum_{q<\frac{1}{M}\cdot q_\ell}\ (r+{1 \over q})^d \cdot \frac{4^d}{ q_\ell}.
\end{align*}
Since $r+{1 \over q} \leq 2 \max\{r, {1 \over q}\}$ and by definition  $M= 4^{d+1}$,  it follows that for all large $\ell$
\begin{eqnarray*}
   m_d(B\cap I_1) &  < &  \sum_{q<\frac{1}{M}\cdot q_\ell}\left(\frac{8^d r^d}{q_\ell}+\frac{8^d}{q^d\cdot q_\ell}\right)\\[2ex]
  &\le & \frac{8^d\cdot r^d}{M}+ \frac{8^d\log q_\ell}{q_\ell} =   \frac{1}{4}m_d(B)+ \frac{8^d\log q_\ell}{q_\ell}  < \frac{1}{2}m_d(B).
\end{eqnarray*}
This, together with  \eqref{trew}  implies that $$
m_d(B\cap I_2)\ge m_d(B)-m_d(B\cap I_1)\ge \frac{1}{2}m_d(B).
$$ Finally, on noting that
\[
{1 \over q q_\ell^{\bar{a_i}}} \leq {M \over q_\ell^{a_i}}\leq {\big(\frac{\widetilde{M}}{q_\ell}\big)^{a_i}},
\]
we have that $
B\cap I_2\subset B\cap E_\ell
$
and so the desired statement follows.
\end{proof}

\section{Application to matrix transformation of tori
  }

  Let $T$ be a $d\times d$ non-singular matrix with real coefficients.
Then, $T$ determines a self-map of  the $d$-dimensional torus
 $X=\T^d:=\R^d/\Z^d$; namely, it sends $\vx\in\T^d$ to $T\vx$ modulo one. In what follows, $T$ will denote both the matrix and the transformation and for $n \in \N$,  by  $T^n$  we will always mean the $n$-th iteration of the transformation $T$ rather than the matrix multiplied $n$ times.

\medskip

  As usual, for $ 1 \le i \le d$, let $\psi_i: \R^+ \to \R^+$ be a real positive function and for convenience, let $\Psi := (\psi_1,\dots,\psi_d)$  and for $n \in \N$ let $\Psi(n) := (\psi_1(n),\dots,\psi_d(n))$.  Fix some point $\vz:= (z_1, \ldots,  z_d)  \in \T^d$ and for $n \in \N$, let
$$
 R\big(\vz, \Psi(n) \big):= \Big\{ \vx  \in \mathbb{T}^d :  \|x_i-z_i\| \le \psi_i(n)   \ (1\le   i \le   d) \Big\}.
$$
Clearly, $R\big(\vz, \Psi(n) \big)$ is  a rectangle centred at the fixed point $\vz$.
In turn, let
$$
W(T,\Psi,\vz):=\big\{\vx\in\mathbb{T}^d: T^n(\vx)\in R\big(\vz, \Psi(n) \big) \ \ \text{for infinitely many}\ n\in\mathbb{N}\big\}  \ . $$
For obvious reasons the sets $R_n:= R\big(\vz, \Psi(n) \big)$  can be thought of as rectangular targets that the orbit under $T$ of points in $X$ have to hit. The interesting situation is when the diameters of $R_n$ tend to zero as $n$ increases and thus it is natural to refer to $ W(T,\Psi,\vz)$ as the corresponding shrinking target set.
It is  easily verified  that if
$$
T = {\rm diag} \, (\beta_1, \ldots, \beta_d  )     \qquad   (\beta_i \in \R)  \,\vspace*{-3ex}
$$
then
\begin{equation}  \label{yesyes}
W(T,\Psi,\vz)=\big\{\vx\in\mathbb{T}^d: |T_{\beta_i}^nx_i-z_i| \le \psi_i(n)  \ (1\le   i \le   d) \ \ \text{for infinitely many}\ n\in\mathbb{N}\big\}  \, ,
\end{equation}
where $T_{\beta_i}$ is the  standard $\beta$-transformation   with $\beta=\beta_i$.
 Recall, given a real number $\beta$  such that  $|\beta|>1$,  the  associated $\beta$-transformation $T_{\beta}: [0,1)\to [0,1)$ is  given by $$T_{\beta}(x):=\beta x \ (\text{mod}\ 1).$$
The $\beta$-transformation admits an absolutely continuous invariant measure $\mu_\beta$, which is called Parry measure for positive $\beta$ and Yrrap measure for negative $\beta$. Let  $K(\beta)$ be the support of $\mu_\beta$ which are finite union of intervals. See \cite[Section 3.3]{LLVZstand} for some more details on $\mu_\beta$ and $K(\beta)$.

The set $W(T,\Psi,\vz)$ is non-trivial only when $\vz\in K:=\prod_{i=1}^d K(\beta_i)$. A straightforward consequence of \cite[Theorem~4]{LLVZstand} is that if  all the  eigenvalues $\beta_1,\beta_2,\dots,\beta_d$ are of absolute value strictly larger than $1$, then for any $\vz\in K$,
we have
\begin{eqnarray*}
		m_d|_K\big(W(T,\Psi,\vz)\big)=
		\begin{cases}
			0 &\text{if}\ \  \sum_{n=1}^\infty   \psi_1(n) \times \dots \times \psi_d(n) <\infty\\[2ex]
			1 &\text{if}\ \ \sum_{n=1}^\infty  \psi_1(n) \times \dots \times \psi_d(n) =\infty \, ,
		\end{cases}
	\end{eqnarray*}
where $m_d|K$ is the $d$-dimensional Lebesgue measure restricted on $K$.  This thus provide a complete description of the size of the shrinking target set $W(T,\Psi,\vz)$ in terms of the ambient measure.

Turning to the problem of determining its size in terms of Hausdorff dimension, it turns out that $ \dim_{\rm H} W(T,\Psi,\vz) $ is dependent on the set $\UP$ of accumulation points $ \mathbf{t}=(t_      1,t_2, \ldots, t_d ) $ of the sequence $$\Big\{\Big(\frac{-\log\psi_1(n)}{n},\cdots,
\frac{-\log\psi_d(n)}{n}\Big)\Big\}_{n\ge   1}  \, . $$
Indeed, the following statement  is in line with that claimed in \cite[Section~5.3]{LLVZstand}.   Under the assumption that $\UP$ is bounded,  it follows from results in \cite{HL, LLVZstand} -- see Remark~\ref{uhj} below.    In view of Theorem~\ref{pointcase}, we are able to remove the `bounded'  assumption and obtain a complete result.  As in the main body of the paper, for $\mathbf{t}=(t_1, \ldots,t_d)\in (\mathbb{R}^+\cup\{\infty\})^d$,  let $$\mathcal{L}(\mathbf{t}):=\{1\le   i \le   d: t_i<+\infty\}    \, .$$

%

\begin{theorem}\label{rectangledimresult}
Let $T$ be a real, non-singular matrix transformation of the torus $\mathbb{T}^d$. Suppose that $T$ is diagonal
and all its eigenvalues  $\beta_1,\beta_2,\dots,\beta_d$  are of modulus strictly larger than~$1$.
For $ 1 \le i \le d$, let $\psi_i: \mathbb{R}^+\to \mathbb{R}^+$ be a real positive non-increasing
		function and $\vz \in K$.  Then
		$$\dim_{\rm H} W(T,\Psi, \vz)=\sup_{\mathbf{t}\in\UP}\min \big\{\min_{i\in \mathcal{L}(\mathbf{t})}\{{\xi}_i(\mathbf{t})\}, \ \#\mathcal{L}(\mathbf{t})\big\}  \, ,$$
		where
		\begin{equation*}
			\xi_i(\mathbf{t}):=
			\sum_{k\in \mathcal{K}_1(i)}1+\sum_{k\in \mathcal{K}_2(i)}\left(1-\frac{t_k}{\log|\beta_i|+t_i}\right) +\sum_{k\in \mathcal{K}_3(i)}\frac{\log|\beta_k|}{\log|\beta_i|+t_i}
		\end{equation*}
and, in turn
			\begin{equation*}
			\mathcal{K}_1(i):= \{1\le   k\le   d: \log|\beta_k|>\log|\beta_i|+t_i\}, \ \mathcal{K}_2(i):= \{k\in\mathcal{L}(\mathbf{t}): \log|\beta_k|+t_k\le   \log|\beta_i|+t_i\}, \
			\end{equation*}
and
\begin{equation*}
  \mathcal{K}_3(i):=\{1, \dots, d\}\setminus (\mathcal{K}_1(i)\cup \mathcal{K}_2(i)).
\end{equation*}
	\end{theorem}

\medskip

\begin{remark} \label{uhj}
The above theorem was established in
\cite[Theorem~12]{LLVZstand}  under the assumption that $\UP$ is bounded  and all eigenvalues are strictly large than one.  In the case of  positive eigenvalues, the non-increasing assumption on the approximation functions  $\psi_i $  is not required. We mention, so as to avoid confusion,  that in the statement of \cite[Theorem~12]{LLVZstand} it is assumed that $
1<\beta_1\le \beta_2\le\cdots\le \beta_d $ but this is not required.
 Furthermore, in \cite[Section~5.2]{LLVZstand} the proof of the lower bound  for $\dim_{\rm H} W(T,\Psi, \vz) $  was generalised  to incorporate negative eigenvalues and this is where the non-increasing assumption on the approximation functions  comes into play.   Subsequently,  staying within the bounded setup,  the  corresponding  upper bound for the dimension was established in \cite{HL}. Indeed, under the assumption that $\UP$ is bounded, our theorem corresponds to the statement of Theorem~1.7 together with Remark 9 in \cite{HL} or equivalently Claim~1 in \cite[Section~5.2]{LLVZstand}.
\end{remark}

\medskip



\begin{proof}[Proof of Theorem~\ref{rectangledimresult}]
As is usual when proving  dimension statements, we establishing  the upper and lower bounds  for $ \dim_{\rm H} W(T,\Psi,\vz) $ separately.  In both cases, we will freely draw upon the frameworks and results used in \cite{HL, LLVZstand}  that deal with the situation when $\UP$ is bounded - see Remark~\ref{uhj} above.   With this in mind, the proofs below are best read in conjunction with the relevant cited aspects of these works.

%

\medskip

{\it Upper bound}.  The goal is to show   that
$$\dim_{\rm H} W(T,\Psi, \vz) \le \sup_{\mathbf{t}\in\UP}\min \big\{\min_{i\in \mathcal{L}(\mathbf{t})}\{{\xi}_i(\mathbf{t})\}, \ \#\mathcal{L}(\mathbf{t})\big\}  \, . $$
We first observe that when $\UP$ is bounded, by definition  we have that $ \#\mathcal{L}(\mathbf{t}) = d$ and so  the desired statement corresponds to \cite[Proposition~4.1]{HL}.   For the proof see \cite[Section~4.1]{HL} and  note that it does not require the non-increasing assumption on the approximating functions.   For the sake of completeness, we mention that in the case the eigenvalues are positive the statement  corresponds to \cite[Proposition~4]{LLVZstand} and is proved in \cite[Section~4.2.1]{LLVZstand}.
Now to deal with the situation that  $\UP$ is unbounded, given that the  `bounded' result exists, we use the same arguments as in the proof of Proposition~\ref{upperbound-unbdd-d} in \S\ref{upperbound-unbdd-d} with obvious notational modification.

\medskip

{\it Lower bound}.   We now turn our attention to establishing the  complementary lower bound for $ \dim_{\rm H} W(T,\Psi, \vz) $. As mentioned in Remark~\ref{uhj}, under the assumption that $\UP$ is bounded, the desired lower bound was established  in
\cite[Section~5.2]{LLVZstand}.  The proof given there built upon the proof of Proposition~5 in \cite[Section~4.2.1]{LLVZstand}  in which the eigenvalues are assumed to be strictly greater than one.

To deal with the unbounded situation we adapt the proof in  \cite[Section~5.2]{LLVZstand}. In short, the  key  change is that we apply  the unbounded Mass Transference Principle (Theorem~\ref{pointcase}) in place of the bounded Mass Transference Principle (appearing as Theorem~11 in \cite{LLVZstand}).  To illustrate how to implement this change, for ease of  clarity suppose that the eigenvalues  are strictly greater than one. Then, the  following statement is the upshot of formulae (53) and (54) in \cite[\S4.2.1]{LLVZstand} and is  at the heart of establishing Proposition~5 in \cite[Section~4.2.1]{LLVZstand}.  It has nothing to do with $\UP$ being  bounded.

\begin{lemma}\label{lem3-1}
Under the setting of Theorem \ref{rectangledimresult}, {if in addition  the eigenvalues are positive}, we have that
   \begin{equation}\label{limsupsubset}
 W(T,\Psi,\vz)\supset\limsup_{n\to\infty}\bigcup_{j_1=1}^{M_{1,n}} \cdots\bigcup_{j_d=1}^{M_{d,n}} B\Big(x_{n,j_1}^{{(1)}},\beta_1^{-n}\psi_1(n)\Big)\times\cdots\times B\Big(x_{n,j_d}^{(d)},\beta_d^{-n}\psi_d(n)\Big) \, ,
 \end{equation}
 and for each $1\leq i\leq d$
  \begin{equation}  \label{meadded}
 \T  \ = \  \bigcup_{j_i=1}^{M_{i,n}} B\big(x_{n, j_i}^{(i)}, (n+3)\beta_i^{-n}\big) \, .
 \end{equation}
Here, for each $1\leq i\leq d$, the centres  $\{x_{n,j_i}^{(i)}, 1\le   j_i\le   M_{i,n}\}$ are the preimages of $z_i$ under $T_{\beta_i}^n$ that fall within full cylinders  of order $n$ for $T_{\beta_i}$ and  $M_{i,n}$ is the number of such full cylinders.
\end{lemma}



We now show that the desired lower bound (in the case the  eigenvalues  are strictly greater than one) is a consequence of Lemma~\ref{lem3-1} and Theorem~\ref{pointcase}.  So with this in mind, let $ \mathbf{t}=(t_1,t_2, \ldots, t_d )\in \UP $. Then, by definition
there exists an integer subsequence $\{n_{\ell}\}_{\ell\in \mathbb{N}}$ such that
$$
\lim_{\ell\to\infty} {-\log \psi_i(n_\ell) \over n_\ell} = t_i\qquad \forall \ 1\le i\le d \, .  $$
In turn,  consider the  following $\limsup$ set
 \begin{equation*}
 W=\limsup_{\ell\to\infty}\bigcup_{j_1=1}^{M_{1,n_\ell}} \cdots\bigcup_{j_d=1}^{M_{d,n_\ell}} B\Big(x_{n_\ell,j_1}^{{(1)}},\beta_1^{-n_\ell}\psi_1(n_\ell)\Big)\times\cdots\times B\Big(x_{n_\ell,j_d}^{(d)},\beta_d^{-n_\ell}\psi_d(n_\ell)\Big) \,.
 \end{equation*}
In view of \eqref{limsupsubset},
$$ W \subset W(T,\Psi,\vz)   \,  . $$
With  Theorem \ref{pointcase} in mind,
for $\ell\in\mathbb{N}$, let
\[
r_\ell :=e^{-n_\ell}
\]
and for any $\epsilon>0$ and  $1\leq i \leq d$, let
\[
\delta_i=1,  \,  \qquad  u_i=\log \beta_i-\epsilon, \qquad v_{i, \ell}= \log\beta_i + {-\log \psi_i(n_\ell) \over n_\ell}   \, . 
\]
Then, for $1\leq i \leq d$ and $ \ell \in \N$ we have that
\[
r_\ell^{v_{i,\ell}}=\beta_i^{-n_\ell}\psi_i(n_\ell)
\]
and
\[
v_i:=\lim_{\ell\to\infty}v_{i, \ell}=\log \beta_i+t_i \, .
\]
Furthermore, with
\[s_{i,\ell}:=\frac{u_i}{v_{i,\ell}} \, , \] it follows that for  $\ell$  sufficiently large
\begin{equation}\label{biggerB}
(\beta_i^{-n_\ell}\psi_i(n_\ell))^{s_{i,\ell}}=r_\ell^{u_i}=\beta_i^{-n_\ell}e^{n_\ell \epsilon}>(n_\ell+3)\beta_i^{-n_\ell}.
\end{equation}
The upshot is that
 if
$$
 B_{i,\ell} := B\big(x_{n_\ell,j_i}^{{(i)}},\beta_i^{-n_\ell}\psi_i(n_\ell)\big)
$$
is a ball appearing in the definition of $W$, then by \eqref{biggerB}, for   $\ell$  sufficiently large
$$
B_{i,\ell}^{s_{i,\ell}}  \supset \textstyle{B\Big(x_{n_\ell,j_i}^{{(i)}},(n_\ell+3)\beta_i^{-n_\ell}\Big)}  \qquad \forall \ 1\le i\le d \, .   $$
Thus, in view of \eqref{meadded}, given the $\limsup $ set $W$ arising from  balls $B_{i,\ell}$ the corresponding   $\limsup$ set comprising of `enlarged' balls $B_{i,\ell}^{s_{i,\ell}}$ is equal to $\T^d$ and thus is  of  full $d$-dimensional Lebesgue measure.
Therefore, in view of  Theorem \ref{pointcase} we conclude that
\begin{equation} \label{poll1} \dim_{\rm H}
W(T,\Psi,\vz) \ge \ \dim_{\rm H} W \geq \   \min\left\{\min_{i\in \mathcal{L}(\mathbf{v})}s(\mathbf{u}{(\epsilon)} \,\mathbf{v},i), \ \#\mathcal{L}(\mathbf{v})\right\}
\end{equation}
with  $\mathbf{u}{(\epsilon)}=(\log \beta_1-\epsilon,\dots, \log \beta_d-\epsilon) $ and   $\mathbf{v}=  (\log \beta_1, \dots, \log \beta_d) + \mathbf{t}$.
It is evident that $\mathcal{L}(\mathbf{v}) =\mathcal{L}(\mathbf{t})$ and  that for all $1\leq i \leq d$
		\begin{equation*}
			s(\mathbf{u}{(\epsilon)},\mathbf{v},i)=\xi_i(\mathbf{t}, \epsilon):=
			\sum_{k\in \mathcal{K}_1(i)}1+\sum_{k\in \mathcal{K}_2(i)}\left(1-\frac{t_k}{\log \beta_i-\epsilon+t_i}\right) +\sum_{k\in \mathcal{K}_3(i)}\frac{\log \beta_k-\epsilon}{\log \beta_i-\epsilon+t_i}
		\end{equation*}
where
			\begin{equation*}
			\mathcal{K}_1(i):= \{1\le   k\le   d: \log \beta_k>\log \beta_i+t_i\}, \ \mathcal{K}_2(i):= \{k\in\mathcal{L}(\mathbf{t}): \log \beta_k+t_k\le   \log \beta_i+t_i\}, \
			\end{equation*}
and
\begin{equation*}
  \mathcal{K}_3(i):=\{1, \dots, d\}\setminus (\mathcal{K}_1(i)\cup \mathcal{K}_2(i)).
\end{equation*}
However, $\epsilon>0$ is arbitrary and so
\[
\dim_{\rm H}
W(T,\Psi,\vz) \ge  \   \min\left\{\min_{i\in \mathcal{L}(\mathbf{t})}\xi_i(\mathbf{t}), \ \#\mathcal{L}(\mathbf{t})\right\}.
\]
This is valid for any $ \mathbf{t}\in \UP $ and so we obtain the desired  lower bound when  the  eigenvalues  are strictly greater than one.

\medskip

We now  turn  our attention to establishing the lower bound for $ \dim_{\rm H} W(T,\Psi, \vz) $ when the eigenvalues  of $T$ are  negative.  For this we  follow the argument in  \cite[Section~5.2]{LLVZstand} and start by  exploiting  the framework of Markov subsystems to obtain the subset $W^*(T,\Psi,\vz)$  of $W(T,\Psi,\vz)$ under the assumption that the absolute values of the eigenvalues are $> 8$.  Then for this subset, we have an analogue of  Lemma~\ref{lem3-1},  without the assumption that the eigenvalues are positive, in which \eqref{limsupsubset} and \eqref{meadded} are respectively replaced by (69) and the displayed formula immediately following (69) in \cite[Section~5.2]{LLVZstand}.    As in the case of positive eigenvalues,  this has nothing to do with $\UP$ being bounded.
By the same arguments as in the above proof for positive eigenvalues, we  make use of Theorem \ref{pointcase} to obtain a lower bound for  $\dim_{\rm H} W^*(T,\Psi,\vz)$ (and hence $\dim_{\rm H} W^(T,\Psi,\vz)$) for the case that $\UP$ is unbounded.  In order to remove the assumption that the absolute values of all the eigenvalues of $T$ are $> 8$, we use the same approximating process employed at the end of \cite[Section~5.2]{LLVZstand}.

%
%

\end{proof}

\section{Properties of the  dimensional number} \label{Subsec:dim-number}

When it comes to proving Theorem~\ref{pointcase}, it will be more convenient to  work with a slightly more general formulation  of the dimensional number than that introduced at the start of \S\ref{MTP}. As before,  let  $p$ be a positive integer and let $\boldsymbol{\delta}=(\delta_1,\dots, \delta_p)\in (\mathbb{R}^+)^p $ be given.  Also,  for any
 $\mathbf{u}=(u_1,\dots, u_p)\in (\mathbb{R}^+)^p $ and
$\mathbf{v}=(v_1, \ldots,v_p)\in (\mathbb{R}^+\cup\{+\infty\})^p$  satisfying   \eqref{111},  let $\mathcal{L}(\mathbf{v})$ and $\mathcal{L}_\infty(\mathbf{v})$ be as in
\eqref{112} and let
$$
 \mathcal{A}=\{u_i : 1\le i\le p\}  \cup \{v_i: 1\le i\le p\}.
 $$
 For  each $A\in \mathcal{A}$, if
\begin{itemize}
  \item $A<+\infty$, we  let
\begin{equation} \label{def_suv}
	{s}(\mathbf{u}, \mathbf{v},A):=
	\sum_{k\in \mathcal{K}_1(A)}\delta_k+\sum_{k\in \mathcal{K}_2(A)}\delta_k\left(1-\frac{v_k-u_k}{A}\right) +\sum_{k\in \mathcal{K}_3(A)}\frac{\delta_ku_k}{A}
\end{equation}
where
\begin{equation*}
	\mathcal{K}_1(A):= \{1\le   k\le   p: u_k> A\}, \qquad  \mathcal{K}_2(A):= \{k\in\mathcal{L}(\mathbf{v}): v_k \le  A\}, \
\end{equation*}
 and $
	\mathcal{K}_3(A):=\{1, \dots, p\}\setminus (\mathcal{K}_1(A)\cup \mathcal{K}_2(A))$ .

\item $A=+\infty$, we let
$$
{s}(\mathbf{u}, \mathbf{v},A):=\sum_{k\in\mathcal{L}(\mathbf{v})}\delta_k.
$$
\end{itemize}
When $A=v_i$, we write $s(\bold{u},\bold{v}, i)$ for  $s(\bold{u},\bold{v}, A)$ and when $A=u_i$ we write $\overline{s}(\bold{u},\bold{v}, i)$ for  $s(\bold{u},\bold{v}, A)$.  Note that in the latter case, since $\mathbf{u}\in (\mathbb{R}^+)^p $,  we are always in the situation $ A < + \infty$. Formally, for every $i=1,\dots,p$ we define
\[
\begin{aligned}
	s(\bold{u},\bold{v}, i)&:=s(\bold{u},\bold{v}, v_i),\\[2ex]
	\overline{s}(\bold{u},\bold{v}, i)&:=s(\bold{u},\bold{v}, u_i).
\end{aligned}
\]
Clearly, the quantity  $s(\bold{u},\bold{v}, i)$ is precisely the same as that appearing in  \S\ref{MTP} and we recall that the dimensional number $s_0(\mathbf{u},\mathbf{v})$ with respect to $\mathbf{u}$ and $\mathbf{v}$  is defined as 
\begin{equation*} 
s_0(\mathbf{u},\mathbf{v})   =    \min_{1 \le i \le p }  \{ {s}(\mathbf{u}, \mathbf{v},i) \}    =  \min \Big\{\min_{i\in \mathcal{L}(\mathbf{v})}\{ {s}(\mathbf{u}, \mathbf{v},i) \}, \  \sum_{i\in\mathcal{L}(\mathbf{v})}\delta_i\Big\}.
\end{equation*}
The main difference in the above presentation to that in  \S\ref{MTP} is that it allows us to bring into play the quantity $\overline{s}(\mathbf{u}, \mathbf{v}, i)$.   This is important since it  will naturally appear in the proof of Theorem~\ref{pointcase}.  Indeed, one initially ends up with  showing  that
\begin{eqnarray*}
			\dim_{\rm H}\left(\limsup_{n\to\infty} {\prod_{i=1}^p}B_{i,n} \right)   & \ge   &   \min \Big\{   \min_{  A \in \mathcal{A} }  \{ {s}(\mathbf{u}, \mathbf{v},A) \},  \sum_{1 \le i \le p } \delta_i  \Big\}  \\[2ex]  & =  &
\min \Big\{\min_{i\in \mathcal{L}(\mathbf{v})}\{ {s}(\mathbf{u}, \mathbf{v},i) \}, \   \min_{1\leq i\leq p}\{ \overline{s}(\mathbf{u}, \mathbf{v},i) \},\  \sum_{i\in\mathcal{L}(\mathbf{v})}\delta_i\Big\}  \, .
	\end{eqnarray*}	
However, the following statement  shows that the   middle term on the right hand side involving $\overline{s}(\mathbf{u}, \mathbf{v}, i)$ plays no role and so the right hand side can be simply replaced by the dimensional number $s_0(\mathbf{u},\mathbf{v})$.

\begin{proposition}\label{Prop:4}
 For any $\mathbf{u}=(u_1,\cdots, u_p)\in (\mathbb{R}^+)^p$
 and $\mathbf{v}=(v_1, \ldots,v_p)\in (\mathbb{R}^+\cup\{+\infty\})^p$ satisfying \eqref{111}, we have
 \begin{align}
 \min_{1\leq i\leq p}\{ \overline{s}(\mathbf{u}, \mathbf{v},i) \}   \ge
 s_0(\mathbf{u},\mathbf{v}) .
 \end{align}
In particular,
 \begin{align}\label{dim-numb}
 s_0(\mathbf{u},\mathbf{v})=\min \Big\{\min_{i\in \mathcal{L}(\mathbf{v})}\{ {s}(\mathbf{u}, \mathbf{v},i) \}, \   \min_{1\leq i\leq p}\{ \overline{s}(\mathbf{u}, \mathbf{v},i) \},\  \sum_{i\in\mathcal{L}(\mathbf{v})}\delta_i\Big\}.
 \end{align}
\end{proposition}

The  above proposition can be viewed as the unbounded version of \cite[Proposition 3.1]{WW2021}.

\bigskip

\begin{remark}\label{rem:7}
It is useful to note that  the definitions of $\K_1(A)$, $\K_2(A)$ and $\K_3(A)$ can be modified to either
$$
\left\{
  \begin{array}{ll}
    \widetilde{\K}_1(A):=\K_1(A)\cup \{ 1\le k\le p : u_k=A\}, & \hbox{} \\[1ex]
    \widetilde{\K}_2(A):=\K_2(A)\setminus \{k\in \mathcal{L}(\bold{v}): v_k=A\}, & \hbox{} \\[1ex]
    \widetilde{\K}_3(A):=\{1, \dots, p\}\setminus (\widetilde{\mathcal{K}}_1(A)\cup \widetilde{\mathcal{K}}_2(A)), & \hbox{}
  \end{array}
\right. {\text{\!\!\!or}}\quad  \left\{
  \begin{array}{ll}
     \widehat{\K}_1(A):=\K_1(A), & \hbox{} \\[1ex]
    \widehat{\K}_2(A):= \widetilde{\K}_2(A), & \hbox{} \\[1ex]
    \widehat{\K}_3(A):=\{1, \dots, p\}\setminus (\widehat{\K}_1(A)\cup \widehat{\K}_2(A)), & \hbox{}
  \end{array}
\right.\
 $$   without effecting the outcome.   More precisely,  if $\widetilde{s}(\bold{u},\bold{v}, A)$ and $\widehat{s}(\bold{u},\bold{v}, A)$ denote the corresponding modified  analogues of ${s}(\mathbf{u}, \mathbf{v},A) $  associated with $\{ \widetilde{\K}_i(A) : 1 \le  i \le 3 \}  $  and  $ \{ \widehat{\K}_i(A) : 1 \le  i \le 3 \}  $ respectively,  then
 \begin{equation}\label{hjg}
   \widetilde{s}(\bold{u},\bold{v}, A)=  {s}(\bold{u},\bold{v}, A)    \qquad {\rm and } \qquad \widehat{s}(\bold{u},\bold{v}, A)=  {s}(\bold{u},\bold{v}, A)   \, .
 \end{equation}
Indeed, to see that this is the case, first note that
 $$
\begin{aligned}
\widetilde{\K}_3(A)&=\K_3(A)\cup \{k\in \mathcal{L}(\bold{v}): v_k=A\}\setminus\{k\in \mathcal{L}(\bold{v}): u_k=A\}.
\end{aligned}
$$
Now when $A=\infty$,   it is easily seen that
\[
\{k\in \mathcal{L}(\bold{v}): v_k=A\}=\{1\le k\le p: u_k=A\}=\emptyset
\]
and  so  $ \widetilde{\K}_i(A)=\K_i(A) $ for $i=1,2,3$ and   we are done. For the case  $A < \infty$,  it follows that
\begin{align*}
\widetilde{s}(\bold{u},\bold{v}, A)  \ & =  \   \sum_{k\in \mathcal{K}_1(A)}\delta_k
+\sum_{1\le k\le p: u_k=A}\delta_k
+\sum_{k\in \mathcal{K}_2(A)}\delta_k\left(1-\frac{v_k-u_k}{A}\right)
\\[2ex]
&  \qquad \qquad - \quad \sum_{k\in \mathcal{L}(\bold{v}): v_k=A}\delta_k\left(1-\frac{v_k-u_k}{A}\right)
 +\sum_{k\in \mathcal{K}_3(A)}\frac{\delta_ku_k}{A}
 \\[2ex]
&  \qquad  \qquad
 + \quad \sum_{k\in \mathcal{L}(\bold{v}): v_k=A}\frac{\delta_ku_k}{A}-\sum_{1\le k\le p: u_k=A}\frac{\delta_ku_k}{A}\\[3ex]
 &=\sum_{k\in \mathcal{K}_1(A)}\delta_k+\sum_{k\in \mathcal{K}_2(A)}\delta_k\left(1-\frac{v_k-u_k}{A}\right)
 +\sum_{k\in \mathcal{K}_3(A)}\frac{\delta_ku_k}{A} \ = \ s(\bold{u},\bold{v}, A).
\end{align*}
This establishes  the first of the equalities  in  \eqref{hjg}.  The proof can be modified in the obvious manner to prove the other equality in  \eqref{hjg} and we  leave the details to the reader.
\end{remark}

\medskip

The next result  concerns the continuity of the dimensional number  $s_0(\bold{u}, \bold{v})$.  Apart from  playing   a role in establishing Theorem~\ref{pointcase}, it has already been exploited  within the context of Remark~\ref{111_to_222S} in \S\ref{MTP}.

 \begin{proposition}\label{ll1} For any fixed $\mathbf{u}=(u_1,\cdots, u_p)\in (\mathbb{R}^+)^p$,
the quantities  $s(\bold{u},\bold{v},i)$ and  $\overline{s}(\bold{u},\bold{v},i)$
are continuous with respect to $\bold{v}\in \prod_{i=1}^p[u_i, +\infty]$ in the following  sense: let $\{\bold{v}(n) \}_{n \in \N}$ be a sequence of vectors $\bold{v}(n)=   (v_{1,n},\cdots, v_{p,n})  \in (\mathbb{R}^+)^p$
with  $u_i\le v_{i,n}$ for all $1\le i\le p$ and  $n$ sufficiently large,  and converging to
 $\bold{v} \in (\mathbb{R}^+\cup\{+\infty\})^p$, then
 \begin{equation}\label{x1}
\liminf_{n\to\infty}\; {s}(\bold{u}, \bold{v}{(n)},i) \ \ge \  {s}(\bold{u}, \bold{v},i), \qquad
\liminf_{n\to\infty} \; \overline{s}(\bold{u}, \bold{v}{(n)},i) \ = \  \overline{s}(\bold{u}, \bold{v},i),
 \end{equation}
 and
 \begin{equation}\label{x2}
 \lim_{n\to\infty}\min\Big\{s(\bold{u}, \bold{v}{(n)}, i): 1\le i\le p\Big\} \ = \ \min\Big\{s(\bold{u}, \bold{v}, i): 1\le i \le p\Big\}.
 \end{equation}
 \end{proposition}

\medskip

We now prove the above propositions.

\subsection{Proof of Propositions~\ref{Prop:4} \& \ref{ll1}} \label{Sec:App}

\begin{proof}[Proof of Proposition \ref{Prop:4}] When the vector $\mathbf{v}\in (\mathbb{R}^+\cup\{+\infty\})^p$ is finite, the statement corresponds to  \cite[Proposition 3.1]{WW2021}. Thus, in what follows we assume that we are in the infinite case.

Fix $\epsilon>0$. We choose $v^*=v^*(\epsilon)>0$ large enough so that
\begin{equation}\label{largeenough2}
\left| \sum_{k\in \mathcal{L}(\mathbf{v})}\delta_k\left(-\frac{v_k-u_k}{v^*}\right)+\sum_{k\in \mathcal{L}_\infty(\mathbf{v})}\frac{\delta_ku_k}{v^*}\right|<\epsilon   \,
\end{equation}
and
	\begin{equation}\label{largeenough}
v^*>\max\left\{\max_{1\le k\le d}u_k, \max_{k\in\mathcal{L}(\mathbf{v})}v_k\right\} \, .
	\end{equation}
Note that by definition $ \mathbf{u}$ is bounded and so the right hand side in \eqref{largeenough} is clearly  finite.
\noindent In turn, consider  the vector $\mathbf{v}^*=(v_1^*, \dots, v_p^*)$  where
	$$v_i^*:=v_i \  \ \text{for}\ \  i\in \mathcal{L}(\mathbf{v}) \qquad \text{and} \qquad v_i^*:=v^*\  \ \text{for}\ \  i\in \mathcal{L}_{\infty}(\mathbf{v}).$$
Note that  by definition $\mathcal{L}(\mathbf{v}^*)=\{1,2,\cdots, p\}$ and that if $v_k^*<v_i^*$, then  $k\in \mathcal{L}(\bold{v})$.

For each  $1\le i\le p$, we now compare ${s}(\mathbf{u}, \mathbf{v}^*, i)$ with ${s}(\mathbf{u}, \mathbf{v}, i)$  and $\overline{s}(\mathbf{u}, \mathbf{v}^*, i)$ with $\overline{s}(\mathbf{u}, \mathbf{v}, i)$.  With Remark \ref{rem:7} in mind, without loss of generality,  we assume that  these quantities for $\bold{u}$ and  $\mathbf{v}$ are defined with
$$
\K_1(A)=\{1\le k\le p: u_k>A\}, \ \ \
\K_2(A)=\{k\in \mathcal{L}(\bold{v}): v_k<A\} $$
and $
	\mathcal{K}_3(A):=\{1, \dots, p\}\setminus (\mathcal{K}_1(A)\cup \mathcal{K}_2(A))$. Also, in order to make a distinction,  we use $\mathcal{G}$ to take the role of $\mathcal{K}$ in the above when considering the corresponding  quantities for $\bold{u}$ and  $ \mathbf{v}^*$.

\medskip

\begin{itemize}
  \item {\em Comparing ${s}(\mathbf{u}, \mathbf{v}^*, i)$ with ${s}(\mathbf{u}, \mathbf{v}, i)$ for  $1\le i\le p$.} We consider two cases.
\begin{itemize}
  \item[] Case 1: Suppose  $i\in \mathcal{L}(\mathbf{v})$. Then, since $v_i^*=v_i$, it immediately follows that \begin{align*}
  \mathcal{G}_1(v_i^*)&:=\Big\{1\le k\le p: u_k> v_i^*\Big\}=\Big\{1\le k\le p: u_k > v_i\Big\}=\mathcal{K}_1(v_i),\\[2ex]
  \mathcal{G}_2(v_i^*)&:=\Big\{k\in\mathcal{L}(\mathbf{v}^*): v_k^* < v_i^*\Big\}=\Big\{k\in\mathcal{L}(\mathbf{v}): v_k< v_i\Big\}=\mathcal{K}_2(v_i)\end{align*}
  and thus \begin{equation}\label{1}
{s}(\mathbf{u}, \mathbf{v}^*, i)={s}(\mathbf{u}, \mathbf{v}, i) \ \quad \forall \ \ i\in \mathcal{L}(\mathbf{v}).
  \end{equation}

  \item[] Case 2: Suppose  $i\in \mathcal{L}_{\infty}(\mathbf{v})$. Then, $v_i^*=v^*$, and by \eqref{largeenough}, it follows  that \begin{align*}
  \mathcal{G}_1(v_i^*)&:=\Big\{1\le k\le p: u_k > v_i^*\Big\}=\emptyset,\\[2ex]
  \mathcal{G}_2(v_i^*)&:=\Big\{k\in\mathcal{L}(\mathbf{v}^*): v_k^*< v_i^*\Big\}=\mathcal{L}(\mathbf{v}) \end{align*}
  and thus  $$
  {s}(\mathbf{u}, \mathbf{v}^*, i)=\sum_{k\in \mathcal{L}(\mathbf{v})}\delta_k\left(1-\frac{v_k-u_k}{v^*}\right)+\sum_{k\in \mathcal{L}_\infty(\mathbf{v})}\frac{\delta_ku_k}{v^*}   \, .
  $$
Therefore, by (\ref{largeenough2}), we have that
 $$
  \left|  {s}(\mathbf{u}, \mathbf{v}^*, i)-\sum_{k\in \mathcal{L}(\mathbf{v})}\delta_k\right|<\epsilon.
  $$
Since the quantities  ${s}(\mathbf{u}, \mathbf{v}^*, i)$ for $i\in\mathcal{L}_{\infty}(\mathbf{v})$ are equal, we have that
   \begin{equation}\label{2}
  \left|\min_{i\in \mathcal{L}_{\infty}(\mathbf{v})}  {s}(\mathbf{u}, \mathbf{v}^*, i)-\sum_{k\in \mathcal{L}(\mathbf{v})}\delta_k\right|<\epsilon.
  \end{equation}
\end{itemize}

\medskip

\item {\em Comparing $\overline{s}(\mathbf{u}, \mathbf{v}^*, i)$ with $\overline{s}(\mathbf{u}, \mathbf{v}, i)$ for  $1\le i\le p$.} It is evident that
 $$
  {\mathcal{G}}_1(u_i)={\mathcal{K}}_1(u_i)   \quad {\rm and }  \quad {\mathcal{G}}_2(u_i)={\mathcal{K}}_2(u_i)  \qquad  (1 \le i \le p) \, .
  $$ Thus, it immediately follows that
  \begin{equation}\label{3}
  \overline{s}(\mathbf{u}, \mathbf{v}^*, i)=\overline{s}(\mathbf{u}, \mathbf{v}, i)  \ \quad \forall \ \ 1 \le i \le p
  \end{equation}

\end{itemize}

\medskip

 \noindent We continue the proof of the proposition by making use of  the above comparisons. Let $\ddot{s}_0(\mathbf{u}, \mathbf{v})$ denote the right hand side of \eqref{dim-numb}.  Then, by (\ref{1}) and (\ref{3}), we have that
 \begin{align*}
   \begin{split}
 \ddot{s}_0(\mathbf{u}, \mathbf{v})= \min \Big\{\min_{i\in \mathcal{L}(\mathbf{v})}\{{s}(\mathbf{u}, \mathbf{v}^*, i), \ \min_{{1\le i\le p}}\{\overline{s}(\mathbf{u}, \mathbf{v}^*, i)\},\ \sum_{i\in \mathcal{L}(\mathbf{v})}\delta_i\Big\}.  \end{split}
   \end{align*}
   Thus, by \eqref{2}, it follows that
    \begin{align}\label{eq:1}
   \begin{split}
    \ddot{s}_0(\mathbf{u}, \mathbf{v})  \ \geq  &  \ \min\Big\{\min_{i\in \mathcal{L}(\mathbf{v})}\{{s}(\mathbf{u}, \mathbf{v}^*, i)\}, \ \min_{{1\le i\le p}}\{\overline{s}(\mathbf{u}, \mathbf{v}^*, i)\},\ \min_{i\in \mathcal{L}_{\infty}(\mathbf{v})}{s}(\mathbf{u}, \mathbf{v}^*, i)\Big\}-\epsilon \\[2ex]
  =& \ \min\Big\{\min_{1\leq i\leq p}\{{s}(\mathbf{u}, \mathbf{v}^*, i)\}, \ \min_{{1\le i\le p}}\{\overline{s}(\mathbf{u}, \mathbf{v}^*, i)\}\Big\}-\epsilon.
  \end{split}
   \end{align}

  \noindent On the other hand, by \eqref{2} and \eqref{1}, it follows that
  \begin{align}\label{eq:2}
  \begin{split}
 \min_{1\leq i\leq p}\{{s}(\mathbf{u}, \mathbf{v}^*, i)\}\} \
  =& \  \min\big\{\min_{i\in \mathcal{L}(\mathbf{v})}\{{s}(\mathbf{u}, \mathbf{v}^*, i)\}, \  \min_{i\in \mathcal{L}_{\infty}(\mathbf{v})}{s}(\mathbf{u}, \mathbf{v}^*, i)\big\}\\[2ex]
  \ge & \ \min\big\{\min_{i\in \mathcal{L}(\mathbf{v})}\{{s}(\mathbf{u}, \mathbf{v}^*, i)\}, \  \sum_{k\in \mathcal{L}(\mathbf{v})}\delta_k\big\}-\epsilon\\[2ex]
  =& \  \min\big\{\min_{i\in \mathcal{L}(\mathbf{v})}\{{s}(\mathbf{u}, \mathbf{v}, i)\}, \  \sum_{k\in \mathcal{L}(\mathbf{v})}\delta_k\big\}-\epsilon\\[2ex]
  =& \ {s}_0(\mathbf{u},\mathbf{v})-\epsilon.
  \end{split}
 \end{align}
However, in the bounded case we know by  Wang \& Wu  \cite[Proposition 3.1]{WW2021} that
  \begin{align*}
\min\Big\{\min_{1\leq i\leq p}\{{s}(\mathbf{u}, \mathbf{v}^*, i)\}, \ \min_{{1\le i\le p}}\{\overline{s}(\mathbf{u}, \mathbf{v}^*, i)\}\Big\}=\min_{1\leq i\leq p}\{{s}(\mathbf{u}, \mathbf{v}^*, i)\}.
   \end{align*}
   Thus, the upshot of \eqref{eq:1} and \eqref{eq:2} is that
 \[
     \ddot{s}_0(\mathbf{u}, \mathbf{v}) \geq {s}_0(\mathbf{u},\mathbf{v})-2\epsilon.
 \]
 Now  $\epsilon > 0$ is arbitrary and so
  \[
     \ddot{s}_0(\mathbf{u}, \mathbf{v}) \geq {s}_0(\mathbf{u},\mathbf{v}).
 \]
 This together with the fact that the opposite inequality is immediate from the definition of the quantities under consideration,   completes the proof of the proposition.
\end{proof}

\medskip

\begin{proof}[Proof of Proposition \ref{ll1}]
Recall that for $\bold{u}\in (\mathbb{R}^+)^p$ and $\bold{v}\in \prod_{i=1}^p[u_i, +\infty]$, we set $$ \mathcal{L}(\bold{v}):=\{1\le k\le p: v_k<\infty\} \quad {\text{and}} \quad
s(\bold{u}, \bold{v}, i):=\sum_{k\in \mathcal{L}(\bold{v})}\delta_k  \  \ {\text{for}}\ \  i\not\in \mathcal{L}(\bold{v}).
$$
\noindent \emph{Step 1.}  Let $i\in \mathcal{L}(\bold{v})$. Then, if $v_k< v_i$ we have $k\in \mathcal{L}(\bold{v})$. Note also that for any $n\in\mathbb{N}$, $\mathcal{L}(\bold{v}(n))=\{1,2,\cdots, p\}$.  With Remark \ref{rem:7} in mind, without loss of generality,  we take $$
\ \ \K_1(v_i)=\{1\le k\le p: u_k>v_i\},\ \ \ \ \ \ \ \K_2(v_i)=\{k\in\mathcal{L}(\bold{v}): v_k< v_i\}
$$
and  $$
 \ \ \ \ \ \   \K_1(v_{i,n})=\{1\le k\le p: u_k> v_{i,n}\},\ \ \  \K_2(v_{i,n})=\{k\in\mathcal{L}(\bold{v}): v_{k,n}< v_{i,n}\}.
  $$
  Note that if $u_k>v_i$, then $u_k>v_{i,n}$ for all $n$ large and  if $v_k<v_i$, then $v_{k,n}<v_{i,n}$ for all $n$ large. Thus, it follows that for $n$ large enough $$
  \K_1(v_i)\subset \K_1(v_{i,n}), \quad \text{and} \quad \K_2(v_i)\subset \K_2(v_{i,n}).
  $$
Next, for $n$ large enough define
\[
\mathcal{P}_1= \mathcal{P}_1(i,n):=\K_1(v_{i,n})\setminus\K_1(v_i) \quad \text{and} \quad \mathcal{P}_2= \mathcal{P}_2(i,n):=\K_2(v_{i,n})\setminus\K_2(v_i).
\]
%
Then, $ s(\bold{u}, \bold{v}(n), i)$  can be rewritten as
\begin{align}
\label{poll}
  s(\bold{u}, \bold{v}(n), i)
  = & \sum_{k\in \K_1(v_i)}\delta_k+\sum_{k\in \K_2(v_i)}\delta_k+\sum_{k\in \mathcal{P}_1\cup \mathcal{P}_2}\delta_k
  +\frac{1}{v_{i,n}}\bigg[\sum_{k\in \K_3(v_{i})}u_k\delta_k-\sum_{k\in \mathcal{P}_1\cup \mathcal{P}_2}u_k\delta_k \nonumber \\[2ex]
   & \hspace*{10ex} - \sum_{k\in \K_2(v_i)}(v_{k,n}-u_k)\delta_k-\sum_{k\in \mathcal{P}_2}(v_{k,n}-u_k)\delta_k\bigg] \nonumber \\[2ex]
  = &  \ s(\bold{u}, \bold{v}, i)  \ + \ \sum_{k\in \mathcal{P}_1}\delta_k+\sum_{k\in \mathcal{P}_2}\delta_k-\frac{1}{v_{i,n}}\bigg[\sum_{k\in \mathcal{P}_1}u_{k}\delta_k \, + \, \sum_{k\in \mathcal{P}_2}v_{k,n}\delta_k\bigg].
\end{align}
Also note that by definition
\begin{itemize}
  \item for $k\in \mathcal{P}_1$,
  $$
  u_k>v_{i,n}  \quad  {\rm and }   \quad u_k\le v_i,\ {\text{hence}}\ \ \frac{u_k}{v_{i,n}}=1+o(1)  \quad  {\rm as} \  n \to \infty \ ;
  $$
  \item for $k\in \mathcal{P}_2$, $$
  v_{k,n}< v_{i,n} \quad  {\rm and }   \quad v_k\ge v_i,\ {\text{hence}}\ \ \frac{v_{k,n}}{v_{i,n}}=1+o(1) \quad  {\rm as} \  n \to \infty \ .
  $$
\end{itemize}
The upshot of this together with \eqref{poll} is that as $ n \to \infty$,  we have that  \begin{equation}\label{6}
s(\bold{u}, \bold{v}(n), i)=s(\bold{u}, \bold{v}, i)+o(1)  \qquad \forall \ \ i\in \mathcal{L}(\bf{v}).
\end{equation}

\medskip

\noindent \emph{Step 2.}  Let $i\in \mathcal{L}_{\infty}(\mathbf{v})$. 
  Note that for $n$ sufficiently large  $$
\K_1(v_i)=\{1\le k\le p: u_k>v_i\}=\emptyset,\ \ \ \ \K_2(v_i)=\{k\in \mathcal{L}(\bold{v}): v_k< v_i\}=\mathcal{L}(v),
$$
and
  $$
  \K_1(v_{i,n})=\{1\le k\le p: u_k> v_{i,n}\}=\emptyset,\ \ \K_2(v_{i,n})=\{1\le k\le p: v_{k,n}< v_{i,n}\}\supset \mathcal{L}(\mathbf{v}).
  $$ Then \begin{align} \label{dragon}
    s(\bold{u}, \bold{v}(n), i)&=\sum_{k\in \K_2(v_{i,n})}\delta_k-\sum_{k\in \K_2(v_{i,n})}\frac{\delta_k(v_{k,n}-u_k)}{v_{i,n}}+\sum_{k\in \K_3(v_{i,n})}\frac{\delta_k u_k}{v_{i,n}}.\end{align}
    It is clear that the last term is $o(1)$ since $v_{i,n}\to \infty$ as $n \to \infty$. Moreover,
    \begin{align*}
    s(\bold{u}, \bold{v}(n), i)&=\sum_{k\in \mathcal{L}(\mathbf{v})}\delta_k-\sum_{k\in \mathcal{L}(\mathbf{v})}\frac{\delta_k(v_{k,n}-u_k)}{v_{i,n}} \  \ +  \! \sum_{k\in \K_2(v_{i,n})\setminus \mathcal{L}(\mathbf{v})} \! \delta_k\left(1-\frac{v_{k,n}-u_k}{v_{i,n}}\right)+o(1).
    \end{align*}
  Note that in the above formula, the second summation is $o(1)$ since the numerator is finite and the denominator tends to infinity as $n$ increases; and the third term is positive since for $k\in \K_2(v_{i,n})\setminus \mathcal{L}(\mathbf{v})$, we have that  $v_{k,n}-u_k< v_{k,n}\leq  v_{i,n}$. Thus, we conclude that
     \begin{align}\label{7}
     s(\bold{u}, \bold{v}(n), i)
    & \ \ge \  \sum_{k\in \mathcal{L}(\mathbf{v})}\delta_k-o(1)  \ \qquad \forall   \ \ i\in \mathcal{L}_{\infty}(\mathbf{v}).
  \end{align}
   This establishes  the first assertion in (\ref{x1}). The second assertion can be proved  with the same arguments and is in fact easier since when dealing with $ \overline{s}(\bold{u}, \bold{v},i)$ each component $u_i \ (1 \le i \le d)$ of the vector $ \bold{u} $ is bounded.  We omit the proof.

\medskip

\noindent \emph{Step 3.}  Regarding the limit statement (\ref{x2}), in view of  (\ref{6}) and (\ref{7}), it suffices to show that
\begin{equation}
\label{dragsv}
  \min_{i\in \mathcal{L}_{\infty}(\mathbf{v})}s(\bold{u}, \bold{v}(n), i)\le \sum_{k\in \mathcal{L}(\mathbf{v})}\delta_k+o(1)  \, .
  \end{equation}
  For each $n\ge 1$, let $v_{i^*,n}$ be the smallest element among $\{v_{i,n}: i\in \mathcal{L}_{\infty}(\mathbf{v})\}.
  $ Then, it follows that for $n$ large enough  $$\K_1(v_{i^*,n})=\emptyset, \ \ \K_2(v_{i^*,n})=\mathcal{L}(\mathbf{v}).$$ In turn, this together with \eqref{dragon} implies that  $$
  s(\bold{u}, \bold{v}(n), i^*)=\sum_{k\in \mathcal{L}(\mathbf{v})}\delta_k-\sum_{k\in \mathcal{L}(\mathbf{v})}\frac{\delta_k(v_{k,n}-u_k)}{v_{i^*,n}}+\sum_{k\in \mathcal{L}_{\infty}(\mathbf{v})}\frac{\delta_k u_k}{v_{i^*,n}} .
  $$
  This together with the fact that
  $$
  \min_{i\in \mathcal{L}_{\infty}(\mathbf{v})}s(\bold{u}, \bold{v}(n), i)\le s(\bold{u}, \bold{v}(n), i^*)
  $$
  and that $v_{i^*,n}\to \infty$ as $n \to \infty$ implies the desired inequality \eqref{dragsv}.
\end{proof}

\section{Proof of Theorem~\ref{pointcase}}

The overall strategy for establishing  Theorem \ref{pointcase} is simple enough. For any $\epsilon>0$,  we will use the classical method of constructing a Cantor type subset $\mathbb{E}_\infty(\epsilon)$ of  the set $ \limsup_{n\to\infty} {\prod_{i=1}^p}B_{i,n}
$ under consideration  and a measure $\nu_\epsilon$ supported on this subset. The measure will satisfy the  H\"older exponent  estimate:
\begin{equation} \label{ub_general_ball}
\nu_\epsilon(B(x,r))   \; \le \;    c \,   r^{s_0(\bold{u},\bold{v})-2\varepsilon}   \, \quad {\rm for \ all \ } r < r_0  \, , \ x \in \mathbb{E}_{\infty}(\epsilon) \, ,
\end{equation}
where $c, r_0  $  are  positive constants and
$s_0(\bold{u},\bold{v})$ is the dimension number defined by~\eqref{dimensionalnumber}.
This implies via the standard Mass Distribution Principle (see \cite[\S4.2]{F}) that
\begin{equation}  \label{desiredCantordim}
\dim_{\rm H}  \mathbb{E}_\infty(\epsilon)
   \; \ge \;   s_0(\mathbf{u},\mathbf{v})-2\varepsilon.
\end{equation}
In turn, this together with the fact that
\[
\mathbb{E}_\infty(\epsilon) \;   \subset \; \limsup_{n\to\infty} {\prod_{i=1}^p}B_{i,n}   \, ,
\]
implies that
\[
\dim_{\rm H}\left(\limsup_{n\to\infty} {\prod_{i=1}^p}B_{i,n} \right)\ge  s_0(\mathbf{u},\mathbf{v})-2\varepsilon.
\]
However, $\ep >0$ can be made arbitrarily small, whence~\eqref{pointcase_result} follows and thereby completes the proof of Theorem~\ref{pointcase}.

\medskip

We start the process of carrying out the above strategy with some preliminaries that will help simplify the construction of the desired Cantor type subset $\mathbb{E}_\infty(\epsilon)$.   Throughout, the Vinogradov symbols $\ll$ and
$\gg$ will be used to indicate an inequality with an unspecified
positive multiplicative constant.  If $a \ll b $ and $ a \gg b $, we
write $a \asymp b $ and say that the two quantities $a$ and $b$ are
comparable.   Also, for any $
\kappa>0$, given a ball $B=B(x,r)$  or a rectangle $R:=\prod_{i=1}^p B(x_i, r_i)$ we denote by $\kappa B$  (respectively,  $\kappa R$)  the  ball $B$  (respectively, the rectangle $R$)  scaled by a factor $\lambda$;  that is
\begin{equation} \label{def_standard_dilatation}
  \kappa B:=B(x, \kappa r)  \quad  {\rm and }  \quad \kappa R:=\prod_{i=1}^p B(x_i, \kappa r_i)  \, .
 \end{equation}
In turn, we say that a family $\Gamma$ of balls (respectively, rectangles) is $\kappa r$-separated if for any two balls $B_1, B_2$ (respectively, two rectangles $R_1, R_2$), we have $\kappa B_1\cap \kappa B_2=\emptyset$ (respectively, $ \kappa R_1\cap  \kappa R_2=\emptyset$).

\subsection{Preliminaries for the proof}

\subsubsection{Imposing additional assumptions and structure. }\label{subsect:KGB}

In the course of establishing the theorem, without loss of generality, we can invoke Remark~\ref{WwwwwWsv} and assume that
\begin{equation} \label{unbounded_case}
\max_{i=1,\dots,p}v_i=+\infty.
\end{equation}
Also, in view of Remark~\ref{111_to_222S} we are able to replace condition ~\eqref{222S} imposed on $\mathbf{u}$ in the statement of the theorem by the weaker assumption~\eqref{111}.

Next, observe that~\eqref{222S}  together with condition~\eqref{cond:r-n} imposed in the statement of the theorem implies that there exists an integer $N>0$ such that, for all $n\in\N$ with $n\geq N$,  we have
\begin{equation} \label{v_in_bigger_than_u_i}
v_{i,n}>u_i  \quad  \ {\text{for each}}\ \ 1\le i\le p  \, .
\end{equation}
However, it is easily seen that the $\limsup$ set under consideration in Theorem~\ref{pointcase} is
 unaffected if we drop a finite number of indices $n$,. Thus, without loss of generality,   when establishing Theorem~\ref{pointcase},  we can assume  that \eqref{v_in_bigger_than_u_i} is valid for all $n\in\N$.   In the same spirit, given any $ \varepsilon > 0 $ it follows via  \eqref{cond:r-n},  \eqref{222S} and   Proposition~\ref{ll1} that for all $n$ sufficiently large
\begin{equation} \label{s0_is_continuous}
\left|s_0(\mathbf{u},\mathbf{v})-s_0(\mathbf{u},\mathbf{v}(n))\right|<\varepsilon   \,
\end{equation}
where $\mathbf{v}(n):=(v_{1,n},\dots,v_{p,n})$.   However, within the context of establishing Theorem~\ref{pointcase}  we can assume that \eqref{s0_is_continuous} is valid for all $n\in\N$.

Finally,  in establishing Theorem~\ref{pointcase} it will be convenient to  assume that the 
sequence of ``radii''   $\{ r_n  :  n \in \N  \} $  appearing  in the statement  are integer powers of 2. To this end, we slightly enlarge the original $r_n$ from the statement by taking
\[
\tilde{r}_n:=2^{\lfloor\log_2 r_n\rfloor+1}.
\]
Then, the full measure assumption \eqref{fullmeasure1} in Theorem~\ref{pointcase} is also satisfied by $\tilde{r}_n$.	
Replace $v_{i,n}$ by the values
\[
\tilde{v}_{i,n}:= v_{i,n} \cdot {\log r_n \over \log \tilde{r}_n},
\]
so that we have
\[
r_n^{v_{i,n}}=\tilde{r}_n^{\tilde{v}_{i,n}}.
\]
Then, as $r_n\to 0$ when $n\to\infty$, we have
\[
\lim_n\tilde{v}_{i,n}=\lim_n v_{i,n}=v_i.
\]
So, the modifications of $\{r_n\}$ and $\{v_{i,n}\}$ do not affect neither the hypothesis nor the conclusion of Theorem~\ref{pointcase}.

\medskip


\subsubsection{The $K_{G,B}$ covering lemma for rectangles. }\label{subsect:additions}

To obtain our unbounded Mass Transference Principle, we need an analog of the $K_{G,B}$ lemma that was very much at the heart of the proof of the original Mass Transference Principle \cite{BV2006}.   We begin by establishing a  covering result  that extends the standard $5r$-covering lemma for balls to  rectangles. It can be strengthened in various forms but the following is more than adequate for our purpose.

\begin{lemma}[Covering lemma for rectangles] \label{Vitaly_cover_rectangles}
For each $1 \le i \le p$, let
 $X_i$ be a locally compact subset of $\R^{d_i}$ and  let $\{B(x_{i,n},r_n)\}_{n \in \N}$ be a sequence of balls in $X_i$.
Let $(u_1,\dots,u_p)\in(\R^+)^p$ and let
\begin{equation} \label{Vitaly_cover_rectangles_F}
\mathcal{H}:=\left\{\prod_{i=1}^pB(x_{i,n},r_n^{u_i}) : n\in\mathcal{I} \right\}
\end{equation}
 be a collection of rectangles indexed by the set $\mathcal{I}$.  Assume that $\{r_n  :  n\in\mathcal{I}\}$ is uniformly bounded.
Then, there exists a subfamily $\mathcal{H}_1\subset\mathcal{H}$ such that $\mathcal{H}_1$ is at most countable, consists of disjoint rectangles, and
\begin{equation} \label{Vitaly_cover_rectangles_result}
\bigcup_{R\in\mathcal{H}}R\subset \bigcup_{R\in \mathcal{H}_1}5^{\frac{\max_{i=1,\dots, p} u_i}{\min_{i=1,\dots, p} u_i}}R.
\end{equation}
\end{lemma}

\begin{proof}
To make the notation easier, let  $u:=\min_{i=1,\dots,p}u_i$.   For $i=1,\dots,p$, let $\rho_i$ be Euclidean metrics on $X_i$ inherited from $\R^{d_i}$. Consider the following metric on  the product space $X$:
\begin{equation} \label{def_tilde_rho}
\tilde{\rho}:=\max_{i=1,\dots,p}\rho_i^{\frac{u}{u_i}}   \qquad {\rm where }  \quad   u:=\min_{i=1,\dots,p}u_i    .
\end{equation}
It can be verified  by straightforward computations that any rectangle
\begin{equation*} 
R=\prod_{i=1}^pB(x_{i,n},r_n^{u_i})
\end{equation*}
is a ball in the metric space $(X,\tilde{\rho})$ with centre $(x_{1,n},\dots,x_{p,n})$ and of radius $r_n^{u}$. So, we can apply  the $5r$-covering lemma~\cite[Theorem~2.1]{MAT} to the family $\mathcal{H}$ of balls  in the metric space $(X,\tilde{\rho})$. This gives rise to a subfamily of balls $\mathcal{H}_1\subset \mathcal{H}$ such that
\[
\bigcup_{R\in\mathcal{H}}R\subset \bigcup_{R\in \mathcal{H}_1}5\tilde{\cdot}R,
\]
where we denote by $5\tilde{\cdot}R$ the dilatation by factor $5$ of the ball $R$ in the metric space $(X,\tilde{\rho})$, that is the ball with the
same center as $R$ and the radius 5 times bigger than the radius of $R$.

Recall that any $R=\prod_{i=1}^pB(x_{i,n},r_n^{u_i})\in\mathcal{H}_1$ is a ball in $(X,\tilde{\rho})$ of the radius $r_n^u$. So $5\tilde{\cdot}R$ has radius $5r_n^u$. Thus, returning to the notation of the initial product space we get that $5\tilde{\cdot}R$ is nothing but
\[
\prod_{i=1}^pB(x_{i,n},5^{\frac{u_i}{u}}r_n^{u_i}).
\]
The latter set is clearly contained in $5^{\frac{\max_{i=1,\dots,p}u_i}{\min_{i=1,\dots,p}u_i}}R$ (where we now use the notation for dilatation $\kappa R$ defined in~\eqref{def_standard_dilatation}), and this completes the proof.
\end{proof}


 We are ready to establish the $K_{G,B}$ lemma for rectangles. Here and throughout,  $$ \mu := \mu_1 \times \cdots \times \mu_p$$ is the natural  product measure  in Theorem~\ref{pointcase}.

  \begin{lemma}[$K_{G,B}$ lemma for rectangles]\label{k2}
Let $\{B_{i,n}\}_{n \in \N}$ be a sequence of balls satisfying the assumptions in Theorem~\ref{pointcase}, and let $B$ be a ball in $X=\prod_{i=1}^pX_i$. For any $G\in \N$, there exists a finite
sub-collection $K_{G,B}$ of rectangles in the family
\begin{equation}\label{aa1}
\Gamma:=\Big\{\widetilde{R}=\prod_{i=1}^pB(x_{i,n}, r^{u_i}_n): n\ge G,\  \widetilde{R}\cap \textstyle{\frac12}B\ne \emptyset\Big\}
\end{equation} such that \begin{itemize}
\item the rectangles in $K_{G,B}$ are contained in ${2\over 3}B$;

\item the rectangles in $K_{G,B}$ are $5r$-separated; i.e. for  distinct  $\widetilde{R}$ and $\widetilde{R'}$ in $K_{G,B}$
\begin{equation} \label{KGB_property_5r_separation}
5\widetilde{R}\cap 5\widetilde{R'}=\emptyset;
\end{equation}

\item the rectangles in $K_{G,B}$ almost pack the ball $B$; i.e.
\begin{equation} \label{KGB_property_almost_full_measure}
\mu\left(\bigcup_{\widetilde{R}\in K_{G,B}}\widetilde{R}\right)\asymp \mu(B),
\end{equation}
where the implied constants are absolute.
\end{itemize}

\end{lemma}

\begin{proof}
Since, $r_n^{u_i}\to 0$ as $n\to\infty$   for all $1\leq i \leq p$, we can ensure that every rectangle in $\Gamma$ with $n$ large enough, is contained in ${2\over 3}B$. Hence, without loss of generality, we can assume that all rectangles in $\Gamma$ are contained in ${2\over 3}B$. Furthermore, the conditions of Lemma~\ref{Vitaly_cover_rectangles} are satisfied for the $5$-times scaled  rectangles in $\Gamma$.  So, applying Lemma~\ref{Vitaly_cover_rectangles}, we find a sub-collection $\Gamma_1\subset\Gamma$ such that all rectangles of $\Gamma_1$ are $5r$-separated and for which~\eqref{Vitaly_cover_rectangles_result} holds.

\noindent Recall the measure condition  (\ref{fullmeasure1}) in the statement of  Theorem \ref{pointcase} or  equivalently  \eqref{fullmeasure2}; that is,  for some $(u_1,\dots,u_p)\in(\mathbb{R}^+)^p$, the  $\limsup$ set $$
 \limsup_{n\to \infty}\prod_{i=1}^pB(x_{i,n}, r^{u_i}_n) \, ,
 $$
 is of full $\mu$-measure. Applying \eqref{fullmeasure1}, we have that the ball $\frac12 B$ is covered in measure by $\bigcup_{R\in \Gamma}R$. Hence,
\[
  \mu\left(\textstyle{\frac12} B\right)\le \mu\left(\bigcup_{\widetilde{R}\in \Gamma}\widetilde{R}\right)\le \mu\left(\bigcup_{\widetilde{R}\in \Gamma}5\widetilde{R}\right)\le \sum_{\widetilde{R}\in \Gamma_1}\mu\left(5^{\frac{\max_{i=1,\dots,p}u_i}{\min_{i=1,\dots,p}u_i}} (5\widetilde{R})\right)\asymp \sum_{\widetilde{R}\in \Gamma_1}\mu(\widetilde{R}).
\]
Thus,
\[
  \sum_{\widetilde{R}\in \Gamma_1}\mu(\widetilde{R})\gg  \mu(\textstyle{\frac12} B)\gg \mu(B).
 \]
As all rectangles in $\Gamma_1$ are disjoint and contained in ${2\over 3}B$, we have
\begin{equation} \label{finite_1}
\sum_{\widetilde{R}\in \Gamma_1}\mu(\widetilde{R})\ll \mu(B).
\end{equation}

If the family  $\Gamma_1$ is finite, then we can just take $K_{G,B}=\Gamma_1$ and the proof is complete. If $\Gamma_1$ is countable, proceed as follows.   First, we introduce a numbering of the  rectangles in $\Gamma_1$; namely
\[
\Gamma_1=\left\{\widetilde{R}_i:  i\in\N\right\}.
\]
Then, the upper bound~\eqref{finite_1} implies that the series $\sum_{\widetilde{R}\in \Gamma_1}\mu(\widetilde{R})$ is convergent. Hence, for $N$ sufficiently large we have that
\[
\sum_{i=1}^N \mu(\widetilde{R}_i) \geq \frac12 \sum_{i=1}^{\infty} \mu(\widetilde{R}_i)  \, .
\]
To complete the proof we take  $K_{G,B}=\{\widetilde{R}_i\in\Gamma_1\mid i=1,\dots,N\}$  and observe that
\[
\sum_{i=1}^N \mu(\widetilde{R}_i)\asymp \sum_{i=1}^{\infty} \mu(\widetilde{R}_i) \asymp \mu(B).
\]
\end{proof}

\bigskip

\subsubsection{Notions and sub-collections associated with the $K_{G,B}$ lemma.}\label{subsect:Notations-KGB}

 We now introduce various notions and sub-collections associated with the $K_{G,B}$ lemma that will be used in the proof of Theorem~\ref{pointcase}.  \textit{Throughout,   $\{B_{i,n}\}_{n \in \N}$  is a sequence of balls satisfying the assumptions in Theorem~\ref{pointcase} and  $B$ is an arbitrary  ball in $X=\prod_{i=1}^pX_i$.} With this in  mind, denote by $\mathcal{I}_{G,B}$  the finite collection of indices $n\in \mathbb{N}$ associated with the rectangles of the finite sub-collection $K_{G,B}$ coming via Lemma \ref{k2},  so that
\begin{equation}  \label{bnm}
K_{G,B}=\Big\{\widetilde{R}:=\prod_{i=1}^pB\Big(x_{i, n}, r_n^{u_i}\Big): n\in \mathcal{I}_{G,B}\Big\}.
\end{equation}
We remark that the set $\mathcal{I}_{G,B}$ is not uniquely defined since potentially  we have many choices of the family $K_{G,B}$. However, as we will see,  this will be made concrete when it comes to the construction of the desired Cantor set satisfying \eqref{desiredCantordim}.  Next,  in view of the discussion in the previous section, the radii $r_n$ appearing in \eqref{bnm}  will be assumed to be powers of two and as a consequence  we partition the indexing set $\mathcal{I}_{G,B}$ according to such powers, thus
\[
\mathcal{I}_{G,B}=\bigsqcup_{k\in \N}\mathcal{I}_{G,B,k},
\]
where
\begin{equation}\label{f2}
\mathcal{I}_{G, B, k}:=\Big\{n\in \mathcal{I}_{G,B}: r_n=2^{-k}\Big\}.
\end{equation}
Note that for all but finitely many $k$  the set $\mathcal{I}_{G, B, k}$ is empty. Now, for any $k\in\N$ for which $\mathcal{I}_{G, B, k} \neq \emptyset$  and for every $1\leq i\leq p$, define
\begin{equation}\label{def-v-i}
    v_i{(G, B, k)}:=\max_{n\in \mathcal{I}_{G, B, k}}v_{i,n} \, ,
\end{equation}
and in turn,  let
\[
S_{G, B,k}:=\left\{R=\prod_{i=1}^pB\Big(x_{i, n}, 2^{-k \, v_i{(G,B,k)}}\Big) : \ n\in \mathcal{I}_{G,B, k}\right\}
\]
and
\[
S_{G,B}:=\bigcup_{k\ge 1}S_{G,B,k}.
\]

\begin{remark}\label{rem-R-tilde}
We emphasise  that $K_{G,B}$ is a collection of ``big'' rectangles coming from Lemma~\ref{k2} and $S_{G,B}$   is a collection of ``small'' rectangles obtained by suitably shrinking the rectangles in $K_{G,B}$.    Note that  in view  of \eqref{v_in_bigger_than_u_i} we have  that $u_i <  v_i{(G,B,k)}$ and so the descriptors ``big'' and ``small'' are justified.  There is clearly a one-to-one correspondence between the rectangles in $K_{G,B}$ and  $S_{G,B}$; namely that
 $$
K_{G,B}\ni \widetilde{R}:=\prod_{i=1}^pB\Big(x_{i, n}, 2^{-k \, u_i}\Big)\longleftrightarrow R:=\prod_{i=1}^pB\Big(x_{i, n}, 2^{-k \,v_i{(G,B,k)}}\Big)\in S_{G,B}.
$$
The upshot is that inside  each rectangle $\widetilde{R}\in K_{G,B}$ there is precisely one shrunk rectangle  $R\in S_{G,B}$. 
\end{remark}

\medskip

With \eqref{def-v-i} in mind, for each  $k\in\mathbb{N}$ for which $\mathcal{I}_{G, B, k} \neq \emptyset$,  let 
\begin{equation}  \label{think}
    w(G,B,k):=\max_{1\leq i\leq p}v_i{(G, B, k)}  \, .
\end{equation} 
Then, it follows from the $5r$-covering lemma and basic  volume arguments that for each rectangle $R\in S_{G, B,k}$, we can find a finite collection $\mathcal{C}(R)$ of disjoint balls $B'$ all with the same  radius $$r:= 2^{-k \, w(G,B,k)} $$  such that they  are $5r$-separated and
\[
R\subset\bigsqcup_{B'\in\mathcal{C}(R)}25B'.
\]
By comparing the volume of the rectangle $R$ with the sum of  volumes of (disjoint) balls in $\mathcal{C}(R)$, we find that
    \begin{equation} \label{cardinality_of_CR}
\# \mathcal{C}(R)\asymp \prod_{i=1}^p\left(\frac{2^{-k \, v_i(G,B,k)\delta_i}}{2^{-k \, w(G,B,k)\delta_i}}\right)
\end{equation}
where the implied constants are absolute (most importantly they are  independent of the ball $B$).  
Also, note that in view of the additional assumption~\eqref{unbounded_case}, 
the quantity $w(G,B,k)$ increases as $G$ increases and so  
\begin{equation} \label{w_tends_to_infty}
\lim_{G\to\infty} \ \min_{k: \mathcal{I}_{G, B, k}\ne\emptyset}w(G,B,k)=\infty   \, . 
\end{equation}

\medskip

The following  ``separation'' statement is a simple consequence of the above definitions and construction.

\begin{proposition}\label{pp2}   
The balls in $\mathcal{F}_G(B):=\left\{B' \in \mathcal{C}(R): R\in S_{G, B} \right\} $ are $5r$-separated. In particular, if a given  ball $B_{\lambda} := B(x,\lambda)$ intersects at least two balls in $\mathcal{F}_G(B)$, then its radius $\lambda$ is larger than the radius of any ball in $\mathcal{F}_G(B)$ intersecting $B_{\lambda}$. In turn, the  balls in $\mathcal{F}_G(B)$  that intersect $B_\lambda$ are contained in $2 B_{\lambda}$.
\end{proposition}
\begin{proof}
The  ``in turn'' part  following directly from the ``in particular'' part and  this follows directly from the first part. Regarding the latter, let $B_1,B_2\in\mathcal{F}_G(B)$. 
Then, $B_1\in\mathcal{C}(R_1)$ and $B_2\in\mathcal{C}(R_2)$ for some $R_1,R_2\in S_{G,B}$.
If $R_1=R_2$, then $B_1$ and $B_2$ are $5r$-separated by the definition of $\mathcal{C}(R)$. In the case  $R_1\ne R_2$, let $\widetilde{R}_1$ and $\widetilde{R}_2$ respectively be the corresponding ``big'' rectangles in $K_{G,B}$. It is clear that $\widetilde{R}_1\ne \widetilde{R}_2$, so they are $5r$-separated by the definition of $K_{G,B}$ which in turn  implies the  $5r$-separation of $B_1$ and $B_2$. This completes the proof of the proposition.
\end{proof}

\subsection{Constructing the Cantor set $\mathbb{E}_\infty(\epsilon)$ and the measure $\nu_\epsilon$ }

We fix an $\varepsilon>0$, which should be understood  to be  an arbitrarily small number given in advance.  Recall,  given  a sequence $\{B_{i,n}\}_{n \in \N}$  of  balls in $X_i$  satisfying the  assumptions of Theorem~\ref{pointcase}, the overall goal is to 
 construct a Cantor type set  
 \[
\mathbb{E}_\infty(\epsilon) \;   \subset \; \limsup_{n\to\infty} {\prod_{i=1}^p}B_{i,n}   \, ,
\]
and a measure $\nu_\epsilon$ supported on  $\mathbb{E}_\infty(\epsilon)$ satisfying  the  H\"older exponent  estimate   \eqref{ub_general_ball}. 
Throughout,   for any Borel measure $\theta$ on $X= {\prod_{i=1}^p}X_i $ and  any measurable subset $F\subset X$ we use the standard notion  $\theta\restriction_F$ to  denote  the restriction of $\theta$ to $F$; that is 
\[
\theta\restriction_F(A)=\theta(A\cap F), \ \text{for any measurable subset} \  A \subset X.
\]
 Also, recall that $ \mu := \mu_1 \times \cdots \times \mu_p $  is the natural  product measure  in Theorem~\ref{pointcase}.   Finally, from this point onwards, since  $\epsilon > 0 $ is fixed,  in order to  simplify  notation we put
 $$
 \nu  = \nu_\epsilon  \qquad {\rm and}  \qquad \mathbb{E}_\infty=\mathbb{E}_\infty(\epsilon)  \, . 
 $$

Having established the  preliminaries  we  are  now in the position to construct $\mathbb{E}_\infty$ with the desired properties.  
The construction is inductive and we start by taking an arbitrary  closed  ball $B_0\subset X$ such that  $\mu(B_0)=1$  and we set
$$
\mathcal{E}_0=\{B_0\}   \, . $$ 
In turn, we let 
$$\nu_0:=\mu\!\restriction_{B_0}  \, .  $$ 
In other words, $\nu_0$ is simply the probability measure with support $B_0$ that has a constant density on $B_0$ with respect to $\mu$.    Now since   $r(B_{i,n})\to  0 $ as $n\to\infty$,  we  have  that
 $r_n\to 0$ as $n\to \infty$  and so there exists an integer $G_0$ such that for all  $n\geq G_0$  \, . 
\begin{equation*} 
\frac{\nu_{0}({B_{0}})}{\mu({B_{0}})} \le r_{n}^{-\varepsilon}.
\end{equation*}
Furthermore, in view of \eqref{w_tends_to_infty} the integer $G_0$ 
can be chosen such that for every $G\geq G_0$ 
\begin{equation*} 
\min_{k: \mathcal{I}_{G, B_0, k}\ne\emptyset}w(G,B_0,k)>1.
\end{equation*}
Let
\[\mathbb{E}_0:=\bigcup_{{B}\in\mathcal{E}_0}{B} = B_0 \, .\]

\noindent To proceed with the induction, assume that we have constructed, for certain $j\in\N$,
\begin{itemize}

    \item the finite disjoint collection of balls $\mathcal{E}_{j-1}$  and thus the $(j-1)$-th level set
\begin{equation*} 
\mathbb{E}_{j-1} :=\bigcup_{{B}\in\mathcal{E}_{j-1}}{B}.
\end{equation*}
\item the probability measure $\nu_{j-1}$ supported on $ \mathbb{E}_{j-1}  $ satisfying 
\begin{equation*}\label{nu-sum-1}
    \sum_{B_{j-1}\in\mathcal{E}_{j-1}}\nu_{j-1}(B_{j-1})=1,
\end{equation*}
\item the sufficiently large integer $G_{j-1}$ so that for all ${B}_{j-1}\in\mathcal{E}_{j-1}$,
\begin{equation*} \label{w_bigger_than_one-B-j-1}
\min_{k: \mathcal{I}_{G_{j-1}, B_{j-1}, k}\ne\emptyset}w(G_{j-1},B_{j-1},k)>1,
\end{equation*}
and 
\begin{equation*} \label{ratio_nu_mu_is_small-j-1}
\frac{\nu_{j-1}({B_{j-1}})}{\mu({B_{j-1}})} \le r_{n}^{-\varepsilon}  \ \quad  \text{for all \  $n\ge G_{j-1}$}   \, . 
\end{equation*}
\end{itemize}

\noindent  Then,  with  the  notions and notation of \S\ref{subsect:Notations-KGB} in mind,  we let
\[
\mathcal{E}_{j}:=\bigcup_{{B_{j-1}}\in\mathcal{E}_{j-1}}\bigcup_{R\in S_{G_{j-1},B_{j-1}}}\mathcal{C}(R) 
\]
and  set  
\begin{equation*} \label{def_F_j}
\mathbb{E}_j :=\bigcup_{{B}\in\mathcal{E}_j}{B}.
\end{equation*}
By construction
$$
\mathbb{E}_{j}  \subset \mathbb{E}_{j-1}   \, . 
$$
Furthermore,  the balls in $\mathcal{E}_{j}$ are disjoint and so to define  the probability measure  $\nu_j$ supported on $\mathbb{E}_j$ it suffices to define  $\nu_j$ on each ball $B_j\in\mathcal{E}_{j}$.   
In order to do this, recall (again by construction)  that for each $B_j\in\mathcal{E}_{j}$, there exist a unique ball $B_{j-1}\in\mathcal{E}_{j-1}$ and a unique rectangle $R\in S_{G_{j-1},B_{j-1}}$ such that
\begin{equation} \label{double_inclusion}
B_j\subset R\subset B_{j-1}.
\end{equation}
Then, we define  
\begin{equation} \label{def_nu_j}
\nu_j\restriction_{B_j}:=\frac{1}{\#\mathcal{C}(R)}\cdot\frac{\mu(\widetilde{R})}{\sum_{R'\in S_{G_{j-1},B_{j-1}} }\mu(\widetilde{R'})}
\cdot \nu_{j-1}(B_{j-1})
\cdot \frac{\mu\restriction_{B_j}}{\mu(B_j)}.
\end{equation}
Hence,
\begin{equation} \label{link_nu_j_B_j_nujm1_Bjm1}
\nu_j(B_j)=\frac{1}{\#\mathcal{C}(R)}\cdot\frac{\mu(\widetilde{R})}{\sum_{R'\in S_{G_{j-1},B_{j-1}} }\mu(\widetilde{R'})}
\cdot \nu_{j-1}(B_{j-1})
\end{equation}
and it follows  that for any $B_{j-1}\in\mathcal{E}_{j-1}$
\begin{equation} \label{stabilization_mu_j_1}
\nu_{j}(B_{j-1})=\sum_{R\in S_{G_{j-1},B_{j-1}}}\sum_{B'\in\mathcal{C}(R)}\nu_j(B')=\nu_{j-1}(B_{j-1}).
\end{equation}
In turn, since  the  balls in $\mathcal{E}_j$ are disjoint, we have that 
\begin{equation*} \label{rty}
  \sum_{B_j\in\mathcal{E}_j, B_j\subset B_{j-1}} \nu_j(B_j)= \nu_{j}(B_{j-1})  \, ,
\end{equation*}
and so together  \eqref{nu-sum-1}   and \eqref{stabilization_mu_j_1} it follows that 
\begin{equation}
   \sum_{B_j\in\mathcal{E}_j} \nu_j(B_j)=\sum_{B_{j-1}\in\mathcal{E}_{j-1}}  \sum_{B_j\in\mathcal{E}_j, B_j\subset B_{j-1}} \nu_j(B_j)=\sum_{B_{j-1}\in\mathcal{E}_{j-1}} \nu_{j}(B_{j-1})=1.
\end{equation}
In other words,   $\nu_j $ is a probability measure supported on $\mathbb{E}_j$ as desired.

We now turn our attention to  choosing the integer $G_{j}$.  Recall that the collection of balls $\mathcal{E}_{j} $ is finite.  Thus, in view of  \eqref{w_tends_to_infty}, we can choose $G_{j}$ sufficiently large so that for all ${B}_{j}\in\mathcal{E}_{j}$ 
\begin{equation} \label{w_bigger_than_one-B-j}
\min_{k: \mathcal{I}_{G_j, B_{j}, k}\ne\emptyset}w(G_j,B_{j},k)>j\geq 1.
\end{equation}
Furthermore,  since $r_{n}\to 0$ as $n\to\infty$,   the integer $G_{j}$ can be chosen such that for all ${B}_{j}\in\mathcal{E}_{j}$ and all $n\ge G_{j}$ 
\begin{equation} \label{ratio_nu_mu_is_small}
\frac{\nu_{j}({B_{j}})}{\mu({B_{j}})} \le r_{n}^{-\varepsilon}.
\end{equation}
 This completes the inductive step and we define 
$$
\mathbb{E}_{\infty} \ := \ \bigcap_{j=1}^{\infty}  \mathbb{E}_j \ =  \ \bigcap_{j=1}^{\infty}\bigcup_{{B}\in \mathcal{E}_j}{B}   \, .
$$
By construction,  $\mathbb{E}_{\infty}$ is a Cantor subset of $\limsup\limits_{n\to\infty} {\prod_{i=1}^p}B_{i,n}$.

Regarding the construction of a  measure supported on $\mathbb{E}_{\infty}$, we first note that  on iterating~\eqref{stabilization_mu_j_1}, we find that for every $j,k\in\N\cup\{0\}$ and any ball $ B_{j} \in  \mathcal{E}_j $
\begin{equation} \label{stabilization_mu_j}
\nu_{j+k}(B_{j})=\nu_{j}(B_{j}).
\end{equation}
In particular, it follows that each  $\nu_j$ is a probability measure on $B_0$ and in turn that the weak-$*$ limit $\nu$ of the sequence $\{\nu_j\}_{j \in \N}$ exists  (see Remark~\ref{zzz} below).  By construction, $\nu$ is a probability measure supported on $E_\infty$ and  satisfies 
\begin{equation} \label{stabilization_nu}
\nu(B_{j})=\nu_{j}(B_{j})      \qquad \forall \ B_j \in \mathcal{E}_j   \quad {\rm and } \quad j \in \N\cup\{0\}   \, .  
\end{equation}
More generally, for any Borel subset $F$ of $X$  
\begin{equation*} \label{falc}
\nu(F):=\nu  (F \cap \mathbb{E}_{\infty} )   \  =  \    \inf\;\sum_{B\in\mathfrak{C}(F)} \nu(B)  \ ,
\end{equation*}
 where the infimum is taken over all coverings $\mathfrak{C} (F)$ of $F
\cap \mathbb{E}_{\infty} $ by balls  $B \in \bigcup_{j\in\N \cup \{0\}}  \ \mathcal{E}_j$.

\medskip

\begin{remark} \label{zzz} Although relatively standard and well know, for completeness we show that the sequence of measures $(\nu_j)_{j\in\mathbb{N}\cup\{0\}}$ has a weak-$*$ limit $\nu$. Let $f:B_0\to [-1,1]$ be a continuous function.  Since $B_0$ is closed, $f$ is uniformly continuous on $B_0$. Hence, for any $\tilde{\varepsilon}>0$, we can find $\delta>0$ such that, for every $x\in B_0$, we have
\[
\max_{y\in B(x,\delta)}|f(x)-f(y)|<\tilde{\varepsilon}.
\]
As the radii of balls in $\mathcal{E}_j$ tend to 0 as $j\to\infty$, we infer that there exists $N\in\mathbb{N}$ such that, for all $m\geq N$ and for all $B_m\in\mathcal{E}_m$, we have $r(B)<\delta$. So, for every $B_m\in\mathcal{E}_m$ and every $x\in B_m$,
\[
\left|\int_{B_m} f d \nu_m - \nu_m(B_m) f(x)\right|<\tilde{\varepsilon}\nu_m(B).
\]
Hence,  for all $n>m\geq N$, taking any $x\in B_m$, by \eqref{stabilization_mu_j} we have 
\begin{align*}
    \left|\int_{B_m} f d\nu_m-\int_{B_m} f d\nu_n\right|&=\left|\int_{B_m} f d\nu_m-\nu_m(B_m) f(x)+\nu_n(B_m) f(x)-\int_{B_m} f d\nu_n\right|\\[2ex]
    &\leq \left|\int_{B_m} f d\nu_m-\nu_m(B_m) f(x)\right|+\left|\nu_n(B_m) f(x)-\int_{B_m} f d\nu_n\right|\\[2ex]
    &<2\tilde{\varepsilon}\nu_m(B_m).
\end{align*}
Summing over all $B_m\in\mathcal{E}_m$, we obtain that
\[
\left|\int_{B_0} f d\nu_m-\int_{B_0} f d\nu_n\right|<2\tilde{\varepsilon}.
\]
Then, by Cauchy's criterion, there exists a limit
\[
\lim_{n\to\infty}\int_{B_0} f d\nu_n.
\]
Therefore, we have a positive linear functional $\phi$ on the set of continuous functions $f:B_0\to [-1,1]$ defined by
\begin{equation} \label{def_phi}
\phi(f):=\lim_{n\to\infty}\int_{B_0} f d\nu_n.
\end{equation}
By Riesz-Markov-Kakutani Theorem (see  \cite[Theorem~2.14]{Rudin}), there exists a probability measure $\nu$ which represents $\phi$, and hence, by~\eqref{def_phi}, $\nu$ is a weak-$*$ limit of $(\nu_j)_{j\in\mathbb{N}}$.
\end{remark}

\medskip 

To complete the proof of Theorem~\ref{pointcase}, it remains  to show that $\nu$  satisfies  the  H\"older exponent  estimate   \eqref{ub_general_ball}.   To start with, we prove the desired estimate for balls $B_j$  appearing in the construction of $\mathbb{E}_j$.

\subsubsection{Measure of balls in the  Cantor construction}

\label{ss_measure_of_ball_in_G_j}

Let $j\in\N$ and consider a ball $B_j\in\mathcal{E}_j$. Then by construction, there exist a unique ball $B_{j-1}\in\mathcal{E}_{j-1}$ and a unique rectangle $R\in S_{G_{j-1},B_{j-1}}$ such that $B_j\in \mathcal{C}(R)$.
It follows from~\eqref{link_nu_j_B_j_nujm1_Bjm1}
and~\eqref{stabilization_nu} that
\begin{equation} \label{link_nu_B_j_nu_Bjm1}
\nu(B_j)=\frac{1}{\#\mathcal{C}(R)}\cdot\frac{\mu(\widetilde{R})}{\sum_{R'\in S_{G_{j-1},B_{j-1}} }\mu(\widetilde{R'})}
\cdot \nu(B_{j-1}).
\end{equation}

Throughout this and the subsequent section,  we  will use the following  shorthand  to simplify the notation. 
Let   $B_{j-1}\in\mathcal{E}_{j-1}$ be given. Then for each  $ 1 \le i \le p$ and $k\in\N$
for which $\mathcal{I}_{G_{j-1}, B_{j-1}, k} \neq \emptyset$, we set 
\begin{equation} \label{shorthand_vw}
\begin{aligned}
v_i(j-1,k)&:=v_i(G_{j-1},B_{j-1},k),\\[2ex]
w(j-1,k) &:=w(G_{j-1},B_{j-1},k) .
\end{aligned}
\end{equation}
Thus, the definition of $w(G_{j-1},B_{j-1},k) $  (see  \eqref{think}) in the above notation reads
\begin{equation} \label{shorthand_vwmax}
w(j-1,k)=\max_{1 \le i  \le p} v_i(j-1,k).
\end{equation}

\medskip

Recall that by the definition of $\mathcal{C}(R)$ and Remark~\ref{rem-R-tilde}, for some $k\in\N$ and  some index $n\in \mathcal{I}_{G_{j-1},B_{j-1},k}$
\[
B_j=  B\left(\mathbf{z}, 2^{-k \, w(j-1,k)}\right), \ \quad
R=\prod_{i=1}^pB\left(x_{i,n}, 2^{-k \,  v_i(j-1,k)}\right), \ \quad \widetilde{R}=\prod_{i=1}^pB\left(x_{i,n}, 2^{-k \,  u_i}\right)  \, . 
\]

By the Ahlfors regularity assumption of the measures $\mu_i$,
we have that
\begin{equation} \label{size_mu_R}
\mu(\widetilde{R})\asymp\prod_{i=1}^p 2^{-k u_i\delta_i}.
\end{equation}
Further, it follows from the definition of $ S_{G_{j-1},B_{j-1}}$ that
\begin{equation} \label{mu_sum_rectangles_Bj}
\sum_{R'\in S_{G_{j-1},B_{j-1}} }\mu(\widetilde{R'})\asymp\mu(B_{j-1}).
\end{equation}

\noindent Then, by \eqref{cardinality_of_CR}, \eqref{link_nu_B_j_nu_Bjm1}, \eqref{size_mu_R} and~\eqref{mu_sum_rectangles_Bj}, it follows that
\begin{eqnarray}\label{ub_nu_j_B_j}
\nu(B_j)
&\asymp & \prod_{i=1}^p\left(\frac{2^{-k \, w(j-1,k)\delta_i}}{2^{-k \, v_i(j-1,k)\delta_i}}\right)\cdot\prod_{i=1}^p 2^{-k \, u_i\delta_i}
\cdot \frac{\nu(B_{j-1})}{\mu(B_{j-1})}
\nonumber \\[2ex]
&= &2^{-k \, w(j-1,k) \, s(j)} \cdot \frac{\nu(B_{j-1})}{\mu(B_{j-1})}  \, ,
\end{eqnarray}
 where $s(j)$ is given by
\[
s(j):=\sum_{i=1}^p\delta_i\left(1-\frac{v_i(j-1,k)-u_i}{w(j-1,k)}\right).
\]

\noindent Now note that by the  definition of $w(j-1,k)$ (see~\eqref{shorthand_vwmax}) and by~\eqref{w_tends_to_infty}, we have that 
$$
\frac{v_i(j-1,k)-u_i}{w(j-1,k)}\le 1   \ \ {\text{for all}}\ i\not\in \mathcal{L}(\bold{v}) \, , $$ \ {\text{and}}\ $$ \lim_{j\to \infty}\frac{v_i(j-1,k)}{w(j-1,k)}=0 \ \ {\text{for all}} \ i\in \mathcal{L}(\bold{v}).$$
Thus, it follows that for all $j$ sufficiently large 
$$
s(j)\ge \sum_{i\in \mathcal{L}(\bold{v})}\delta_i-\epsilon.
$$
Feeding this into~\eqref{ub_nu_j_B_j} together with~\eqref{ratio_nu_mu_is_small} and~\eqref{w_bigger_than_one-B-j}, implies that for all $j$ sufficiently large   
\begin{eqnarray*} 
\nu(B_j)\ll 2^{-k \, w(j-1,k)\cdot s(j)-\epsilon} & = & r(B_j)^{s(j)-\epsilon/w(j-1,k)} \nonumber \\[2ex]
&  \le & r(B_j)^{s(j)-\epsilon } \le r(B_j)^{s(\bold{v},\epsilon)}  \, ,
\end{eqnarray*}
where
\begin{equation*} \label{def_s}
s(\bold{v},\epsilon):=\sum_{i\in \mathcal{L}(\bold{v})}\delta_i-2\epsilon.
\end{equation*}
To obtain the desired estimate \eqref{ub_general_ball} for the ball $B_j$, we note that by the definition of $s_0(\bold{u},\bold{v}) $  (see \eqref{dimensionalnumber}) we have that 
\[
\sum_{i\in \mathcal{L}(\bold{v})}\delta_i\geq s_0(\bold{u},\bold{v}) 
\]
and so 
\begin{equation} \label{Holder-B_j}
\nu(B_j)  \ \ll \   r(B_j)^{s_0(\bold{u},\bold{v})-2\epsilon}    \, . 
\end{equation}

\medskip

\begin{remark}  
For later use, we note at this point that for a rectangle $R\in S_{G_{j-1},B_{j-1}}$ we can use \eqref{link_nu_B_j_nu_Bjm1} to estimate $\nu (R) $. Indeed, since    
the balls of $\mathcal{E}_j$ contained in $R$ have the same $\nu$-measure, it follows via~\eqref{link_nu_B_j_nu_Bjm1} and \eqref{mu_sum_rectangles_Bj} that
\begin{equation} \label{nu_widetilde_R}
\nu(\widetilde{R}) = \nu(R)=\#\mathcal{C}(R)\cdot\nu(B_j)\asymp \mu(\widetilde{R})\cdot\frac{\nu(B_{j-1})}{\mu(B_{j-1})}  \, . 
\end{equation}
The fact that 
$$\nu(\widetilde{R}) = \nu({R}) $$
immediately follows from the one to one correspondence of $R$ and $\widetilde{R}$  (see Remark~\ref{rem-R-tilde}) and  the construction of $\nu$; namely that it is defined via   balls $B_j$ contained in rectangles $R\in S_{G_{j-1},B_{j-1}}$.  
\end{remark}

\subsubsection{Measure of a general ball.} \label{ss_measure of a general ball}

Throughout, let $$ B:= B( \mathbf{x}, r) $$
be an arbitrary ball in $X$ with centre $\mathbf{x} \in \mathbb{E}_{\infty}$ and radius $r>0$.    Since $\mathbf{x}$ is in the Cantor set, it follows that there exists an integer  $j\in\N$ such that $ B( \mathbf{x}, r)$ intersects only one ball $B_{j-1}$ in $\mathcal{E}_{j-1}$ and at least two balls in $\mathcal{E}_{j}$. In view of the latter,  by  Proposition~\ref{pp2} we have that  $r$ is larger than the radius of any ball $B_j $ in $\mathcal{E}_j$  intersecting $ B( \mathbf{x}, r)$; that is 
\begin{equation}   \label{r_is_big}
    r \, > \, \max \Big\{  r(B_{j})   \, : \,  B_{j} \in \ \mathcal{E}_{j}  \ \   {\rm with }  \ \   B_{j} \cap  B( \mathbf{x}, r)  \neq \emptyset \Big\} \, . 
\end{equation}
In this section, we will continue to use the shorthand notation~\eqref{shorthand_vw} introduced in the previous section. Recall, that the goal is to show that $\nu( B( \mathbf{x}, r)) $ satisfies the  H\"older exponent  estimate \eqref{ub_general_ball}.   First, we deal with the case $r> r({B_{j-1}})$. By \eqref{Holder-B_j},
$$
\nu( B( \mathbf{x}, r)) \le \nu(B_{j-1}) \ll r(B_j)^{s_0(\bold{u},\bold{v})-2\epsilon} \le r^{s_0(\bold{u},\bold{v})-2\epsilon}  .  \, 
$$
Hence, the ball $ B( \mathbf{x}, r)$ satisfies  the desired  estimate~\eqref{ub_general_ball}.  Without loss of generality, we assume that 
\begin{equation} \label{r_leq_r_jm1}
r\le r({B_{j-1}}).
\end{equation}
It follows from the construction  that  the balls in $\mathcal{E}_j$ that intersect $ B( \mathbf{x}, r)$ are contained in
$\{B_j\in \mathcal{C}(R): R\in  S_{G_{j-1}, B_{j-1}}\}$. 
Let
\[
\mathcal{R}:=\left\{R\in S_{G_{j-1},B_{j-1}}: \mathcal{C}(R) \text{ contains a ball}
\text{ intersecting }  B( \mathbf{x}, r)\right\}
\]
and let
\[
\widetilde{\mathcal{R}}:=\left\{\widetilde{R}: R\in\mathcal{R}\right\}  \, . 
\]
It follows that 
\begin{equation}   \label{bnmk}
  \nu\Big( B( \mathbf{x}, r) \Big) =   \nu\Big( B( \mathbf{x}, r)\cap \bigcup_{\widetilde{R}\in\widetilde{\mathcal{R}}}\widetilde{R}\Big)
\end{equation}

Recall (see Remark~\ref{rem-R-tilde}),  each rectangle $\widetilde{R}\in\widetilde{\mathcal{R}}$ is of the form 
\begin{equation*}
\widetilde{R}=\prod_{i=1}^pB(x_{i,n}, r_n^{u_i})    \,  \quad  \text{for some} \  \  n\ge 1  \, , 
\end{equation*}
and we let 
\[
\tilde{r} : \widetilde{R} \mapsto \tilde{r}({\widetilde{R}})
:=r_n   \, . 
\]

\noindent For ease of notation, without loss of generality, we assume that $u_1=\min_{1\le i\le d}\{u_i\}$ and we divide $\widetilde{\mathcal{R}}$ into two subfamilies:
\[
\Gamma_1:=\Big\{\widetilde{R}\in\widetilde{\mathcal{R}} : \  r\geq (\textstyle{1 \over 2}  \tilde{r} (\widetilde{R}))^{u_1}\Big\},
\]
and
\[
\Gamma_2:=\widetilde{\mathcal{R}}\setminus\Gamma_1.
\]
With \eqref{bnmk} in mind, to estimate the measure of $ B( \mathbf{x}, r)$, we consider  separately the contribution from  rectangles in  $\Gamma_1$ and from $\Gamma_2$.

\medskip

\noindent $\bullet$  First we consider the contribution from rectangles in $\Gamma_1$ to the measure of $ B( \mathbf{x}, r)$.  By the definition of $\Gamma_1$, the radius $r$ is sufficiently  large so that all that  rectangles in $\Gamma_1$ are contained in the scaled ball $$(2^{u_1+1}+1)B= B(\mathbf{x},(2^{u_1+1}+1)r)   \, . $$ Thus, by~\eqref{nu_widetilde_R},
\begin{align} \label{g7_1}
\nu\Big( B( \mathbf{x}, r)\cap \bigcup_{\widetilde{R}\in\Gamma_1}\widetilde{R}\Big)&\le \sum_{\widetilde{R}\in\Gamma_1}\nu(\widetilde{R})\asymp\sum_{\widetilde{R}\in\Gamma_1}\mu(\widetilde{R})\cdot \frac{\nu(B_{j-1})}{\mu(B_{j-1})}\nonumber \\[2ex]
&\le \mu\big(B(\mathbf{x},(2^{u_1+1}+1)r)\big)\cdot \frac{\nu(B_{j-1})}{\mu(B_{j-1})}.
\end{align}
Then, by using \eqref{Holder-B_j} and then~\eqref{r_leq_r_jm1}, we infer from~\eqref{g7_1} that
 \begin{align}\label{g7}
\nu\Big( B( \mathbf{x}, r)\cap \bigcup_{\widetilde{R}\in\Gamma_1}\widetilde{R}\Big)&\ll \mu( B( \mathbf{x}, r))\cdot \frac{r(B_{j-1})^{s_0(\bold{u},\bold{v})-2\varepsilon}}{\mu(B_{j-1})} \nonumber \\[2ex]
&\asymp r^{\delta_1+\cdots+\delta_p}\cdot\frac{r(B_{j-1})^{s_0(\bold{u},\bold{v})-2\varepsilon}}{r(B_{j-1})^{\delta_1+\cdots+\delta_p}} \nonumber \\[2ex] & \le r^{s_0(\bold{u},\bold{v})-2\varepsilon}.
\end{align}

\medskip

\noindent $\bullet$  To estimate the contribution from  rectangles in $\Gamma_2$, 
we  subdivide $\Gamma_2$  into finer  sub-classes. For every $k\in\N$  for which $\mathcal{I}_{G_{j-1}, B_{j-1}, k} \neq \emptyset$, we let
\[
\Gamma_{2,k}  :=\left\{ \widetilde{R}\in\Gamma_2 : R\in S_{G_{j-1},B_{j-1},k} \right\}   \, 
\]
and we proceed to estimate the quantity 
\begin{equation}  \label{hohoho}
    \nu\Big( B( \mathbf{x}, r)\cap \bigcup_{\widetilde{R}\in \Gamma_{2,k}}\widetilde{R}\Big)  \, .
\end{equation}



We start with some preliminaries.  Let $\widetilde{R}\in\Gamma_{2,k}$ for some $k\in\N$. Recall that $\widetilde{R}$ is of the form $\widetilde{R}=\prod_{i=1}^dB(x_{i,n}, 2^{-k \, u_i})$.   
Furthermore, by the definition of 
 $\Gamma_{2,k}$ we have
\begin{equation} \label{r_ll_u1}
r <\big({\textstyle{1 \over 2}} \tilde{r}  (\widetilde{R})\big)^{u_1} < 2^{-k\cdot u_1}.
\end{equation}
On the other hand, it follows from~\eqref{r_is_big} that 
\begin{equation} \label{r_gg_w}
r\geq  2^{-k\cdot w(j-1,k)}.
\end{equation}

With the quantities appearing in~\eqref{shorthand_vw} in mind, define
\[
\mathcal{A}_{j,k}  \ := \    \big\{u_i : 1\le i\le p\big\}     \  \cup  \  \big\{v_i(j-1,k): 1\le i\le p\big\}  \,  .  
\]
Note that since  $k\in\N$ with $\mathcal{I}_{G_{j-1}, B_{j-1}, k} \neq \emptyset$, both the above sets are non-empty. Also, note that by definition and the fact that  we can assume that  \eqref{v_in_bigger_than_u_i} is true for all $n \in \N$,  it follows that    $w(j-1,k)$ is the maximum of the finite set $\mathcal{A}_{j,k}$.  It is clear that there are $ 2 \le t\le 2p$ elements in $\mathcal{A}_{j,k}$. In what follows, we will denote by $A_{\ell}$    $(1 \le \ell \le t)$ the elements of $\mathcal{A}_{j,k}$ listed in ascending order. In other words,
\[
\mathcal{A}_{j,k}=\left\{ A_1,\dots,A_t \right\}
\]
and $A_{\ell}<A_{\ell+1}$ for all $1 \le \ell \le t$.
It follows from~\eqref{r_ll_u1} and~\eqref{r_gg_w} that there exists an index $\ell$ such that for two consecutive terms $A_\ell$ and  $A_{\ell+1}$ of $\mathcal{A}_{j,k}$, we have that
\begin{equation} \label{bhy}
2^{-k\cdot A_{\ell+1}} \, \le \  r \ < \,  2^{-k\cdot A_\ell}  \, .
\end{equation}
With this in mind, let 
\begin{equation*}
\K_1=\{i: u_i\ge A_{\ell+1}\}, \quad \ \K_2=\{i: v_i(j-1,k)\le A_\ell\},  \quad \ \K_3=\{1,\cdots,p\}\setminus (\K_1\cup \K_2).
\end{equation*}

\medskip 

 We are  now in the position  to estimate \eqref{hohoho}. We proceed in two steps.
\begin{itemize}\item[(i)] We estimate the cardinality of $\Gamma_{2,k}$. Note that if $i\not\in \K_1$, then $u_i<A_{\ell+1}$ and so $u_i\leq A_{\ell}$.
Therefore, we can define an ``enlarged'' body of the ball $ B( \mathbf{x}, r)$:
$$
H:=\prod_{i=1}^pB(x_i, (2^{u_i+1}+1)\epsilon_i)  \quad {\text{where}}\quad  \epsilon_i : =\left\{
\begin{array}{ll}
r & \hbox{if \ $i\in \K_1$;} \\[2ex]
2^{-k \cdot u_i}  & \hbox{otherwise.}
\end{array}
\right.
$$ Then, all the rectangles $\widetilde{R}\in \Gamma_{2,k}$ are contained in $H$. Since these rectangles $\widetilde{R}$ are disjoint, a straightforward   volume argument shows that  
\begin{equation}\label{ee1}\# \Gamma_{2,k}  \ \ll  \ \prod_{i\in \K_1}\left(\frac{r^{\delta_i}}{2^{-k \cdot u_i\delta_i}}\right). 
\end{equation}

\item[(ii)] For every $\widetilde{R}\in \Gamma_{2,k}$, we bound from above the total number $T$ of balls $B_{j}\in \mathcal{C}(R)$ intersecting the ball $ B( \mathbf{x}, r)$.
By~\eqref{r_is_big} and if necessary also see Proposition~\ref{pp2}, it follows that these balls $B_j$ are contained in $B(\mathbf{x},2r)\cap R$.  Also recall that by definition,  the  radius of any $B_{j}$  in $ \mathcal{C}(R)$ is  $2^{-k \cdot w(j-1, \, k)}$.   Thus a  straightforward  volume argument yields that
\begin{eqnarray*}
T\cdot 2^{-k\cdot w(j-1,k)\sum_{i=1}^p\delta_i} &\ll & \mu\Big({R}\cap B(x,2r)\Big)\\[2ex] 
& \leq & {} \prod_{i\in \K_1\cup\K_3}2^{-k\cdot v_i(j-1,k)\delta_i}\cdot \prod_{i\in \K_2} r^{\delta_i}.
\end{eqnarray*}

In turn, this together with (\ref{ee1}), enables us to  bound from above the total number $M$ of balls $B_{j}$ in the collection $\bigcup_{R\in  S_{G_{j-1}, B_{j-1}}}\mathcal{C}(R)$
intersecting $ B( \mathbf{x}, r)$:

\begin{align}\label{size-M}
M  \ \ll  \ &\prod_{i\in \K_1}\left(\frac{r^{\delta_i}}{2^{-k\cdot u_i\delta_i}}\right)\cdot T \nonumber\\[2ex]
 \ \ll  \ &\prod_{i\in \K_1}\left(\frac{r^{\delta_i}}{2^{-k\cdot u_i\delta_i}}\right)\cdot 2^{k\cdot w(j-1,k)\sum_{i=1}^p\delta_i}\cdot \prod_{i\in \K_1\cup\K_3}2^{-k\cdot v_i(j-1,k)\delta_i}\cdot \prod_{i\in \K_2} r^{\delta_i}.
\end{align}
\end{itemize}

\medskip

\noindent It follows from~\eqref{link_nu_j_B_j_nujm1_Bjm1} and  \eqref{stabilization_mu_j},  
that for every  $B_{j}\in\bigcup_{R\in  S_{G_{j-1}, B_{j-1}}}\mathcal{C}(R)$,
\[
	\nu(B_j)\asymp\frac{\mu(\widetilde{R})}{\#\mathcal{C}(R)}\cdot\frac{\nu(B_{j-1})}{\mu(B_{j-1})},
\]
where $R\in S_{G_{j-1},B_{j-1}}$ is such that $B_j\in \mathcal{C}(R)$. 
Hence, it follows that  
\[
\nu( B( \mathbf{x}, r)\cap \Gamma_{2,k}) \ \ll  \  M\cdot (\# \mathcal{C}(R))^{-1} \cdot \prod_{i=1}^p2^{-k\cdot u_i\delta_i}\cdot \frac{\nu(B_{j-1})}{\mu(B_{j-1})}.
\]
By \eqref{cardinality_of_CR}, and \eqref{size-M} we have
\begin{align}\label{f4}
\nu( B( \mathbf{x}, r)\cap \Gamma_{2,k}) 
& \ \ll  \prod_{i\in \K_1}r^{\delta_i}\cdot \prod_{i\in \K_2}r^{\delta_i}\cdot \prod_{i\in \K_2}\frac{2^{-k\cdot u_i\delta_i}}{2^{-k\cdot v_i(j-1,k)\delta_i}}\cdot\prod_{i\in \K_3}2^{-k\cdot u_i\delta_i}\cdot \frac{\nu(B_{j-1})}{\mu(B_{j-1})}\nonumber\\[2ex]
& \le {r^s}\cdot \frac{\nu(B_{j-1})}{\mu(B_{j-1})},
\end{align}
where the last inequality is true if $s$ satisfies
\begin{equation}\label{pplm}
s    \ \le  \   \sum_{i\in \K_1}\delta_i+\sum_{i\in \K_2}\delta_i+ \frac{1}{\log r } \left({\sum_{i\in \K_3}u_i\delta_i}-\sum_{i\in \K_2}(v_i(j-1,k)-u_i))\delta_i\right)\log 2^{-k}
\end{equation} 
for all $r$ in  the range    \eqref{bhy}; i.e.   
$
2^{-k\cdot A_{\ell+1}}\le r< 2^{-k\cdot A_\ell}   $ .

Note that the right hand side of \eqref{pplm} is monotonic in $r$. Thus,
 the optimal choice of $s$ is attained at either  $r=2^{-k\cdot A_{\ell+1}}$ or  $r=2^{-k \cdot A_\ell}$;  i.e. the minimum   of the following two numbers:
\begin{equation*}
  s_1:=\sum_{i\in \K_1}\delta_i+\sum_{i\in \K_2}\delta_i+\frac{1}{A_\ell}\Big(\sum_{i\in \K_3}u_i\delta_i-\sum_{i\in \K_2}(v_i(j-1,k)-u_i)\delta_i\Big),
\end{equation*}
and 
\begin{equation*}
s_2:=\sum_{i\in \K_1}\delta_i+\sum_{i\in \K_2}\delta_i+\frac{1}{A_{\ell+1}}\Big(\sum_{i\in \K_3}u_i\delta_i-\sum_{i\in \K_2}(v_i(j-1,k)-u_i)\delta_i\Big).
\end{equation*}
Now,  recall the definition of the dimensional number  $ {s}_0(\mathbf{u}, \mathbf{v})$ given  by~\eqref{dimensionalnumber} and   the quantity   ${s}(\mathbf{u}, \mathbf{v},A)$ given  by~\eqref{def_suv}.   Also, recall  that $A_{\ell}$ and $A_{\ell+1}$ are two different but neighboring numbers in $\mathcal{A}_{j,k}$. It follows that \begin{align*}
       {\K}_1(A_\ell)=\{i: u_i>A_\ell\}&=\{i: u_i\ge A_{\ell+1}\}{=\K_1},\ \quad \  {\K}_2(A_\ell)=\{i: v_i(j-1,k)\le A_{\ell}\}=\K_2.
      \end{align*}
  Hence, $s_1$ is nothing but the quantity  $s(\bold{u}, \bold{v}(j-1,k), A_\ell)$ as defined in~\eqref{def_suv},
 and $s_2$ is  equal to the quantity  $s(\bold{u}, \bold{v}(j-1,k), A_{\ell+1})$.
 The upshot of this  is that
\begin{eqnarray*}
 s    \ \ge \  \min\{s_1, s_2\}  \   \ge  \    s_0(\bold{u},\bold{v}(j-1,k))    \ \ge \   s_0(\bold{u},\bold{v})-\varepsilon,
  \end{eqnarray*} where 
we have used first Proposition~\ref{Prop:4} and then~\eqref{s0_is_continuous}.  On  feeding this into   (\ref{f4}), we obtain that 
\begin{align*}
\nu\Big( B( \mathbf{x}, r)\cap \bigcup_{\widetilde{R}\in \Gamma_{2,k}}\widetilde{R}\Big)\le {r^{s_0(\bold{u},\bold{v})-\varepsilon}}\cdot \frac{\nu(B_{j-1})}{\mu(B_{j-1})}.
 \end{align*}
It follows from \eqref{ratio_nu_mu_is_small} that
$$
 \frac{\nu(B_{j-1})}{\mu(B_{j-1})}\leq r_n^{-\varepsilon}=2^{-k\epsilon}.
$$
Hence,
\begin{eqnarray*}
\nu\Big( B( \mathbf{x}, r)\cap \bigcup_{\widetilde{R}\in\Gamma_2}\widetilde{R}\Big) &\le & 
\sum_{k }\nu\Big( B( \mathbf{x}, r)\cap \bigcup_{\widetilde{R}\in \Gamma_{2,k}}\widetilde{R}\Big) \\[2ex] 
& \le  &  \sum_{k\ge 1}r^{s_0(\bold{u},\bold{v})-\epsilon} \cdot 2^{-k \epsilon} \ \asymp  \  r^{s_0(\bold{u},\bold{v})-\epsilon}.
\end{eqnarray*}
This together with \eqref{bnmk}  and \eqref{g7} establishes~\eqref{ub_general_ball} for an arbitrary ball  and thereby completes the proof of Theorem~\ref{pointcase}.

\section{Extending Theorem~\ref{pointcase} to resonant sets \label{extendingour}}

In the setup of  Theorem~\ref{pointcase}, the rectangles arise as products of balls $B_{i,n}$ in $X_i$.  Balls of radius $r$ are of course $r$-neighbourhoods of special points; namely their centres.
The general `rectangles to rectangles'  Mass Transference Principle of  Wang $\&$  Wu \cite{WW2021} is based on the framework of ubiquitous systems. This  allows them to naturally consider the situation in which  the balls $B_{i,n}$  are replaced by neighbourhoods of special sets called  resonant sets.  It is not particularly  surprising that on combining the arguments in Wang-Wu \cite{WW2021} and the current paper,  Theorem~\ref{pointcase} can  be extended to case of resonant sets.   In this section we briefly  describe the setup  and state the resonant set version of Theorem~\ref{pointcase}.  The details are left to the energetic reader. 

For each $1 \le i \le p$, let
 $X_i$ be a locally compact subset of $\R^{d_i}$ equipped with a  $\delta_i$-Ahlfors regular measure $\mu_i$.
Let $J$ be an infinite countable index set and  $\beta: J\to\mathbb{R}^+$ be a positive function such that for any $M>1$, the set $\{\alpha\in J: \beta_\alpha<M\}$ is finite.  For $1\leq i\leq p$, let ${\cal R}_i:=\{R_{i,\alpha}: \alpha\in J\}$ be a sequence of subset of $X_i$. The  sets associated with  ${\cal R}_i$ are referred to as \emph{resonant sets}.   Denote by $\Delta(R_{i,\alpha},r)$ the $r$-neighborhood of $R_{i,\alpha}$ in $X_i$.  Naturally, in the case that the resonant sets are points $\{ x_{i,\alpha} : \alpha \in J \}$  as  within the framework of Theorem~\ref{pointcase},   the neighbourhood $\Delta(R_{i,\alpha},r)$ corresponds to a ball $B(x_{i,\alpha}, r)$.   In order to state the analogue of Theorem~\ref{pointcase} for resonant sets, we need to impose a ``scaling'' condition on the neighbourhood of resonant sets  and in turn  appropriately modify the dimension number.

\begin{itemize}
    \item As in \cite{WW2021}, we assume that for each $1 \le i \le p$, the sequence ${\cal R}_i$  of  subsets of $X_i$ satisfy the  $\kappa$-scaling property. Recall (see for instance \cite[Definition 3.1]{WW2021}), given ${\cal R}_i$  and    $0\leq \kappa<1$, we say that ${\cal R}_i$  satisfies the \emph{$\kappa$-scaling property} if there exist constants $c_1,c_2, r_0 >0 $  such that for any $\alpha\in J$, any ball $B(x_i,r)\subset X_i$  with center $x_i\in R_{i,\alpha}$ and any $0<\varepsilon<r < r_0$, we have
$$ c_1 \; r_i^{\delta_i\kappa}\varepsilon^{\delta_i(1-\kappa)}   \, \le \,  \mu_i(B(x_i, r_i)\cap\Delta(R_{i, \alpha},\varepsilon))  \, \le \, c_2  \; r_i^{\delta_i\kappa}\varepsilon^{\delta_i(1-\kappa)}   \; .  $$ 
Observe that in the case the resonant sets  are points,  ${\cal R}_i$  satisfies the $\kappa$-scaling property  with   $\kappa = 0$. 
\item  With the setup leading to \eqref{dimensionalnumber} in mind, the  following  modification of the dimension number $s_0(\mathbf{u},\mathbf{v})$   appearing in Theorem~\ref{pointcase}  reflects the role  of the above  $\kappa$-scaling property:
$$s_0^*(\mathbf{u},\mathbf{v}):=\min \big\{\min_{i\in \mathcal{L}(\mathbf{v})}\{{\theta}^*_i(\mathbf{u}, \mathbf{v})\}, \ \sum_{i\in\mathcal{L}(\mathbf{v})}\delta_i\big\},$$
where
$${\theta}^*_i(\mathbf{u}, \mathbf{v}):=
	\sum_{k\in \mathcal{K}_1(i)}\delta_k+\sum_{k\in \mathcal{K}_2(i)}\delta_k\left(1-\frac{{(1-\kappa)}(v_k-u_k)}{v_i}\right) +{(1-\kappa)}\sum_{k\in \mathcal{K}_3(i)}\frac{\delta_ku_k}{v_i}.$$

\end{itemize}

 \noindent  The following statement constitutes  the  unbounded version of  the  `rectangles to rectangles' Mass Transference Principle  for resonant sets.

\begin{theorem}\label{dimension} 
Under the setting above,  for  each  $1\leq i \leq p$, let $\psi_i: \mathbb{R}^+\to\mathbb{R}^+$ be a real positive function such that  $\psi_i(t)\to 0$ as $t\to\infty$.  
Furthermore,  assume that there exist functions $\rho: \mathbb{R}^+\to\mathbb{R}^+$ and  $v_i: \mathbb{R}^+\to\mathbb{R}^+$ ($1\leq i \leq p$) such that $$\lim\limits_{t\to\infty}\rho(t) =  0   \, , $$ 
\[
\psi_i(\beta_\alpha)=\rho(\beta_\alpha)^{v_i(\beta_\alpha)}
\] 
and $$\lim\limits_{t\to\infty}v_i(t)=v_i,$$
for some ${\mathbf{v}}=(v_1,\dots, v_p)\in(\mathbb{R}^+ \cup\{+\infty\} )^p$.
Finally, suppose that there exists $\mathbf{u}=(u_1,\dots, u_p)\in (\mathbb{R}^+)^p $ verifying~\eqref{111} and 
		\begin{equation}\label{h4sv}
			\mu_1\times\cdots\times\mu_p\left(\limsup_{\alpha\in J: \beta_{\alpha}\to \infty}\prod_{i=1}^p \Delta\Big(R_{\alpha, i}, \rho(\beta_{\alpha})^{u_i}\Big)\right)=\mu_1\times\cdots\times\mu_p\left(\prod_{i=1}^pX_i\right). \end{equation}
		Then $$
		\dim_{\rm H} \left(\limsup_{\alpha\in J: \beta_{\alpha}\to \infty}\prod_{i=1}^p \Delta\Big(R_{\alpha, i}, \psi_i(\beta_{\alpha})\Big)\right)\ge s_0^*(\mathbf{u},\mathbf{v}).
		$$
	\end{theorem}

 \medskip

\begin{remark}
    In the case the resonant sets are points and with Theorem~\ref{pointcase} in mind,    the sequence $\{\rho(\beta_\alpha)\}_{\alpha\in J}$ corresponds to $\{r_n\}_{n\in\mathbb{N}}$ and $\{v_i(\beta_\alpha)\}_{\alpha\in J}$ corresponds to the sequence $\{v_{i,n}\}_{n\in\mathbb{N}}$.    Moreover,   by definition, it follows that 
    $$  \Delta\Big(R_{i, \alpha}, \rho(\beta_{\alpha})^{u_i}\Big) = B(x_{i,n},r_n^{u_i})  :=
    B_{i,n}^{s_{i,n}}   \, $$ 
    where $s_{i,n} $ is given by \eqref{sin}.  Thus,   the full measure statements  \eqref{fullmeasure1}	and \eqref{h4sv} coincide. 
\end{remark}

\section*{Acknowledgement}

B.L.~was partially supported by NSFC (No. 12271176) and Guangdong Natural Science Foundation 2024A1515010946. 
L.L. was partially supported by the Fundamental Research Funds for the Central Universities of China (1301/600460054). B.W. was partially supported by NSFC (No. 12331005).

\vspace*{6ex}




 { }


\vspace*{10ex}

\noindent Bing Li: Department of Mathematics,
South China University of Technology,

\vspace{-2mm}

\noindent\phantom{Bing Li: }Wushan Road 381, Tianhe District, Guangzhou, China


\noindent\phantom{Bing Li: }e-mail: scbingli@scut.edu.cn


\vspace{5mm}

\noindent Lingmin Liao: School of Mathematics and Statistics,
Wuhan University, 
\vspace{-2mm}

\noindent\phantom{Lingmin Liao: }Bayi Road 299, Wuchang District, Wuhan, China

\noindent\phantom{Lingmin Liao: }e-mail: lmliao@whu.edu.cn 


\vspace{5mm}

\noindent Sanju Velani: Department of Mathematics,
University of York,

\vspace{-2mm}

\noindent\phantom{Sanju Velani: }Heslington, York, YO10
5DD, England.


\noindent\phantom{Sanju Velani: }e-mail: sanju.velani@york.ac.uk


\vspace{5mm}

\noindent Baowei Wang:  School of Mathematics, Huazhong University of Science and Technology,

\vspace{-2mm}

\noindent\phantom{Baowei Wang: } Wuhan
430074, China.


\noindent\phantom{Baowei Wang: } e-mail: bwei\_wang@hust.edu.cn

\vspace{5mm}

\noindent Evgeniy Zorin: Department of Mathematics,
University of York,

\vspace{-2mm}

\noindent\phantom{Evgeniy Zorin: }Heslington, York, YO10
5DD, England.


\noindent\phantom{Evgeniy Zorin: }e-mail: evgeniy.zorin@york.ac.uk

\end{document}